\pdfoutput=1

\documentclass{amsart}

\usepackage[mathscr]{eucal}
\usepackage{amssymb}
\usepackage[usenames,dvipsnames]{xcolor} 
\usepackage[normalem]{ulem}
\usepackage{wasysym}
\usepackage{amsthm}
\usepackage{bbold}
\usepackage{comment}
\usepackage{enumitem}
\usepackage{amsmath}
\usepackage{tikz-cd}
\usepackage{stmaryrd} 
\usepackage{etoolbox} 
\usepackage{microtype}
\usepackage{adjustbox}

\usepackage[unicode]{hyperref} 

\definecolor{dark-red}{rgb}{0.5,0.15,0.15}
\definecolor{dark-blue}{rgb}{0.15,0.15,0.6}
\definecolor{dark-green}{rgb}{0.15,0.6,0.15}

\hypersetup{
    colorlinks, linkcolor=OliveGreen,
    citecolor=OliveGreen, urlcolor=OliveGreen
}

\usepackage[nameinlink,capitalise,noabbrev]{cleveref}

\numberwithin{equation}{section}
\usepackage[all]{xy}
\xyoption{line}
\usepackage{graphicx}
\usepackage{mathtools}
\usepackage{quiver}

\usepackage{caption}

\newtheorem{thmx}{Theorem}

\newtheorem{Thm}[equation]{Theorem}
\newtheorem*{Thm*}{Theorem}
\newtheorem*{MainThm*}{Main Theorem}
\newtheorem{Prop}[equation]{Proposition}
\newtheorem{Lem}[equation]{Lemma}
\newtheorem{Cor}[equation]{Corollary}

\newtheorem*{Que*}{Question}

\theoremstyle{remark}
\newtheorem{Def}[equation]{Definition}
\newtheorem{Ter}[equation]{Terminology}
\newtheorem{Not}[equation]{Notation}
\newtheorem{Exa}[equation]{Example}

\newtheorem{Cons}[equation]{Construction}

\newtheorem{Rec}[equation]{Recollection}
\newtheorem{War}[equation]{Warning}

\newtheorem{Rem}[equation]{Remark}

\tikzset{
    labelrotatebelow/.style={anchor=north, rotate=90, inner sep=1.0mm}
}
\tikzset{
    labelrotateabove/.style={anchor=south, rotate=90, inner sep=1.0mm}
}
\SelectTips{cm}{10}

\newcommand{\nc}{\newcommand}
\nc{\dmo}{\DeclareMathOperator}

\renewcommand{\emptyset}{\varnothing}

\nc{\Drew}[1]{{\color{Orange}#1}}
\nc{\Tobi}[1]{{\color{Green}#1}}
\nc{\Natalia}[1]{{\color{Yellow}#1}}
\nc{\Dout}[1]{\Drew{\sout{#1}}}
\nc{\Tout}[1]{\Tobi{\sout{#1}}}
\nc{\Nout}[1]{\Natalia{\sout{#1}}}

\usepackage{todonotes}

\nc{\overbar}[1]{\mkern 1.5mu\overline{\mkern-1.5mu#1\mkern-1.5mu}\mkern 1.5mu}

\nc{\kappaaux}{g}
\nc{\kappaCh}{{\kappaaux(\cat C_h)}}
\nc{\kappam}{{\kappaaux({\mathfrak m})}}
\nc{\kappaP}{\Gamma_{\cat P}\unit}
\nc{\kappaQ}{{\kappaaux(\cat Q)}}
\nc{\kappaCP}{{\kappaaux_{\cat C}(\cat P)}}
\nc{\kappaDP}{{\kappaaux_{\cat D}(\cat P)}}
\nc{\kappaCQ}{{\kappaaux_{\cat C}(\cat Q)}}
\nc{\kappaDQ}{{\kappaaux_{\cat D}(\cat Q)}}
\nc{\kappaphiB}{{\kappaaux(\phi(\cat B))}}
\nc{\kappaphiQ}{{\kappaaux(\varphi(\cat Q))}}
\newcommand{\noloc}{\;\mathord{:}\,}
\dmo{\Sub}{Sub}
\dmo{\Proj}{Proj}
\dmo{\LMod}{LMod}
\dmo{\cell}{cell}
\nc{\SpEn}{\cat S_{E(n)}}
\nc{\SpEnf}{\cat S_n}
\nc{\Lcomp}{L^{\mathrm{com}}} 
\nc{\Ucomp}{U^{\mathrm{com}}}
\nc{\Loco}[1]{\Loc_{\otimes}\hspace{-0.3ex}\langle #1 \rangle}
\nc{\bbullet}{{\scriptscriptstyle\hspace{-1pt}\bullet}}
\nc{\bullett}{{\scriptscriptstyle\bullet}\hspace{-1pt}}
\nc{\LF}{L\hspace{-0.2ex}F}
\nc{\SpG}{\Sp^G}
\nc{\EG}{\bbE_G}
\nc{\DEG}{\Der(\EG)}
\nc{\DE}{\Der(\bbE)}
\nc{\Prst}{{\cat P}\mathrm{r^{st}}}
\nc{\Mack}[2]{\mathrm{Mack}_{#1}(#2)}
\nc{\SC}{S\cat C}
\dmo{\fin}{{fin}}
\dmo{\DM}{DM}
\dmo{\fp}{fp}
\nc{\DMQ}{\DM_Q}
\dmo{\DerKal}{DMack}
\dmo{\Der}{D}
\dmo{\DMot}{DMot}
\dmo{\rmH}{H}
\dmo{\piu}{\underline{\pi}}
\dmo{\Sphere}{\mathbb{S}}
\dmo{\Alg}{Alg}
\dmo{\CAlg}{CAlg}
\nc{\HA}{{\rmH \hspace{-0.2em}\bbA}}
\nc{\HZ}{{\rmH \hspace{-0.2em}\bbZ}}
\nc{\HZbar}{{\rmH \hspace{-0.2em}\underline{\bbZ}}}
\nc{\Fp}{{\bbF_{\hspace{-0.1em}p}}}
\nc{\HFp}{{\rmH \hspace{-0.15em}\bbF_{\hspace{-0.1em}p}}}
\nc{\DHZG}{\Der(\HZ_G)}
\nc{\DHZH}{\Der(\HZ_H)}
\nc{\DHZK}{\Der(\HZ_K)}
\nc{\DHZGN}{\Der(\HZ_{G/N})}
\nc{\DHZGG}{\Der(\HZ_{G/G})}
\nc{\DHZCp}{\Der(\HZ_{C_p})}
\nc{\DHZGprime}{\Der(\HZ_{G'})}
\nc{\DHZ}{\Der(\HZ)}
\nc{\mathfrakp}{\mathfrak{p}}
\nc{\mathfrakq}{\mathfrak{q}}
\nc{\mathfrakS}{\mathfrak{S}}
\nc{\mathfrakT}{\mathfrak{T}}
\nc{\Z}{\mathbb{Z}}
\nc{\SSG}{\text{sSet}_*^G}
\nc{\sSet}{\text{sSet}}

\dmo{\csupp}{csupp}
\dmo{\Con}{Conj}
\dmo{\Id}{Id}
\dmo{\Loc}{Loc}
\dmo{\rmK}{\textrm{\rm K}}
\dmo{\Spc}{Spc}
\dmo{\thick}{thick}
\nc{\thickt}[1]{\thick_\otimes\langle #1 \rangle}
\dmo{\cone}{cone}
\dmo{\End}{End}
\dmo{\Derperf}{D_{perf}}
\dmo{\Mor}{Mor}
\dmo{\Hom}{Hom}
\dmo{\id}{id}
\dmo{\incl}{incl}
\dmo{\Img}{Im}
\dmo{\im}{im}
\dmo{\Ker}{Ker}
\dmo{\ind}{ind}
\dmo{\Ind}{Ind}
\dmo{\CoInd}{coind}
\dmo{\res}{res}
\dmo{\infl}{infl}
\dmo{\Derqc}{D_{qc}}
\dmo{\triv}{triv}
\dmo{\Tel}{Tel} 
\dmo{\grMod}{grMod}%
\dmo{\Mod}{Mod}%
\dmo{\opname}{op}
\dmo{\SH}{SH}
\dmo{\smallb}{b}
\dmo{\Spec}{Spec}
\dmo{\supp}{supp}
\dmo{\Supp}{Supp}
\dmo{\cosupp}{cosupp}
\dmo{\Cosupp}{Cosupp}
\nc{\SHc}{{\SH^c}}
\nc{\SHp}{{\SH_{(p)}}}
\nc{\SHcp}{{\SH^c_{(p)}}}
\nc{\SHG}{\SH(G)}
\nc{\SHGp}{\SH(G)_{(p)}}
\nc{\SHGc}{\SHG^c}
\nc{\SHGcp}{\SHG^c_{(p)}}
\nc{\quadtext}[1]{\quad\textrm{#1}\quad}
\nc{\qquadtext}[1]{\qquad\textrm{#1}\qquad}
\nc{\adj}{\dashv}
\nc{\adjto}{\rightleftarrows}
\nc{\bbL}{\mathbb{L}}
\nc{\bbA}{\mathbb{A}}
\nc{\bbE}{\mathbb{E}}
\nc{\bbN}{\mathbb{N}}
\nc{\bbQ}{\mathbb{Q}}
\nc{\bbZ}{\mathbb{Z}}
\nc{\bbF}{\mathbb{F}}
\nc{\cat}[1]{\mathscr{#1}}
\nc{\ie}{{\sl i.e.}, }
\nc{\into}{\mathop{\rightarrowtail}}
\nc{\inv}{^{-1}}
\nc{\isoto}{\mathop{\overset{\sim}\to}}
\nc{\isotoo}{\mathop{\overset{\sim}\too}}
\nc{\onto}{\mathop{\twoheadrightarrow}}
\nc{\too}{\mathop{\longrightarrow}\limits}
\nc{\mapstoo}{\longmapsto}
\nc{\adh}[1]{\overline{#1}}
\nc{\adhpt}[1]{\adh{\{#1\}}}
\nc{\aka}{{a.\,k.\,a.}\ }
\nc{\calF}{\mathcal{F}}
\nc{\eg}{{\sl e.\,g.}}
\nc{\Homcat}[1]{\Hom_{\cat #1}}
\nc{\hook}{\hookrightarrow}
\nc{\ideal}[1]{\langle #1\rangle}
\nc{\ihom}{{\underline{\hom}}}
\nc{\iHom}{\mathcal{H}\mathrm{om}}
\nc{\Mid}{\,\big|\,}
\nc{\MMod}{\,\text{-}\Mod}%
\nc{\GrMMod}{\,\text{-}\grMod}%
\nc{\op}{^{\opname}}
\nc{\oto}[1]{\overset{#1}\to}
\nc{\otoo}[1]{\overset{#1}{\,\too\,}}
\nc{\sminus}{\!\smallsetminus\!}
\nc{\poplus}[1]{^{\oplus #1}}%
\nc{\potimes}[1]{^{\otimes #1}}
\nc{\sbull}{{\scriptscriptstyle\bullet}}
\nc{\SET}[2]{\big\{\,#1\Mid#2\,\big\}}
\nc{\SpcK}{\Spc(\cat K)}
\nc{\then}{\Rightarrow}
\nc{\unit}{\mathbb{1}}
\nc{\xra}{\xrightarrow}
\nc{\phigeom}[1]{\widetilde{\Phi}^{#1}}
\dmo{\Oname}{O}
\dmo{\proper}{proper}
\dmo{\lenormal}{\unlhd}
\dmo{\lnormal}{\lhd}
\nc{\normal}{\trianglelefteq}
\nc{\Op}{\Oname^p}
\nc{\Oq}{\Oname^q}
\dmo{\Sp}{Sp}
\dmo{\Ho}{Ho}
\dmo{\Fin}{Fin}
\dmo{\add}{add}
\dmo{\Fun}{Fun}
\dmo{\Ext}{Ext}

\dmo{\CMon}{CMon}
\dmo{\CC}{\cat C}
\dmo{\DD}{\cat D}
\dmo{\OO}{\mathcal{O}}
\dmo{\Map}{Map}
\dmo{\Span}{Span}
\dmo{\N}{N}
\dmo{\Cat}{Cat}
\dmo{\colim}{colim}
\dmo{\hocolim}{hocolim}
\dmo{\Ch}{Ch}
\dmo{\A}{\mathbb{A}^{eff}}
\nc{\AGeff}{\mathbb{A}_G^{\mathrm{eff}}}
\nc{\BGeff}{\mathcal{B}_G^{\mathrm{eff}}}
\nc{\BG}{{\mathcal{B}_G}}
\nc{\NBGeff}{{\N}{\BGeff}}
\dmo{\Ab}{Ab}
\dmo{\Set}{Set}
\dmo{\ev}{ev}
\dmo{\Spcl}{Spcl}
\nc{\Funadd}{\Fun_{\add}}
\dmo{\proj}{proj}
\dmo{\cof}{cof}

\dmo{\Coideal}{Coideal}
\dmo{\gen}{gen}
\dmo{\Coloc}{Coloc}
\nc{\Coloco}[1]{\Coloc^{\iHom}\hspace{-0.3ex}\langle #1 \rangle}

\dmo{\dual}{dual}

\nc{\LOCO}{\mathcal{L}\mathrm{oc}_{\otimes}}
\nc{\COLOCO}{\mathcal{C}\mathrm{oloc}^{\iHom}}

\nc{\Perff}[1]{\mathrm{Perf}_{#1}}
\nc{\Modd}[1]{\mathrm{Mod}_{#1}}
\dmo{\Perf}{Perf}
\dmo{\tel}{tel}
\nc{\Lococat}[2]{\Loc^{#1}_{\otimes}\hspace{-0.3ex}\langle #2 \rangle}
\dmo{\rk}{rk}
\dmo{\StMod}{StMod}

\nc{\QW}{W^{\text{Q}}}

\nc{\mT}{\kern-0.5em\mod\kern-0.1em\text{-}\cat{T}^c}
\nc{\mTc}{\kern-0.5em\mod\kern-0.1em\text{-}\cat{T}^c}
\nc{\MTc}{\Mod\kern-0.1em\text{-}\cat{T}^c}
\nc{\MT}{\Mod\kern-0.1em\text{-}\cat{T}}
\newcounter{enum-resume-hack}
\Crefname{Thm}{Theorem}{Theorems}
\Crefname{Prop}{Proposition}{Propositions}
\Crefname{Lem}{Lemma}{Lemmas}
\Crefname{thmx}{Theorem}{Theorems}

\begin{document}

\title[Quillen stratification in equivariant homotopy theory]{Quillen stratification in \\ equivariant homotopy theory}

\author[Barthel]{Tobias Barthel}
\author[Castellana]{Nat{\`a}lia Castellana}
\author[Heard]{Drew Heard}
\author[Naumann]{Niko Naumann}
\author[Pol]{Luca Pol}
\date{\today}

\makeatletter
\patchcmd{\@setaddresses}{\indent}{\noindent}{}{}
\patchcmd{\@setaddresses}{\indent}{\noindent}{}{}
\patchcmd{\@setaddresses}{\indent}{\noindent}{}{}
\patchcmd{\@setaddresses}{\indent}{\noindent}{}{}
\makeatother

\address{Tobias Barthel, Max Planck Institute for Mathematics, Vivatsgasse 7, 53111 Bonn, Germany}
\email{tbarthel@mpim-bonn.mpg.de}
\urladdr{https://sites.google.com/view/tobiasbarthel/}

\address{Nat{\`a}lia Castellana, Departament de Matem\`atiques, Universitat Aut\`onoma de Barcelona, 08193 Bellaterra, Spain, and Centre de Recerca Matemàtica}
\email{Natalia.Castellana@uab.cat}
\urladdr{http://mat.uab.cat/$\sim$natalia}

\address{Drew Heard, Department of Mathematical Sciences, Norwegian University of Science and Technology, Trondheim}
\email{drew.k.heard@ntnu.no}
\urladdr{https://folk.ntnu.no/drewkh/}

\address{Niko Naumann, Fakult{\"a}t f{\"u}r Mathematik, Universit{\"a}t Regensburg, Universit{\"a}tsstraße 31, 93053 Regensburg, Germany}
\email{Niko.Naumann@mathematik.uni-regensburg.de}
\urladdr{https://homepages.uni-regensburg.de/$\sim$nan25776/}

\address{Luca Pol, Fakult{\"a}t f{\"u}r Mathematik, Universit{\"a}t Regensburg, Universit{\"a}tsstraße 31, 93053 Regensburg, Germany}
\email{luca.pol@mathematik.uni-regensburg.de}
\urladdr{https://sites.google.com/view/lucapol/}

\begin{abstract}
We prove a version of Quillen's stratification theorem in equivariant homotopy theory for a finite group $G$, generalizing the classical theorem in two directions. Firstly, we work with arbitrary commutative equivariant ring spectra as coefficients, and secondly, we categorify it to a result about equivariant modules. Our general stratification theorem is formulated in the language of equivariant tensor-triangular geometry, which we show to be tightly controlled by the non-equivariant tensor-triangular geometry of the geometric fixed points. 

We then apply our methods to the case of Borel-equivariant Lubin--Tate $E$-theory $\underline{E_n}$, for any finite height $n$ and any finite group $G$, where we obtain a sharper theorem in the form of cohomological stratification. In particular, this provides a computation of the Balmer spectrum as well as a cohomological parametrization of all localizing $\otimes$-ideals of the category of equivariant modules over $\underline{E_n}$, thereby establishing a finite height analogue of the work of Benson, Iyengar, and Krause in modular representation theory.
\vspace{-2em}
\end{abstract}

\subjclass[2020]{18F99, 55P42, 55P91, 55U35}

\maketitle

\begin{figure}[h]
    \centering{}
    \includegraphics[scale=0.40]{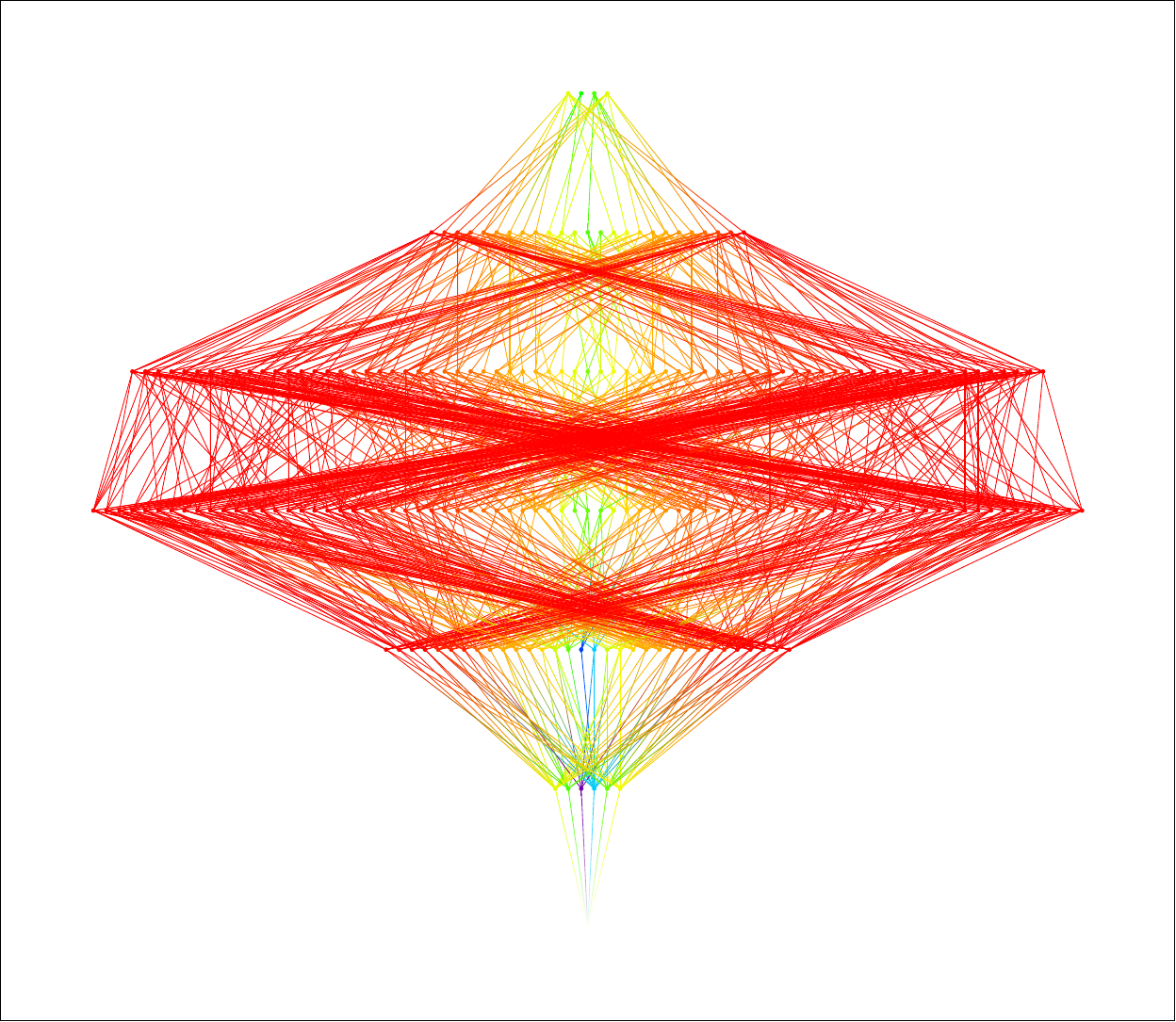}
\end{figure}
\vspace{-2em}

\setcounter{tocdepth}{1}
\tableofcontents
\vspace{-2em}
{\let\thefootnote\relax\footnotetext{The image is a visualization of Quillen's category for the symmetric group $S_{12}$ at the prime 2, made by Jared Warner. We are grateful to him for allowing us to include it here.}}

\newpage
\section{Introduction}\label{sec:introduction}

Quillen's celebrated stratification theorem \cite{Quillen71} provides a geometric description of the cohomology of any finite group $G$ with coefficients in a field $k$ in terms of a decomposition of its Zariski spectrum into locally closed subsets:
\begin{equation}\label{eq:classicalquillen}\tag{$\ast$}
\Spec H^{\bullet}(G,k) = \bigsqcup_{(E) \subseteq G} \mathcal{V}_{G,E}^+.
\end{equation}
Here, the disjoint union is indexed on the conjugacy classes of elementary abelian subgroups $E \subseteq G$ and the strata $\mathcal{V}_{G,E}^+$ are orbits of the Weyl group action on an open subset of the Zariski spectrum of the cohomology of $E$. Since the cohomology of elementary abelian subgroups is well-known, Quillen's work gives a formula to understand the cohomology ring of any finite group geometrically.

The first goal of the present paper is to develop and prove a far-reaching generalization of Quillen's theorem, in the following two directions:
    \begin{enumerate}
        \item We establish an analogue of Quillen stratification for an arbitrary commutative equivariant ring spectrum\footnote{Here and throughout, by commutative equivariant ring spectrum we mean an $\mathbb{E}_{\infty}$-monoid in $G$-spectra, without any additional norm structures.} $R$. Specializing to the Borel-equivariant theory $R = \underline{k}$ recovers a version of Quillen's theorem.
        \item We categorify the stratification of the group cohomology to a decomposition of the entire category of modules over $R$, for appropriate equivariant ring spectra $R$. The special case of $R = \underline{k}$ is contained in the seminal work of Benson, Iyengar, and Krause \cite{BensonIyengarKrause11a} on the stable module category of $k$-linear $G$-representations.
    \end{enumerate}
Our results take place in the context of equivariant homotopy theory and are formulated in the language of tensor-triangular geometry. This `equivariant tensor-triangular geometry' was initiated by Strickland (unpublished), Balmer \cite{Balmer16b}, and Balmer and Sanders in \cite{BalmerSanders17} and pursued further in \cite{BHNNNS19}, \cite{bgh_balmer}, \cite{PatchkoriaSandersWimmer20pp}, and \cite{bhs1}. The present paper conceptualizes these developments in the form of a unifying perspective encompassing equivariant tensor-triangular geometry, Quillen's stratification of group cohomology, and stratifications in modular representation theory. In particular, we establish a Quillen-type decomposition of the spectrum of any equivariant tensor-triangulated category and study the extent to which it is reflected in a stratification of the category defined over it; precise definitions are given in the next section. Our first main theorem is:

\begin{Thm*}[Informal version] The category of equivariant modules over any commutative $G$-equivariant ring spectrum $R$ is stratified in terms of the geometric fixed points $\Phi^HR$ equipped with their Weyl group actions for all subgroups $H$ of $G$.
\end{Thm*}

Of particular importance to us is the special case of Borel-equivariant Lubin--Tate $E$-theory. Their non-equivariant counterparts constitute outstanding examples of commutative rings of mixed chromatic characteristic (\cite{HopkinsSmith98}) and play a pivotal role in chromatic stable homotopy theory (see for example \cite{Morava1985,GHMR2005}). While they are known to afford a generalized character theory in the sense of Hopkins, Kuhn, and Ravenel \cite{HopkinsKuhnRavenel2000Generalized}, a classification of isomorphism classes of $E$-linear representation theory of finite groups seems out of reach. In light of this, our second main theorem provides a `coarse' parametrization of all (permutation) representations of a finite group $G$ over any Lubin--Tate $E$-theory in terms of the Zariski spectrum of the $E$-cohomology of $G$:

\begin{Thm*}[Informal version]
The category of $G$-equivariant modules over Borel-equivariant Lubin--Tate $E$-theory is stratified over the spectrum $\Spec(E^0(BG))$.
\end{Thm*}

In other words, we obtain a chromatic analogue in all mixed intermediate characteristics of the aforementioned work of Benson, Iyengar, and Krause. This completes the stratification program for the higher representation theory of finite groups, i.e., representations in all characteristics over the sphere spectrum. 
\vskip 0.1in

\begin{center}
\begin{tabular}{ |p{2cm}||p{3cm}|p{3cm}|p{2.7cm}|  }
 \hline
 \multicolumn{4}{|c|}{Overview of cohomological stratification for representations of finite groups} \\
 \hline
 Coefficients & Chromatic height & Spectrum & Reference \\
 \hline
 $\bbQ$   & $0$    &  $\Spec(\bbQ)$  &  \cite{Maschke1899} \\
 $\Fp$ &  $(p,\infty)$  & $\Spec^h(H^*(G,\Fp))$  & \cite{BensonCarlsonRickard97,BensonIyengarKrause11a} \\
 $\bbZ_{p}$ & $0, (p,\infty)$ & $\Spec^h(H^*(G,\bbZ_{p}))$ &  \cite{Lau2021Balmer, Barthel2021pre} \\
 $E_n$    & $0,(p,1),\ldots, (p,n)$ & $\Spec(E_n^0(BG))$ &  [\cref{thm:etheorycohomstratification}] \\
 \hline 
\end{tabular} 
\end{center}
\vskip 0.1in

\subsection*{Main results}

Based on previous ideas by Balmer--Favi \cite{BalmerFavi11}, Benson--Iyengar--Krause \cite{BensonIyengarKrause11b}, as well as Stevenson \cite{Stevenson13} and extending their works, a subset of the present authors together with Sanders have developed a general framework of stratification in tensor-triangular geometry. This theory utilizes a universal support function to parametrize localizing ideals of a suitable compactly generated tensor-triangulated category $\cat T$ in terms of subsets of the Balmer spectrum of $\cat T^c$:
\[
\begin{tikzcd}[ampersand replacement=\&]
	\Supp\colon {\begin{Bmatrix}
	\text{Localizing $\otimes$-ideals} \\ 
	\text{of } \cat T
	\end{Bmatrix}} 
	\& 
	{\begin{Bmatrix}
	\text{Subsets} \\
	\text{of } \Spc(\cat T^c) 
	\end{Bmatrix}.}
	\arrow["", from=1-1, to=1-2]
\end{tikzcd}
\]
We then call the tt-category $\cat T$ {\em stratified} if $\Supp$ is a bijection, terminology that originates in \cite{BensonIyengarKrause11a} and is motivated by its topological counterpart for $\Spc(\cat T^c)$. This separates the problem of classifying localizing $\otimes$-ideals into two questions:
    \begin{itemize}
        \item Is $\cat T$ stratified ?
        \item What is $\Spc(\cat T^c)$ as a \emph{set}?
    \end{itemize}
In \cite{bhs1}, we provide various techniques for approaching both questions; in practice, this often turns out to be easier than a full computation of $\Spc(\cat T^c)$ which is equivalent to a parametrization of all thick $\otimes$-ideals of compact objects $\cat T^c$ in $\cat T$. The general paradigm goes as follows: First construct a suitable family of jointly conservative tt-functors 
\[
\begin{tikzcd}[ampersand replacement=\&]
{f_i\colon \cat T}
\&
{\cat S_i}
\arrow["", from=1-1, to=1-2]
\end{tikzcd}
\]
with all $\cat S_i$ stratified. Secondly, show that stratification descends along the family of functors $f_i$, using the above tt-methods established in \cite{bhs1} and \cite{Barthel2021pre}. Assuming for simplicity that the family is finite, it follows that there is a continuous surjection
\[
\begin{tikzcd}[ampersand replacement=\&]
{\bigsqcup_i \Spc(\cat S_i^c)}
\&
{\Spc(\cat T^c),}
\arrow["", from=1-1, to=1-2]
\end{tikzcd}
\]
see \cref{cor:surjective-map}. The third step relies on the geometry of the functors $f_i$ to determine the spectrum of $\cat T^c$ in terms of $\Spc(\cat S_i^c)$, potentially using additional structure available in the given context. 

Our first aim is to carry out this strategy in the context of stable equivariant homotopy theory. Consider a finite group $G$, fixed throughout, and let $\Sp_G$ be the category of genuine $G$-equivariant spectra. Let $R$ be a commutative  $G$-equivariant ring spectrum, i.e., a commutative algebra in the symmetric monoidal category $\Sp_G$, and let $\Modd{G}(R)$ be the tt-category of $G$-equivariant modules over $R$. Write $\Perff{G}(R)$ for the full subcategory of compact objects in $\Modd{G}(R)$. Equivariant homotopy theory supplies excellent candidates for the family $f_i$, namely the geometric fixed points
\[
\begin{tikzcd}[ampersand replacement=\&]
{\Phi^H\colon \Modd{G}(R)}
\&
{\Mod(\Phi^HR)}
\arrow["", from=1-1, to=1-2]
\end{tikzcd}
\]
for $H$ varying over the subgroups of $G$. There is a natural action of the Weyl group $W_G(H) = N_G(H)/H$ on $\Mod(\Phi^HR)$. With this preparation, the next result is then a precise statement of our first main theorem written above:

\begin{thmx}[\cref{thm:eqstratdescent,thm:quillen_decomposition}]\label{thmx:abstract}
Let $R$ be a commutative equivariant ring spectrum and write $\Modd{G}(R)$ for the category of $G$-equivariant modules over $R$. Then:
    \begin{enumerate}
        \item The spectrum of perfect $R$-modules admits a locally-closed decomposition
            \begin{equation}\label{eq:ttquillen}\tag{$\dagger$}
            \Spc(\Perff{G}(R)) \simeq \bigsqcup_{(H) \subseteq G} \Spc(\Perf(\Phi^HR))/W_G(H),
            \end{equation}
            with the set-theoretic disjoint union being indexed on conjugacy classes of subgroups of $G$;
        \item $\Modd{G}(R)$ is stratified if the categories $\Mod(\Phi^HR)$ are stratified with Noetherian spectrum for all subgroups $H$ in $G$.
    \end{enumerate}
In both (a) and (b) it suffices to index on a family $\mathcal{F}$ of subgroups $H \subseteq G$ such that $R$ is $\mathcal{F}$-nilpotent.
\end{thmx}

To orient intuition and reconnect to Quillen's work on group cohomology, recall that any field $k$ may be viewed as a commutative equivariant ring spectrum $\underline{k}$ by passing to the Borel-completion of the associated Eilenberg--MacLane ring spectrum. The corresponding tt-category of compact equivariant modules identifies\footnote{This identification requires a result on generation by permutation modules due to Rouquier, Mathew, and Balmer--Gallauer, see \cref{rem:genpermmodules} for further details.} with the bounded derived category of finitely generated $kG$-modules:
\[
\Perff{G}(\underline{k}) \cong \mathrm{D}^b(\mathrm{mod}(kG)).
\]
Coupled with the computation of the spectrum in \cite{BensonCarlsonRickard97}, \cref{thmx:abstract}, $(a)$ applied to $R = \underline{k}$ then essentially recovers Quillen's stratification theorem\footnote{for the Zariski spectrum of homogeneous prime ideals as opposed to all prime ideals} from the beginning. For a more thorough discussion of the comparison, see \cref{ssec:ttmodpquillen}. Part $(b)$ reduces equivariant stratification to a non-equivariant question, one for each subgroup of $G$. In the case of $R = \underline{k}$, this reduction does not constitute a significant simplification, but it's surprisingly powerful in many other instances, as we will demonstrate below. 

Specializing in a different direction, \cref{thmx:abstract} is also interesting for equivariant ring spectra $R=\infl(R^e)$ with \emph{trivial} $G$-action, i.e., those that are inflated from a non-equivariant commutative ring spectrum $R^e$. The initial example is the equivariant sphere spectrum $S_G^0$, the monoidal unit of $\Sp_G$: in this case, \cref{thmx:abstract}, $(a)$ reduces to the computation of the underlying set of $\Spc(\Sp_G^c)$ obtained by Strickland and Balmer--Sanders \cite{BalmerSanders17}. More generally, for such $R$, the tt-category $\Modd{G}(R)$ is equivalent to the tt-category of spectral Mackey functors with coefficients in $\Mod(R^e)$, see \cite[Corollary 4.11]{PatchkoriaSandersWimmer20pp} (for $R^e=\bbZ$) and \cite[Proposition 14.3]{bhs1} (for the general case), all Weyl group actions on the relevant spectra are trivial by \cref{prop:infl-guys}, and \cref{thmx:abstract} recovers the results of \cite[Part 5]{bhs1}.

Under additional assumptions on the equivariant ring spectrum $R$, we can find a more economical model of the decomposition of \cref{thmx:abstract}:

\begin{thmx}[\cref{thm:globalquillenstrat}]\label{thmx:global}
    If $R$ arises as the restriction of a global equivariant ring spectrum $R^{gl}$, then the Weyl group $W_G(H)$ in \cref{thmx:abstract} may be replaced by the `global Weyl group' $W_G^{gl}(H) = N_G(H)/H\cdot C_G(H)$, i.e., the centralizer acts trivially in this case. In particular, if $H$ is abelian, then it suffices to consider the action of the `Quillen--Weyl group' $W_G^Q(H) = N_G(H)/C_G(H)$ on $\Spc(\Perf(\Phi^HR))$.
\end{thmx}

We provide an example that demonstrates that the conclusion of this theorem fails for arbitrary equivariant ring spectra (\Cref{ex:gl-hyp-needed}) and that, even if it holds, the resulting action is in general neither free nor trivial (\cref{rem:freeness} and \cref{rem:res-not-injective}). While \cref{thmx:global} seems formal, it becomes very useful in practice when we try to compute actions in explicit examples. The next theorem collects several new stratification
results for prominent equivariant ring spectra from \cref{sec:examples}; each of them follows in a straightforward way from the novel techniques discussed so far.

\begin{thmx}\label{thmx:examples}
The category $\Modd{G}(R)$ is stratified in each of the following cases:
    \begin{enumerate}
        \item (\Cref{ex:constantgreen}) The integral constant Green functor $R=H\underline{\bbZ}$ for any cyclic $p$-group $G$, with stratified Balmer spectrum described in \Cref{spc:constantgreen}.
        \item (\Cref{thm:kug-stratification}) Equivariant $K$-theory $R = KU_G$ for any finite group $G$. In this case,  $\Spc(\Perff{G}(KU_G)) \cong \Spec(\pi_0KU_G)$, where $\pi_0KU_G \cong R(G)$ is the complex representation ring of $G$.
        \item (\Cref{exa:kr-stratification}) Atiyah's $K$-theory with reality $R = K\mathbb{R}$ for $G=C_2$. In this case, $\Spc(\Perff{C_2}(K\mathbb{R})) \cong \Spec(\Z)$.
    \end{enumerate}
\end{thmx}

In summary, our equivariant stratification theorem expresses a substantial part of the equivariant tt-geometry of any equivariant ring spectrum $R$ in terms of the non-equivariant tt-geometry of their geometric fixed points. As such, the list in \cref{thmx:examples} is not exhaustive, but rather intended as a proof of concept. In order to give a complete description of the equivariant tt-geometry of $R$, it remains to determine how the strata of $\Spc(\Perff{G}(R))$ in \cref{thmx:abstract}, $(a)$ are glued together. In general, this is a problem that poses substantial difficulties, as witnessed by the case of $\Sp_G$ for which we can still only resolve this question for abelian groups \cite{BHNNNS19}. To illustrate this point further, note that \cref{thmx:examples}, $(a)$ gives a short and independent proof of a recent result of Balmer--Gallauer \cite{BalmerGallauer2022pre} for cyclic $p$-groups, the case relevant for the motivic tt-geometry of Artin motives over finite fields. However, it requires additional techniques to determine the remaining specialization relations in \Cref{spc:constantgreen}, and this corresponds precisely to the gluing data between the strata. While for cyclic groups, this could be done `by hand' and is part of recent work of Balmer and Gallauer \cite{BalmerGallauer_announcement}, a more systematic understanding for arbitrary equivariant ring spectra $R$ would be desirable. 

Turning attention to the main example of interest to us in this paper, let $E=E_n$ be a Lubin--Tate $E$-theory of height $n$ at the prime $p$ and consider its $G$-Borel completion $\underline{E}$. From the point of view of chromatic homotopy theory, $E$ is a commutative ring spectrum of mixed characteristic\footnote{Here, being of characteristic $(p,n)$ is tested against the Morava $K$-theories $K(p,n)$, which form the prime fields of the stable homotopy category.} $((0), (p,1),\ldots,(p,n))$ and it is arguably the most important example of such a spectrum. It plays the same fundamental role in higher algebra as the $p$-adic integers $\mathbb{Z}_p$ in ordinary algebra; note that, viewed chromatically, the latter is of characteristic $((0),(p,\infty))$. 

Unsurprisingly, understanding the $E$-cohomology of finite groups has a long history, starting with the work of Ravenel from the 1970s, but also stimulating much recent activity, such as \cite{Stapleton2013}, \cite{MathewNaumannNoel2019}, \cite{Lurie_elliptic3}, and \cite{CMNN2020} with connections to algebraic $K$-theory. A version of Quillen's stratification for $\Spec(E^0(BG))$ was previously found by Greenlees and Strickland \cite{greenleesstrickland_varieties}. It is therefore natural to wonder if the categorification of Quillen's classical stratification theorem in the work of Benson--Carlson--Rickard \cite{BensonCarlsonRickard97} and Benson--Iyengar--Krause \cite{BensonIyengarKrause11a} admits a chromatic counterpart. Our next theorem provides an affirmative answer:

\begin{thmx}[\cref{thm:etheorycohomstratification,cor:decomp-specEBG}]\label{thmx:etheory}
Let $\underline{E}=\underline{E_n}$ be a $G$-Borel-equivariant Lubin--Tate $E$-theory of height $n$ and at the prime $p$. The category $\Modd{G}(\underline{E})$ is cohomologically stratified, and there is a decomposition into locally-closed subsets\footnote{Note that the first map is homeomorphism, while the second map provides only a stratification.}
\[
\Spc(\Perff{G}(\underline{E})) \cong \Spec(E^0(BG)) \simeq \bigsqcup_A \Spec(\pi_0\Phi^A\underline{E})/W_G^Q(A),
\]
where the disjoint union is indexed on abelian $p$-subgroups $A$ of $G$ generated by at most $n$ elements. In particular, the generalized telescope conjecture holds for $\Modd{G}(\underline{E})$ and there are explicit bijections
\[
\begin{tikzcd}[ampersand replacement=\&]
	{\begin{Bmatrix}
	\text{Thick $\otimes$-ideals of} \\ \Perff{G}(\underline{E})
	\end{Bmatrix}} 
	\& 
	{\begin{Bmatrix}
	\text{Specialization closed} \\
	\text{subsets of } \Spec(E^0(BG)) 
	\end{Bmatrix}}
	\arrow["\sim", from=1-1, to=1-2]
\end{tikzcd}
\]
and
\[
\begin{tikzcd}[ampersand replacement=\&]
	{\begin{Bmatrix}
	\text{Localizing $\otimes$-ideals of} \\ \Modd{G}(\underline{E})
	\end{Bmatrix}} 
	\& 
	{\begin{Bmatrix}
	\text{Subsets of} \\
	\Spec(E^0(BG)) 
	\end{Bmatrix}.}
	\arrow["\sim", from=1-1, to=1-2]
\end{tikzcd}
\]
\end{thmx}

We include a brief outline of the proof. Since $\underline{E}$ is Borel-complete and $E_n$ is complex orientable, in a first step we utilize work of Mathew, Naumann, and Noel \cite{MathewNaumannNoel2019} to reduce to the case of a finite abelian $p$-group $A$. From here on, our strategy diverges from the approach of Benson, Iyengar, and Krause at chromatic height $\infty$. The next step applies \cref{thmx:abstract} and \cref{thmx:global} to reduce further to the following two tasks:
    \begin{enumerate}
        \item Show that $\Mod(\Phi^A\underline{E})$ is stratified.
        \item Identify $\Spc(\Perff{A}(\underline{E}))$ with $\Spec(E^0(BA))$.
    \end{enumerate}
In order to solve (b), we use Balmer's comparison map \cite{Balmer10b} to compare the stratification of $\Spc(\Perff{A}(\underline{E}))$ to the analogous stratification of $\Spec(E^0(BA))$. In light of results proven by Dell'Ambrogio and Stanley \cite{DellAmbrogioStanley16}, both (a) and (b) are then consequences of the following key technical input to the proof of \cref{thmx:etheory}:

\begin{thmx}\label{thmx:regularity}
    The commutative ring $\pi_0\Phi^A\underline{E}$ is regular Noetherian for any finite abelian $p$-group $A$.
\end{thmx}

This regularity result is an elaboration on a theorem of Drinfel'd \cite{Drinfeld1974Elliptic} and its extension by Strickland \cite{Strickl1997Finite} on moduli of level structures, and its proof follows the analysis undertaken in \cite{BHNNNS19}. With that, we conclude our outline of the proof of \cref{thmx:etheory}.

\subsection*{Relation to previous work}

Quillen's work has found numerous extensions in various contexts; here, we briefly review the ones closest to the present work and indicate their relation. In the classical setting of a finite group $G$ and a field $k$ of characteristic $p$ dividing the order of $G$, Quillen first established a `weak version'\footnote{The adjective `weak' is used here to distinguish this version of Quillen's theorem from its `strong' form \eqref{eq:classicalquillen}, which gives more precise information about the spectrum. For instance, the latter readily identifies the underlying set  of $\Spec H^{\bullet}(G,k)$ as opposed to only providing a cover.} of the stratification theorem, in the form of a homeomorphism
\begin{equation}\label{eq:weakquillen}\tag{$\ddagger$}
{{\mathop{\colim}\limits_{G/E\in \mathcal{O}_{\mathcal{E}}(G)} \Spec H^{\bullet}(E,k) }} \cong \Spec H^{\bullet}(G,k),
\end{equation}
where the colimit is indexed on the orbit category of elementary abelian $p$-subgroups of $G$. On the one hand, Mathew--Naumann--Noel \cite{MathewNaumannNoel2019} produce a generalization of this formula for coefficients in an arbitrary commutative equivariant ring spectrum in place of $\underline{k}$. On the other hand, \eqref{eq:weakquillen} has found a tt-geometric incarnation for the spectrum of a $k$-linear tt-category $\cat K(G)$ which admits a tt-functor $\mathrm{D}^b(\Fp G) \to \cat K(G)$ in Balmer's \cite{Balmer16b}:
\[
{{\mathop{\colim}\limits_{G/E\in \mathcal{O}_{\mathcal{E}}(G)} \Spc(\cat K(E))}} \cong \Spc(\cat K(G)).
\]
The proof of \cref{thmx:abstract} uses parts of both of these results as input, and refines them tt-geometrically in three ways: firstly, \eqref{eq:ttquillen} establishes a `strong version' of Quillen's theorem as in \eqref{eq:classicalquillen}; in particular, the strata in our decomposition are given by \emph{non-equivariant} data via geometric fixed points. Secondly, we work in the generality of equivariant modules over an arbitrary equivariant ring spectrum. And thirdly, we also control the tt-geometry of not-necessarily compact objects in terms of stratification, akin to the passage from \cite{BensonCarlsonRickard97} to \cite{BensonIyengarKrause11a}. This last part is modelled on the approach taken in \cite{bhs1} for equivariant ring spectra with trivial action, and draws from the general theory of stratification developed therein.

The approach to probe the tt-geometry of equivariant categories through their geometric fixed points is based on the work of Strickland and Balmer--Sanders \cite{BalmerSanders17}. Our results have found inspiration and precursors in their work ($R = S_G^0$), \cite{bgh_balmer} ($R = S_G^0$, compact Lie groups), \cite{PatchkoriaSandersWimmer20pp} ($R = H\bbZ_G$), as well as \cite{bhs1} ($R = \infl(R^e)$); as such, the present paper unifies these different developments. 

A related notion of stratification also appears in the work of Ayala, Mazel-Gee, and Rozenblyum~\cite{AMGR2019}. There, the authors construct an adelic decomposition of the category of $G$-spectra in terms of geometric fixed points together with gluing data between these; a generalization to equivariant ring spectra with trivial action was subsequently given in \cite{AMGR2021}. However, in contrast to \cref{thmx:abstract}, this does not result in a classification of $\otimes$-ideals in these categories: from the point of view of tt-stratification as considered here, \cite{AMGR2019, AMGR2021} give an explicit manifestation of the local-to-global principle. 

An analogue of Quillen's stratification for the $E_n$-cohomology of finite group has also been studied by Greenlees and Strickland~\cite{greenleesstrickland_varieties}, extending earlier work of Hopkins--Kuhn--Ravenel \cite{HopkinsKuhnRavenel2000Generalized}. They prove a decomposition of the Zariski spectrum $\mathrm{Spec}(E_n^0(BG))$ into irreducible varieties indexed on abelian $p$-subgroups of $G$ of rank at most $n$. Pulled back to a pure (or monochromatic) stratum, their decomposition is into disjoint subsets, thereby establishing a version of the strong form \eqref{eq:classicalquillen} of Quillen stratification for these monochromatic strata.

\subsection*{Outline of the document}
The paper is organized as follows. We begin with preliminaries on tensor-triangular geometry in \Cref{sec:abstractstratification}, in particular reviewing some background material on stratification. In \Cref{sec:equivariant-homotopy} we introduce the equivariant stable homotopy category, and prove that stratification can be detected by the geometric fixed point functors (part $(b)$ of \Cref{thmx:abstract}). \Cref{section:strong-quillen,sec:strong-quillen-global} are dedicated to strong Quillen stratification in equivariant homotopy; in particular, they contain the proofs of part $(a)$ of \Cref{thmx:abstract} and of \Cref{thmx:global}. In \Cref{sec:cohomologicalstratification,sec:spectrum-borel-e} we establish stratification for Borel-equivariant Lubin--Tate $E$-theory and compute the Balmer spectrum, proving \Cref{thmx:etheory,thmx:regularity}, respectively. Finally,  \Cref{sec:examples} is dedicated to the proofs of the remaining examples in \Cref{thmx:examples}.

\subsection*{Notations and conventions} 
For a finite group $G$, we let $\Sp_G$ denote the $\infty$-category of genuine $G$-spectra. Given an algebra object $R\in\Alg(\Sp_G)$, we write $\Modd{G}(R)$ for the $\infty$-category of left $R$-modules in $G$-spectra, and $\Perff{G}(R)$ for its full subcategory of compact objects. We refer to a commutative algebra object $R\in\CAlg(\Sp_G)$ as a commutative equivariant ring spectrum, i.e., an $\mathbb{E}_{\infty}$-monoid in $\Sp_G$. We will often implicitly view a symmetric monoidal stable $\infty$-category as a tensor-triangulated category via the homotopy category functor $\Ho$. For example given an essentially small symmetric monoidal stable $\infty$-category $\cat C$, we will write $\Spc(\cat C)$ for the Balmer spectrum of the associated tensor-triangulated category $\Ho(\cat C)$. Similarly, we will use the homotopy category functor to extend common notions in tensor-triangulated geometry (such as that of a finite \'etale functor) to the world of $\infty$-categories, see~\cref{def:finite-etale}. Given a closed symmetric monoidal ($\infty$-)category $\cat C$, we denote the symmetric monoidal structure by $\otimes$, the internal hom by $\iHom$ and the unit object by $\unit$, and write $\mathbb{D}(-)=\iHom(-,\unit)$ for the duality functor.

\subsection*{Acknowledgements}

We are grateful to Scott Balchin, Paul Balmer, David Benson, Clover May, Beren Sanders, and Nat Stapleton for useful conversations about the subject matter of this paper, and we thank Akhil Mathew and Nat Stapleton for pointing out that something like \Cref{thm:regular} should be true. Special thanks to John Greenlees, Beren Sanders, and the anonymous referees for helpful comments on earlier versions of this manuscript.

TB is supported by the European Research Council (ERC) under Horizon Europe (grant No.~101042990) and would like to thank the Max Planck Institute for its hospitality. NC is partially supported by Spanish State Research Agency project PID2020-116481GB-I00, the Severo Ochoa and María de Maeztu Program for Centers and Units of Excellence in R$\&$D (CEX2020-001084-M), and the CERCA Programme/Generalitat de Catalunya. DH is supported by grant number TMS2020TMT02 from the Trond Mohn Foundation. NN and LP are supported by the SFB 1085 Higher Invariants in Regensburg. This material is partially based upon work supported by the Swedish Research Council under grant no.~2016-06596 while TB, NC, and LP were in residence at Institut Mittag-Leffler in Djursholm, Sweden during the semester Higher algebraic structures in algebra, topology and geometry. The authors would also like to thank the Hausdorff Research Institute for Mathematics for the hospitality in the context of the Trimester program Spectral Methods in Algebra, Geometry, and Topology, funded by the Deutsche Forschungsgemeinschaft under Germany’s Excellence Strategy – EXC-2047/1 – 390685813.
\newpage

\section{Tensor-triangular preliminaries}\label{sec:abstractstratification}
In this section we recall the key concepts from tensor-triangular geometry that we will use throughout the paper. With the exception of our abstract nilpotence theorem (\cref{thm:abstractnilpotence}), the material presented here is standard; we therefore invite the reader familiar with basic tt-geometry to skip this section and only refer to it for notation and terminology.

\subsection{Recollections on basic tt-geometry}

We begin by fixing some relevant terminology in tt-geometry, as developed by Balmer \cite{Balmer05a}.

\begin{Def}\label{def:tt-category}A \emph{tensor-triangulated category} (tt-category for short) is a triple $(\cat T,\otimes,\unit)$ consisting of a triangulated category $\cat T$ with symmetric monoidal product $\otimes$ which is exact in each variable, and tensor unit $\unit$. A \emph{tt-functor} $F \colon \cat T \to \cat S$ is an exact symmetric monoidal functor.
\end{Def}

\begin{Rec}\label{rec:thickstuff}
Let $\cat T$ be a tt-category. A triangulated subcategory $\cat J \subseteq \cat T$ is \emph{thick} if it is closed under retracts and it is called a \emph{$\otimes$-ideal} if $\cat T \otimes \cat J\subseteq \cat J$. 
For any class of objects $\cat E\subseteq \cat T$, we will write $\thick\langle\cat E\rangle$ for the smallest thick subcategory of $\cat T$ containing $\cat E$ and $\thickt{\cat E}$ for the smallest thick $\otimes$-ideal containing $\cat E$. If in addition $\cat T$ admits arbitrary (set-indexed) coproduts, we call a triangulated subcategory $\cat L \subseteq \cat T$ \emph{localizing} if it is closed under coproducts. For any class of objects $\cat E\subseteq \cat T$, we will write $\Loc\langle\cat E\rangle$ for the smallest localizing subcategory containing $\cat E$, and $\Loco{\cat E}$ for the smallest localizing $\otimes$-ideal containing $\cat E$. 

Observe that, for any exact functor $f\colon \cat T \to \cat S$ and any collection $\cat E$ of objects in $\cat T$, there is an inclusion $f(\thick\langle \cat E \rangle) \subseteq \thick\langle f(\cat E) \rangle$. This uses that the collection of $t \in \cat T$ with $f(t) \in \thick\langle f(\cat E) \rangle$ is thick and contains $\cat E$. If the exact functor $f$ additionally preserves coproducts, then the corresponding statement is true for localizing subcategories. The analogous results hold for tt-functors and ideals.
\end{Rec}

\begin{War}\label{war:meaningofthick}
When the tt-category $\cat T$ admits set-indexed coproducts, one may consider thick $\otimes$-ideals either in $\cat T$ or in the full subcategory $\cat T^c$ of $\cat T$ spanned by the compact objects. Both concepts will appear in this paper, and we will usually specify the ambient category in case it is not clear from context.
\end{War}

\begin{Rec}
Let $\cat K$ be an essentially small tt-category. Balmer~\cite{Balmer05a} constructed a topological space $\Spc(\cat K)$, called the \emph{spectrum} of $\cat K$, defined as follows: the points of $\Spc(\cat K)$ are the \emph{prime} $\otimes$-\emph{ideals} of $\cat K$, that is those thick $\otimes$-ideals $\cat J\subseteq \cat K$ which are proper and satisfy 
\[
\forall \; a,b\in\cat T \colon \; (a \otimes b \in \cat J \implies a \in \cat J \text{ or } b \in \cat J). 
\] The \emph{support} of an object $a \in \cat K$ is defined to be 
\[
\supp(a) = \SET{\cat P \in \Spc(\cat K)}{a \not \in \cat P},
\]
and the topology of $\Spc(\cat K)$ is the one having $\{\supp(a)\}_{a \in \cat K}$ as a basis of closed subsets.
\end{Rec}

\begin{Rem}
In \cite[Theorem 3.2]{Balmer05a} Balmer proves that the pair $(\Spc(\cat K),\supp)$ is the final support datum. Explicitly, this means that given a pair $(X,\sigma)$ consisting of a topological space $X$ and an assignment $\sigma$ which associates to any object $a \in \cat K$ a closed subset $\sigma(a) \subseteq X$ satisfying the conditions given in \cite[Definition 3.1]{Balmer05a}, then there exists a unique map $f \colon X \to \Spc(\cat K)$ such that $\sigma(a) = f^{-1}(\supp(a))$.

Moreover, by \cite[Theorem 4.10]{Balmer05a} there is an order-preserving bijection
\[
\begin{Bmatrix}
\text{Radical thick $\otimes$-ideals} \\
\text{of $\cat K$} 
\end{Bmatrix} 
\xymatrix@C=2pc{ \ar@<0.75ex>[r]^{\supp} &  \ar@<0.75ex>[l]^{\supp^{-1}}}
\begin{Bmatrix}
\text{Thomason subsets}\\
\text{of $\Spc(\cat K)$} \end{Bmatrix},
\]
where $\supp(\cat J) = \bigcup_{a\in \cat J}\supp(a)$, and $\supp^{-1}(\cat V) = \SET{a\in \cat K}{\supp(a) \subseteq \cat V}.$ Here a thick $\otimes$-ideal $\cat J$ is \emph{radical} if it satisfies $a^{\otimes n} \in \cat J \implies a \in \cat J$ and a subset $Y\subseteq \Spc(\cat K)$ is \emph{Thomason} if it is a union of closed subsets each of which has quasi-compact complement. Note that if $\cat K$ is rigid, i.e., every object has a dual (which it will be in all examples that we consider), then every thick $\otimes$-ideal is automatically radical \cite[Proposition 2.4]{Balmer07}. Moreover, if $\Spc(\cat K)$ is Noetherian, then Thomason subsets are exactly specialization closed subsets.
\end{Rem}

\begin{Rem}\label{rem:comparison_map}
In any essentially small tt-category $\cat K$ the graded endomorphism ring $R \coloneqq \End^*_{\cat K}(\unit)$ is graded commutative. In \cite[Corollary 5.6]{Balmer10b} Balmer defines a natural continuous \emph{comparison map}
\[
\begin{split}
\rho \colon \Spc(\cat K) &\to \Spec^h(R) \\
\cat P &\mapsto \{ f \mid \cone(f) \not \in \cat P \}.
\end{split}
\]
Let $R_0$ denote the commutative ring of degree zero elements in $R$. By combining $\rho$ with the continuous map $\Spec^h(R) \to \Spec(R_0)$ which sends a homogeneous prime ideal $\mathfrak p$ to $\mathfrak p^0 \coloneqq \mathfrak p \cap R_0$, we also obtain an \emph{ungraded comparison map}
\[
\rho_0 \colon \Spc(\cat K) \to \Spec(R_0).
\]
We recall that if $R$ is a 2-periodic evenly graded commutative ring, then 
\[
\xymatrix{\mathfrak q\mapsto \mathfrak q\cap R_0\colon \Spec^h(R)  \ar@<0.5ex>[r] &  \Spec(R_0) \ar@<0.5ex>[l] \noloc \mathfrak p\cdot R \mapsfrom \mathfrak p}
\] 
are inverse homeomorphisms. 

Suppose the graded endomorphism ring $R$ of the unit $\unit$ in $\cat K$ is 2-periodic. Then, if the comparison map $\rho \colon \Spc(\cat K) \to \Spec^h(R)$ is a homeomorphism, we deduce from the commutative diagram 
\[\begin{tikzcd}[ampersand replacement=\&]
	\Spc({\cat K}) \& {\Spec^h(R)} \\
	\& {\Spec(R_0)}
	\arrow["\rho", from=1-1, to=1-2]
	\arrow["\cong", from=1-2, to=2-2]
	\arrow["{\rho_0}"', from=1-1, to=2-2]
\end{tikzcd}\]
of \cite[Corollary 5.6]{Balmer10b} that the ungraded comparison map $\rho_0$ is also a homeomorphism. Under certain conditions $\rho$ is a homeomorphism if and only if it is  bijective, see \cref{prop:bijectioniffhomeomorphism} below.
\end{Rem}

\subsection{Stratification in tt-geometry}\label{ssec:ttstratification}

This subsection contains a rapid review of the theory of stratification of rigidly-compacty generated tt-categories $\cat T$ relative to the Balmer spectrum $\Spc(\cat T^c)$ as developed in \cite{bhs1}, as well as a comparison to the notion of cohomological stratification of Benson, Iyengar, and Krause \cite{BensonIyengarKrause11b}. 

\begin{Def}\label{def-big-tt}
A tt-category  $\cat T$ with arbitrary coproducts and a closed symmetric monoidal structure is \emph{rigidly-compactly generated} if it is compactly generated, the unit $\unit$ is compact, and every compact object is rigid, or dualizable (under these conditions, the compact and rigid objects of $\cat T$ coincide, see \cite[Theorem 2.1.3(d)]{HoveyPalmieriStrickland97}). We write $\cat K = \cat T^c$ for the subcategory of rigid and compact objects of $\cat T$. Note that $\cat T^c$ is an essentially small tt-category. 
\end{Def}
\begin{Exa}
Here are two examples, studied in \cite[Section  9]{HoveyPalmieriStrickland97}. 
\begin{enumerate}
\item For any commutative ring $R$, the derived category $\cat T = \cat D(R)$ of unbounded chain complexes of $R$-modules is a tt-category, with $\cat K$ the subcategory of complexes quasi-isomorphic to a bounded complex of finitely generated projective $R$-modules. 
\item For any finite group $G$, the homotopy category of $G$-spectra (studied in more detail in \Cref{sec:equivariant-homotopy}) is a rigidly-compactly generated tt-category. 
\end{enumerate}
\end{Exa}

\begin{Def}\label{def:noeth-tt-cat}
    A rigidly-compactly generated tt-category $\cat T$ is called \emph{Noetherian} if the endomorphism ring $\End_{\cat T}^*(\unit)$ is graded Noetherian and the module $\End_{\cat T}^*(C)$ is finitely generated over $\End_{\cat T}^*(\unit)$ for any compact object $C \in \cat T$.
\end{Def}

\begin{Prop}[{\cite[Corollary  2.8]{Lau2021Balmer}}]\label{prop:bijectioniffhomeomorphism}
If $\cat T$ is Noetherian, then the comparison map $\rho$ is a homeomorphism if and only if it is bijective. 
\end{Prop}

\begin{Exa}\label{exa:weakly_affine}
If $\cat T$ is an affine weakly regular tensor-triangulated category in the sense of \cite[Theorem 1.1]{DellAmbrogioStanley16}, then $\cat T$ is Noetherian, see \cite{DellAmbrogioStanley16erratum}. 
\end{Exa}

\begin{Exa}\label{ex:comm_rings}
For $R$ a commutative ring, work of Hopkins and Neeman \cite{Hopkins87,Neeman92a} (for $R$ Noetherian) and Thomason \cite{Thomason97} (in general) imply that the comparison map
\[
\rho \colon \Spc(\cat D(R)^c) \xrightarrow{\cong} \Spec(R)
\]
is a homeomorphism, see \cite[Proposition 8.1]{Balmer10b} for details. Moreover, Neeman has shown in \cite[Theorem 2.8]{Neeman92a} that if $R$ is Noetherian, then in fact $\Spec(R) \cong \Spc(\cat D(R)^c)$ parameterizes all localizing $\otimes$-ideals of $\cat D(R)$: there is a bijection
\[
\begin{Bmatrix}
\text{localizing $\otimes$-ideals} \\
\text{of $\cat D(R)$} 
\end{Bmatrix} 
\xymatrix@C=2pc{ \ar[r]^{\sim} &  }
\begin{Bmatrix}
\text{Subsets of $\Spc(\cat D(R)^c)$}
\end{Bmatrix}.
\]
\end{Exa}

\begin{Rem}
Following \Cref{ex:comm_rings}, one would like to know for which rigidly-compactly generated tt-categories (\Cref{def-big-tt}) the Balmer spectrum of compact objects also parameterizes the localizing $\otimes$-ideals. The first step towards doing this is to extend the notion of support from compact objects to all objects. To that end, suppose that $\cat T$ is a rigidly-compactly generated tt-category whose Balmer spectrum of compact objects is Noetherian. Then, Balmer--Favi \cite{BalmerFavi11} and Stevenson \cite{Stevenson13} have extended the universal support function from $\cat T^c$ to all of $\cat T$. We denote this extension by $\Supp$; we thus obtain maps: 
\begin{equation}\label{eq:localizing_ideals}
\begin{Bmatrix}
\text{Localizing $\otimes$-ideals} \\
\text{of $\cat T$} 
\end{Bmatrix} 
\xymatrix@C=2pc{ \ar@<0.75ex>[r]^{\Supp} &  \ar@<0.75ex>[l]^{\Supp^{-1}}}
\begin{Bmatrix}
\text{Subsets of $\Spc(\cat T^c)$} \end{Bmatrix}.
\end{equation}
\end{Rem}
Following and extending work of Benson--Iyengar--Krause \cite{BensonIyengarKrause11b} and Stevenson \cite{Stevenson13}, Barthel--Heard--Sanders \cite{bhs1} introduced the following definition. 

\begin{Def}\label{def:stratification}
Let $\cat T$ be a rigidly-compactly generated tt-category with $\Spc(\cat T^c)$ Noetherian. If the maps \eqref{eq:localizing_ideals} are inverse bijections, then we say that $\cat T$ \emph{is stratified}. 
\end{Def}
\begin{Rem}
The theory developed in \cite{bhs1} allows for the more general case where $\Spc(\cat T^c)$ is only \emph{weakly Noetherian}, i.e., every point in $\Spc(\cat T^c)$ is the intersection of a Thomason subset and the complement of a Thomason subset. In this paper, we will only consider the case where $\Spc(\cat T^c)$ is Noetherian, in which case the local-to-global principle of \cite[Definition 3.8]{bhs1} holds automatically \cite[Theorem 3.21]{bhs1}.
\end{Rem}

\begin{Rem}\label{rem:stratification_conditions}
Let $\cat T$ be a rigidly-compactly generated tt-category with $\Spc(\cat T^c)$ Noetherian. To each prime $\cat P \in \Spc(\cat T^c)$ one can associate a tensor idempotent $\Gamma_{\cat P}\unit\in\cat T$.\footnote{Alternatively denoted $g(\cat P)$ in \cite{bhs1}.} By \cite[Theorem 3.21 and Theorem 4.1]{bhs1}, stratification for $\cat T$ is equivalent to the following condition:
\begin{enumerate}
    \item[$(\ast)$] For each $\cat P \in \Spc(\cat T^c)$, $\Gamma_{\cat P}\cat T \coloneqq \Loco{\Gamma_{\cat P}\unit}$ is a minimal localizing $\otimes$-ideal of $\cat T$, or equivalently, for all $t\in \cat T$
    \[
    \Loco{t} = \Loco{\Gamma_{\cat P}\unit \mid \cat P \in \Supp(t) }.
    \]
\end{enumerate}
We say that $\cat T$ {\em has minimality at $\cat P$} if $\Gamma_{\cat P}\cat T$ is a minimal localizing $\otimes$-ideal of $\cat T$.
\end{Rem}
\begin{Rem}\label{rem:stratification-implies-telconj}
If $\cat T$ is a rigidly-compactly generated tt-category which is stratified and has a Noetherian spectrum $\Spc(\cat T^c)$, then we can deduce further consequences beyond a classification of the localizing $\otimes$-ideals of $\cat T$. For example:
    \begin{enumerate}
        \item the telescope conjecture holds \cite[Theorem 9.11]{bhs1}, i.e., the kernel of every smashing (i.e., coproduct-preserving) localization is generated by compact objects; 
        \item the Balmer--Favi support satisfies the tensor product property, i.e., there is an equality $\Supp(t_1 \otimes t_2) = \Supp(t_1) \cap \Supp(t_2)$ for all $t_1,t_2 \in \cat T$, by \cite[Theorem 8.2]{bhs1};
        \item  and the Bousfield lattice is isomorphic to the lattice of subsets of $\Spc(\cat T^c)$ \cite[Theorem 8.8]{bhs1}.
    \end{enumerate} 
\end{Rem}

\begin{Rem}\label{rem:first_bik_strat}
The theory of stratification is based on previous work of Benson--Iyengar--Krause \cite{BensonIyengarKrause11b}. Rather than working with the Balmer spectrum, they consider an action of a graded commutative Noetherian ring $R$ on $\cat T$, and parametrize localizing $\otimes$-ideals of $\cat T$ in terms of subsets of the homogeneous spectrum $\Spec^h(R)$ via a support theory $\Supp^{R \circlearrowright \cat T}$ that depends on the action. 
\end{Rem}

\begin{Rem}\label{rem:cohomstrat}
For a Noetherian tt-category $\cat T$ the following conditions are equivalent:
    \begin{enumerate}
        \item $\cat T$ is stratified and Balmer's comparison map $\rho$ is a homeomorphism.
        \item $\cat T$ is stratified by the action of $R=\End_{\cat T}^*(\unit)$ in the sense of Benson, Iyengar, and Krause \cite{BensonIyengarKrause11b}, i.e., the analogue of \cref{def:stratification} holds for the support theory $\Supp^{R \circlearrowright \cat T}$ of \cref{rem:first_bik_strat}.
    \end{enumerate}
The implication (b) implies (a) is proven in \cite[Corollary 7.14]{bhs1}, while the converse has been observed by Changhan Zou \cite{zou}. Indeed, assume (a) holds, then stratification and the homeomorphism $\rho$ give a bijection
\[\begin{tikzcd}[ampersand replacement=\&]
	{\begin{Bmatrix}\text{Localizing $\otimes$-ideals} \\ \text{of $\cat T$}\end{Bmatrix}} \& {\begin{Bmatrix}\text{Subsets of $\Spec^h(R)$} \end{Bmatrix}}.
	\arrow["\sim", from=1-1, to=1-2]
\end{tikzcd}\]
By \cite[Theorem 4.2]{BensonIyengarKrause11b} it suffices to show that this bijection is given by the Benson--Iyengar--Krause theory of support. This follows from \cite[Proposition 9.4]{Stevenson13}. A more thorough discussion can be found in \cite[Section 9]{zou}.
\end{Rem}

\begin{Def}\label{not:cohomstrat}
A rigidly-compactly generated Noetherian tt-category $\cat T$ is said to be \emph{cohomologically stratified} if the equivalent conditions of \cref{rem:cohomstrat} hold for $\cat T$ equipped with the action of $\End_{\cat T}^*(\unit)$.
\end{Def}

\begin{Exa}
If $R$ is a commutative Noetherian ring, then Neeman's theorem (\Cref{ex:comm_rings}) can be reinterpreted as the statement that $\cat D(R)$ is cohomologically stratified. More generally, if $X$ is a Noetherian scheme and $\Derqc(X)$ denotes the derived category of complexes of $\mathcal{O}_X$-modules with quasi-coherent cohomology, then $\Spc(\Derqc(X)^c) \cong |X|$, the underlying topological space of $X$, and $\Derqc(X)$ is stratified, see \cite[Corollary 8.13]{Stevenson13} and \cite[Corollary 5.10]{bhs1}. Note that the comparison map $\rho$ will not be a homeomorphism in general in this case \cite[Remark 8.2]{Balmer10b}, so that $\Derqc(X)$ is not cohomologically stratified. 
\end{Exa}

\begin{Rem}
The approach of stratification via the Balmer spectrum developed systematically in \cite{bhs1} separates the classification of localizing ideals into two parts: one is to determine the set underlying the spectrum $\Spc(\cat T^c)$, and the other is to show that it stratifies. This is the approach that we will take in this paper. In particular, note that stratification in this sense only relies on an understanding of the underlying \emph{set} of $\Spc(\cat T^c)$, rather than its \emph{topology}. 

In a second step, one can then try to determine the topology on $\Spc(\cat T^c)$, which by Balmer's work \cite{Balmer05a} is tantamount to the classification of thick tensor ideals of the compact objects in $\cat T$. This is in contrast to the approach of \cite{BensonIyengarKrause11b}, which solves both classification problems simultaneously via the auxiliary action of a Noetherian commutative ring on $\cat T$. In practice, this separation is very useful, as there are examples where we can show stratification without a full understanding of the topology on $\Spc(\cat T^c)$, see \cite[Remark 15.12]{bhs1} as well as the results of this paper. 
\end{Rem}

\subsection{A tt-nilpotence theorem}

The next result constitutes a tt-theoretic generalization of the equivariant nilpotence theorem of \cite[Theorem 3.19]{bgh_balmer}, which in turn was inspired by the abstract nilpotence theorem of \cite[Theorem 5.1.2]{HoveyPalmieriStrickland97}. Ultimately, the idea of the proof of \cref{thm:abstractnilpotence} is due to Devinatz, Hopkins, and Smith \cite{DevinatzHopkinsSmith88,HopkinsSmith98}.

\begin{Def}\label{def:jointlynilfaithful}
We say that a collection $(F_i\colon \cat T \to \cat S_i)_{i\in I}$ of exact symmetric monoidal functors between tt-categories is \emph{jointly conservative} if an object $X \in \cat T$ is zero if and only if $F_i(X) = 0$ for all $i \in I$. The collection $(F_i)_{i\in I}$ is said to \emph{detect $\otimes$-nilpotence of morphisms with dualizable source} if the following condition holds: a morphism $\alpha\colon C \to X$ in $\cat T$ with $C$ dualizable satisfies $\alpha^{\otimes m}=0$ for some $m \ge 0$ if and only if for all $i\in I$ there exists $m_i$ such that $F_i(\alpha)^{\otimes m_i} = 0$. 
\end{Def}

Note that in the definition of detection of $\otimes$-nilpotence we do not assume that the $m_i$ are uniformly bounded.

\begin{Thm}\label{thm:abstractnilpotence}
Let $(\Phi_i\colon \cat T \to \cat S_i)_{i\in I}$ be a collection of coproduct-preserving tt-functors between rigidly-compactly generated tt-categories, indexed on a set $I$. If the functors $\Phi_i$ are jointly conservative on $\cat T$, then they detect $\otimes$-nilpotence of morphisms with dualizable source.
\end{Thm}
\begin{proof}
Consider a morphism $\alpha \colon C \to X$ in $\cat T$ with $C$ dualizable. Let $Y=\iHom(C,X)$, and write $y\colon \unit \to Y$ for the adjoint of $\alpha$. The dualizability of $C$ implies that $\iHom(C^{\otimes k},X^{\otimes k}) \simeq \mathbb{D}(C)^{\otimes k}\otimes X^{\otimes k} \simeq \iHom(C,X)^{\otimes k}$ for all $k\ge 0$. Under these identifications, the adjoint of $\alpha^{\otimes k}$ is the composition $y^{(k)}\colon \unit \xrightarrow{y} Y \xrightarrow{y} Y^{\otimes 2} \xrightarrow{y} \ldots \xrightarrow{y} Y^{\otimes k}$, as one easily checks by observing that $y^{(k)}$ identifies with $y^{\otimes k}$ (see also \cite[p.17]{HopkinsSmith98}). Therefore, $\alpha$ is $\otimes$-nilpotent if and only if $y$ is nilpotent, i.e., $y^{(m)}\colon \unit \xrightarrow{y} Y \xrightarrow{y} Y^{\otimes 2} \xrightarrow{y} \ldots \xrightarrow{y} Y^{\otimes m}$ is zero for some $m$. Now, denote the homotopy colimit of the infinite composite sequence by
\[
T(y)=\hocolim(\unit \xrightarrow{y} Y \xrightarrow{y} Y^{\otimes 2} \xrightarrow{y} \ldots).
\]
Since the unit in $\cat T$ is compact, the argument in  \cite[Lemma 2.4]{HopkinsSmith98} shows that $y$ is nilpotent if and only if $T(y)=0$. 

With this preparation, we can now argue as follows. The map $\alpha$ is $\otimes$-nilpotent if and only if $T(y)=0$. Since the collection $(\Phi_i)_{i\in I}$ is jointly conservative, $T(y)=0$ if and only if $\Phi_i(T(y))= 0$ for all $i \in I$. As $\Phi_i$ preserves colimits and tensor products, $\Phi_i(T(y))= T(\Phi_i(y))$ for any $i\in I$. Then, in turn, $\alpha$ is $\otimes$-nilpotent if and only if $\Phi_i(y)$ is nilpotent for all $i \in I$, i.e., that $\Phi_i(\alpha)$ is $\otimes$-nilpotent for all $i \in I$. The last statement uses that $\Phi_i(C)$ is dualizable, following the argument in the first paragraph of the proof.
\end{proof}

\begin{Cor}\label{cor:surjective-map}
Let $\Phi \colon \cat T \to \cat S$ be a coproduct-preserving tt-functor between rigidly-compactly generated tt-categories. If the functor $\Phi$ is conservative, then $\Phi$ induces a surjective map 
\[
\xymatrix{\Spc(\Phi^c)\colon \Spc(\cat S^c) \ar[r] & \Spc(\cat T^c)}
\]
on Balmer spectra.
\end{Cor}
\begin{proof}
By \cref{thm:abstractnilpotence}, the functor $\Phi$ detects $\otimes$-nilpotence between compact objects in $\cat T$. It then follows from \cite[Theorem 1.3]{Balmer18} that $\Spc(\Phi^c)$ is surjective. 
\end{proof}

\begin{Rem}
This result is a `big' variant of Balmer's \cite[Theorem 1.2]{Balmer18}, with a stronger hypothesis and a stronger conclusion.
\end{Rem}

\subsection{Base-change, separable algebras, and \'etale morphisms}
In this section we collect some facts about module categories internal to a tt-category (or symmetric monoidal stable $\infty$-category), base-change, and recall the notions of separable algebras and
\'etale morphisms for tt-categories. The main references for this material are \cite{BalmerDellAmbrogioSanders15} and \cite{Balmer16b}.

\begin{Rec}\label{rec:module_categories}
Let $\cat C$ be a rigidly-compactly generated symmetric monoidal stable $\infty$-category and consider a commutative algebra $A \in \CAlg(\cat C)$ in $\cat C$. There is then an associated symmetric monoidal $\infty$-category of \emph{modules over $A$ internal to $\cat C$}, denoted $\Mod_{\cat C}(A)$, along with an \emph{extension of scalars functor} $F_A\colon \cat C \to \Mod_{\cat C}(A)$. If $\mathcal{G}$ is a set of compact generators for $\cat C$, then $\Mod_{\cat C}(A)$ is rigidly-compactly generated by the set $A\otimes \mathcal{G} \coloneqq \{ F_A(G) \mid G\in \mathcal{G}\}$. Indeed, the claim about compact generation follows because the forgetful functor $U_A\colon \Mod_{\cat C}(A) \to \cat C$ is conservative and preserves colimits \cite[Corollary 4.2.3.5]{HALurie}, and so its left adjoint $F_A \colon \cat C \to \Mod_{\cat C}(A)$ preserves generating sets and compact objects \cite[Proposition 5.5.7.2]{HTTLurie}. Combining \cite[Theorem 4.5.2.1 and Theorem 4.5.3.1]{HALurie} we see moreover that $\Mod_{\cat C}(A)$ inherits a symmetric monoidal structure from $\cat C$ and that $F_A$ is symmetric monoidal. It follows that $F_A$ preserves dualizable objects, and hence that $\Mod_{\cat C}(A)$ is rigidly-compactly generated. 
\end{Rec}

\begin{Rec}\label{rec:basechange_functoriality}
Let $\theta \colon \cat C \to \cat D$ be a lax monoidal functor between presentable symmetric-monoidal $\infty$-categories whose tensor product commutes with all colimits. For any commutative algebra $A \in \CAlg(\cat C)$, there is a diagram 
\[\begin{tikzcd}[ampersand replacement=\&]
	{\cat C} \& {\cat D} \\
	{\Mod_{\cat C}(A)} \& {\Mod_{\cat D}(\theta(A)),}
	\arrow["{U_A}"', shift right=1, from=2-1, to=1-1]
	\arrow["\theta"', from=2-1, to=2-2]
	\arrow["\theta", from=1-1, to=1-2]
	\arrow["{U_{\theta(A)}}"', shift right=1, from=2-2, to=1-2]
	\arrow["{F_A}"', shift right=1, from=1-1, to=2-1]
	\arrow["{F_{\theta(A)}}"', shift right=1, from=1-2, to=2-2]
\end{tikzcd}\]
where $F$ and $U$ denote extension and restriction of scalars along the units of $A$ and $\theta(A)$, respectively. The functor $\theta$ is compatible with restriction of scalars, in the sense that $\theta \circ U_A \simeq  U_{\theta(A)} \circ \theta$. If $\theta$ is \emph{symmetric} monoidal and preserves geometric realizations of simplicial objects (for example, if $\theta$ preserves colimits), then we additionally have $\theta \circ F_A \simeq F_{\theta(A)} \circ \theta$. Both of these statements follow from the discussion in \cite[Remark 1.1.11]{ergus2022hopf}. We also note that under these stronger assumptions, the induced functor $\theta \colon \Mod_{\cat C}(A) \to \Mod_{\cat D}(\theta(A))$ is symmetric monoidal \cite[Proposition 3.18]{Robalo}, and that if $\theta \colon \cat C \to \cat D$ preserves colimits, then so does $\theta \colon \Mod_{\cat C}(A) \to \Mod_{\cat D}(\theta(A))$ as one can check using that $U_{\theta(A)}$ is conservative and preserves colimits.
\end{Rec}

\begin{Rem}
In general, one does not expect the category of modules over a commutative algebra object $A$ in a tt-category to inherit the structure of a tt-category. However, under suitable conditions on $A$, it does. One such notion is that of a separable algebra, first introduced in the tt-context by Balmer in \cite{Balmer11}, which turns out to be especially well-behaved. In particular, it allows us to prove an extension result for natural transformations, see \cref{prop:nat-tra} below. 
\end{Rem}

\begin{Def}\label{def:separablealgebra}
Let $\cat T$ be a tt-category. A ring object $A \in \cat T$ is \emph{separable} if the multiplication map $\mu \colon A \otimes A \to A$ admits a section as a map of $(A,A)$-bimodules. Explicitly, there exists $\sigma \colon A \to A \otimes A$ such that $\mu \sigma = \text{id}_A$ and $(\unit \otimes \mu) \circ (\sigma \otimes \unit) = \sigma \mu = (\mu \otimes \unit) \circ (\unit \otimes \sigma) \colon A \otimes A \to A \otimes A$. 
\end{Def}

\begin{Rem}\label{rem:ttdegree}
   Under some mild conditions on the tt-category, any separable ring object has a notion of tt-degree by work of Balmer~\cite{Balmer14}. The tt-degree can either be a non-negative integer or be equal to infinity. In practice, the additional assumption of having finite tt-degree is harmless, see for instance~\cite[Section 4]{Balmer14}.
\end{Rem}

\begin{Def}\label{def:finite-etale}
    Let $\cat T$ and $\cat S$ be rigidly-compactly generated tt-categories and let $F\colon \cat T \to \cat S$ be a coproduct-preserving tt-functor with right adjoint $G$. Following \cite{Balmer16b}, we say that $F$ is \emph{finite \'{e}tale} if there is a compact commutative separable algebra~$A$ of finite tt-degree in~$\cat T$ and an equivalence of tt-categories $\cat S\cong \Mod_{\cat T}(A)$ under which the adjunction $F\colon\cat T \leftrightarrows \cat S \colon G$ is identified with the extension-of-scalars/restriction adjunction $F_A\colon \cat T \leftrightarrows  \Mod_{\cat T}(A) \colon U_A$.  A colimit-preserving symmetric monoidal functor between rigidly-compactly generated symmetric monoidal stable $\infty$-categories $F\colon \cat C \to \cat D$ is \emph{finite \'etale} if $\Ho(F)\colon \Ho(\cat C) \to \Ho(\cat D)$ is finite \'etale.
\end{Def}

\begin{Rem}\label{rem:hofinite-etale}
Let $\cat C$ be a rigidly-compactly generated symmetric monoidal stable $\infty$-category and consider $A \in \CAlg(\cat C)$ such that $A \in \Ho(\cat C)$ is a separable algebra of finite degree. Combining the results of \cite{DellAmbrogioSanders18} and \cite[Proposition 3.8]{Sanders21pp} (for the symmetric monoidal structure), there is then a tt-equivalence
\[
\Ho(\Mod_{\cat C}(A)) \simeq \Mod_{\Ho(\cat C)}(A).
\]
In particular, it follows that the base-change functor $\cat C \to \Mod_{\cat C}(A)$ is finite \'etale in the sense of \cref{def:finite-etale}. In other words, this provides the compatibility between enhanced notions of \'etale and their effect on homotopy categories. 
\end{Rem}

We will need the next result later on.

\begin{Prop}\label{prop:nat-tra}
    Let $\cat T$ and $\cat S$ be tt-categories and let $A\in \CAlg(\cat T)$ be separable. Let $F_A\colon \cat T \leftrightarrows \cat \Mod_{\cat T}(A) \colon U_A$ denote the free-forgetful adjunction. Suppose we are given two tt-functors $H_0,H_1 \colon \Mod_{\cat T}(A) \to \cat S$ and a natural transformation $\eta \colon H_0\circ F_A \Rightarrow H_1\circ F_A$. Then there exists a unique natural transformation $\widetilde{\eta}\colon H_0 \Rightarrow H_1$ which extends $\eta$.
\end{Prop}
\begin{proof}
    Since $A$ is separable the counit of the adjunction $\epsilon\colon F_AU_A \Rightarrow 1$ admits a section $\sigma \colon 1 \Rightarrow F_A U_A$~\cite[Proposition 3.11]{Balmer11}. For any $M\in\Mod_{\cat T}(A)$, we define $\widetilde{\eta}_M$ to be given by the composite
    \begin{equation}\label{nat-transf}
        H_0M \xrightarrow{H_0\sigma_M} H_0F_AU_AM \xrightarrow{\eta_{U_A M}} H_1F_AU_A M \xrightarrow{H_1\epsilon_M} H_1 M.
    \end{equation}
    This is natural in $M$ since $\eta_{U_AM},\epsilon_M$ and $\sigma_M$ are so. For any $X\in \cat T$, we have a commutative diagram
    \[
      \begin{tikzcd}[column sep=1.5cm]
          & H_0F_A X \arrow[r,"\eta_X"] & H_1F_A X \arrow[rd, "1"] \\
        H_0F_A X \arrow[ur, "1"]\arrow[rrr, bend right,"\widetilde{\eta}_{F_AX}"] \arrow[r,"H_0\sigma_{F_A X}"'] & H_0F_AU_AF_A X \arrow[u,"H_0\epsilon_{F_AX}"'] \arrow[r,"\eta_{U_AF_A X}"'] & H_1 F_A U_AF_A X \arrow[u,"H_1\epsilon_{F_A X}"] \arrow[r,"H_1\epsilon_{F_AX}"']& H_1 F_A X
      \end{tikzcd}
    \]
    showing that $\widetilde{\eta}$ does extend $\eta$. Finally,  we note that any natural transformation extending $\eta$ must be compatible with $\epsilon$ and $\sigma$ and hence must be given by the formula \eqref{nat-transf}.
\end{proof}

\section{Stratifications in stable equivariant homotopy theory}\label{sec:equivariant-homotopy}
In this section we begin our study of the tt-geometry of modules internal to the category of $G$-equivariant spectra, for $G$ a finite group. We establish a nilpotence theorem for any commutative equivariant ring spectrum $R$ (\cref{thm:eqnilpotence}), generalizing the one for the equivariant sphere due to Balmer and Sanders \cite[Theorem 4.15]{BalmerSanders17}, show that the Balmer spectrum of perfect equivariant $R$-modules is covered by the spectra of perfect modules over the geometric fixed points (\cref{cor:eqspc_surjectivity}), and finally prove that stratification of $\Mod_G(R)$ can be detected by geometric fixed point functors (\cref{thm:eqstratdescent}). 

\subsection{Recollections on equivariant homotopy}
Here we introduce some relevant notions from equivariant stable homotopy theory in the context of equivariant module categories over a commutative equivariant $G$-ring spectrum, for $G$ a finite group. In particular, with a view towards our tt-theoretc applications, we establish the naturality of the various constructions with respect to change of groups. This material is standard, albeit not readily available in the literature as far as we are aware; our exposition in this subsection and the next one loosely follows and generalizes the approach taken in \cite{PatchkoriaSandersWimmer20pp}.

\begin{Rec}\label{rec:infl-res}
 Let $G$ be a finite group. We let $\Sp_G$ denote the stable $\infty$-category of $G$-spectra as constructed in \cite[Definition C.1]{GepnerMeier} with the 
 symmetric monoidal structure arising from~\cite[Corollary C.7]{GepnerMeier}. By~\cite[Proposition C.9]{GepnerMeier} this is symmetric monoidally equivalent to the underlying $\infty$-category associated to the category of orthogonal $G$-spectra with the stable model structure. We denote the symmetric monoidal structure on $\Sp_G$ by $\otimes$ and the unit object by $S_G^0$. The $\infty$-category $\Sp_G$ admits a set of compact and dualizable generators $G/H_+$ for $H \subseteq  G$ (we omit writing the $G$-suspension spectrum functor). This statement can be checked in the homotopy category (see~\cite[Remark 1.4.4.3]{HALurie} for the compact generation), where it is shown in, for example, \cite[Theorem 9.4.3]{HoveyPalmieriStrickland97}. The homotopy category $\Ho(\Sp_G)$ is then a rigidly-compactly generated tt-category in the sense of \cref{def-big-tt}.
 
Let $\cat S_*^{G}$ be the symmetric monoidal $\infty$-category of pointed $G$-spaces. For any group homomorphism $\alpha \colon G' \to G$ we get an essentially unique symmetric monoidal left adjoint $\alpha^* \colon \Sp_G \to \Sp_{G'}$ (see \cite[Remark 3.13]{PatchkoriaSandersWimmer20pp}) making the diagram
\[\begin{tikzcd}[ampersand replacement=\&]
	{\cat S_*^{G'}} \& {\cat S^G_*} \\
	{\Sp_{G'}} \& {\Sp_G}
	\arrow["{\alpha^*}", from=1-1, to=1-2]
	\arrow["{\Sigma^{\infty}}"', from=1-1, to=2-1]
	\arrow["{\alpha^*}"', from=2-1, to=2-2]
	\arrow["{\Sigma^{\infty}}", from=1-2, to=2-2]
\end{tikzcd}\]
commute. At the level of orthogonal $G$-spectra, this functor is constructed in \cite[Construction 3.1.15]{Schwede18_global}, see also the remark at the top of page 353 of \cite{Schwede18_global}. If $\alpha \colon H \hookrightarrow G$ is the inclusion of a subgroup, then we call the resulting functor \emph{restriction}, denoted $\res_H \colon \Sp_G \to \Sp_H$. If $\alpha \colon G \twoheadrightarrow G/N$ is the quotient by a normal subgroup, then the resulting functor will be denoted $\infl_{G/N}\colon \Sp_{G/N} \to \Sp_{G}$ and will be referred to as \emph{inflation}.
\end{Rec}

\begin{Not}
For a finite group $G$ and $R \in \CAlg(\Sp_G)$ we let $\Modd{G}(R)$ denote the $\infty$-category of $R$-modules internal to $\Sp_G$, and write $\Perf_G(R)$ for its full subcategory of compact (or perfect) modules. 
\end{Not}

We now collect some basic facts about the structure of these equivariant module categories and base-change functors between them. 

\begin{Lem}\label{lem:eqmodulegenerators}
If $R\in \mathrm{CAlg}(\Sp_G)$, then the $\infty$-category of modules $\Modd{G}(R)$ is a rigidly-compactly generated symmetric monoidal stable $\infty$-category, with a set of compact generators given by $\SET{R \otimes G/H_+}{H \subseteq  G}$. 
\end{Lem}
\begin{proof}
This is \Cref{rec:module_categories} and \cref{rec:infl-res} for $\cat C = \Sp_G$.
\end{proof}

\begin{Lem}\label{lem:abstract_end_finite}
 Let $R \in \CAlg(\Sp_G)$ with $\pi^G_*(R)$ graded Noetherian. Suppose that for all subgroups $H \subseteq  G$,  $\pi_*^H(R)$ is a finitely generated $\pi_*^G(R)$-module via restriction, then $\Mod_G(R)$ is Noetherian (\Cref{def:noeth-tt-cat}).
\end{Lem}
\begin{proof}
Firstly, we note that $\End^*_{\Mod_G(R)}(R)=\pi_*^G(R)$ which is graded Noetherian by our assumption. Let $X \in \Perff{G}(R)$ and write $\mathbb{D}(X)$ for the internal dual of $X$ in $\Modd{G}(R)$. We are required to show that $\End^*_{\Modd{G}(R)}(X)$ is a finitely generated $\pi_*^G(R)$-module. By adjunction and dualizability of $X$ we have
\[
\End_{\Modd{G}(R)}(X) \simeq \Hom_{\Modd{G}(R)}(R,\mathbb{D}(X) \otimes_R X).
\]
Since $\mathbb{D}(X) \otimes_R X$ is still in $\Perff{G}(R)$, we are reduced to showing that
 \[
\Hom^*_{{\Modd{G}(R)}}(R, X)
 \]
 is a finitely generated $\pi_*^G(R)$-module for each $X \in \Perff{G}(R)$. 
 Using a thick subcategory argument and the fact that $\{ G/H_+ \otimes R \}_{H \subseteq  G}$ is a set of compact generators for $\Modd{G}(R)$, we then reduce to showing that 
 \[
 \Hom^*_{{\Modd{G}(R)}}(R, G/H_+ \otimes R)
 \]
is a finitely generated $\pi_*^G(R)$-module for all $H \subseteq  G$. This now follows from our assumptions since 
  \[
 \Hom^*_{{\Modd{G}(R)}}(R, G/H_+ \otimes R) \cong \Hom^*_{\Sp_G}(S^0, G/H_+ \otimes R) \cong \pi_*^H(R)
 \]
is finitely generated over $\pi_*^G(R)$.
\end{proof}

\begin{Rem}\label{rem:group_functoriality}
Let $\alpha \colon G' \to G$ be a group homomorphism and $R\in\CAlg(\Sp_G)$. By the discussion in \cref{rec:basechange_functoriality}, there is a commutative diagram
\[\begin{tikzcd}[ampersand replacement=\&]
	{\Sp_G} \& {\Sp_{G'}} \\
	{\Modd{G}(R)} \& {\Modd{G'}(\alpha^*(R)).}
	\arrow["{\alpha^*}", from=1-1, to=1-2]
	\arrow["{F_R}"', shift right=1, from=1-1, to=2-1]
	\arrow["{\alpha^*}"', from=2-1, to=2-2]
	\arrow["{F_{\alpha^*(R)}}"', shift right=1, from=1-2, to=2-2]
	\arrow["{U_R}"', shift right=1, from=2-1, to=1-1]
	\arrow["{U_{\alpha^*(R)}}"', shift right=1, from=2-2, to=1-2]
\end{tikzcd}\]
If $\alpha \colon H \hookrightarrow G$ is the inclusion of a subgroup, then we refer to the induced symmetric monoidal functor $\res_H \colon \Mod_G(R) \to \Mod_H(\res_HR)$ as \emph{restriction}.
\end{Rem}

\begin{Rem}
From now on, we will usually omit writing the restriction functor on an object, i.e., we simply write
\[
\res_H \colon \Mod_G(R) \to \Mod_H(R)
\]
for brevity. 
\end{Rem}

\begin{Rem}\label{rem:coind-cons}
 The restriction functors $\Sp_G \to \Sp_H$ and $\Mod_G(R)\to \Mod_H(R)$ are symmetric monoidal and preserve colimits, so they admit  lax monoidal right adjoints, both denoted $\CoInd_H$ and called \emph{coinduction}. These make the following diagram commute (by \Cref{rem:group_functoriality})
 \[
 \begin{tikzcd}
  \Sp_H \arrow[r,"\CoInd_H"] & \Sp_G  \\
  \Mod_H(R) \arrow[u, "U_R"] \arrow[r, "\CoInd_H"'] & \Mod_G(R) \arrow[u,"U_R"'].
 \end{tikzcd}
 \]
 The top horizontal functor is conservative by the proof of~\cite[Theorem 5.32]{MathewNaumannNoel17} and the vertical functors are conservative by~\cite[Corollary 4.2.3.2]{HALurie}. It then follows that the bottom horizontal functor is also conservative. 
\end{Rem}

\begin{Lem}\label{lem:res_finiteetale}
For any subgroup $H \subseteq  G$ and $R\in \CAlg(\Sp_G)$, the restriction functor
\[
\xymatrix{\res_H\colon \Modd{G}(R) \ar[r] & \Modd{H}(R)}
\]
is a finite \'etale tt-functor with corresponding separable commutative algebra given by $\mathbb{D}(G/H_+)\otimes R$.
\end{Lem}

\begin{proof}
 There is a commutative diagram of $\infty$-categories (\Cref{rem:group_functoriality})
 \[
 \begin{tikzcd}
  \Sp_G \arrow[r,"\res_H"] \arrow[d,"F_R"'] & \Sp_H \arrow[d,"F_R"] \\
  \Mod_G(R) \arrow[r,"\res_H"'] & \Mod_H(R),
 \end{tikzcd}
 \]
 where the vertical arrows are the extension of scalars functors. Passing to homotopy categories, we obtain a commutative diagram of rigidly-compactly generated tt-categories in which the top horizontal arrow is finite \'etale with corresponding separable commutative algebra $A=\mathbb{D}(G/H_+)$. Indeed, $A$ is separable by~\cite[Theorem 1.1]{BalmerDellAmbrogioSanders15}. We can check that $\mathbb{D}(G/H_+)$ has finite tt-degree by testing on all geometric fixed points~\cite[Theorem 3.7(b)]{Balmer14} and applying~\cite[Corollary 4.8]{Balmer14}. These results together with the second part of~~\cite[Theorem 1.1]{BalmerDellAmbrogioSanders15} show that the restriction functor $\res_H \colon \Ho(\Sp_G) \to \Ho(\Sp_H)$ is finite \'etale with $A = \mathbb{D}(G/H_+)$.
 
 Using \Cref{rem:coind-cons} and~\cite[Equation 5.15]{MathewNaumannNoel17} we can apply \cite[Lemma 5.5]{Sanders21pp} to see that the bottom horizontal map in the diagram is finite \'etale. By construction, the separable algebra associated to $\res_H$ must be $\CoInd_H \res_H R$. We claim that $\CoInd_H \res_H R\simeq \mathbb{D}(G/H_+)\otimes R$ as commutative algebras in $\Sp_G$. Indeed, by \cite[Construction 5.30]{MathewNaumannNoel17} there is a map $\res_H \mathbb{D}(G/H_+) \to S^0_H$ in $\CAlg(\Sp_H)$ which by base-change and symmetric monoidality of $\res_H $ gives rise to a map $\res_H (\mathbb{D}(G/H_+) \otimes R) \to \res_H R$ in  $\CAlg(\Mod_H(\res_H R))$. By adjunction we get a map $\mathbb{D}(G/H_+) \otimes R \to \CoInd_H \res_H R$ in $\CAlg(\Mod_G(R))$. This map is an equivalence by the projection formula for the adjunction $(\res_H, \CoInd_H)$, see for instance \cite[Theorem 1.3]{BalmerDellAmbrogioSanders16}.
\end{proof}

\begin{Prop}\label{prop:base-change-naturality}
 Let $R\in \CAlg(\Sp_G)$ and a subgroup $H \subseteq  G$ be given. There is a symmetric monoidal equivalence
 \[
 L_H\colon \Mod_G(\mathbb{D}(G/H_+)\otimes R) \rightarrow \Mod_H(R)
 \]
 which makes the following diagram commute up to isomorphism
 \[
 \begin{tikzcd}
 \Mod_G(R)\arrow[d,"\res_H"'] \arrow[dr,"F_{\mathbb{D}(G/H_+)\otimes R}"]& \\
 \Mod_H(R)  & \Mod_G(\mathbb{D}(G/H_+)\otimes R), \arrow[l,"\sim"',"L_H"]
 \end{tikzcd}
 \]
 where $F_{\mathbb{D}(G/H_+)\otimes R}$ denotes the extension of scalars functor. Furthermore, the equivalence $L_H$ is natural with respect to the extension of scalars functors along morphisms in $\CAlg(\Sp_G)$.
\end{Prop}

\begin{proof}
    The equivalence for $R=S_G^0$ is proved in~\cite[Theorem 5.32]{MathewNaumannNoel17} as an application of~\cite[Proposition 5.29]{MathewNaumannNoel17}. The commutativity of the diagram follows from the diagram in~\cite[Construction 5.23]{MathewNaumannNoel17} by passing to left adjoints. The general case follows from this by combining~\cite[Lemma 3.9]{BehrensShah2020equivariant} and~\cite[Proposition 5.29]{MathewNaumannNoel17}. Again, as in the previous lemma, we are using the fact that $\CoInd_H \res_H R\simeq \mathbb{D}(G/H_+)\otimes R$ as commutative algebras in $\Sp_G$. Finally, the equivalence $L_H$ is explicitly given by the composite 
    \[
     \Mod_G(\mathbb{D}(G/H_+)\otimes R) \xrightarrow{\res_H} \Mod_H(\mathbb{D}(G/H_+)\otimes R)\xrightarrow{-\otimes_{\mathbb{D}(G/H_+)\otimes R}R}\Mod_H(R)
    \]
    see~\cite[Construction 5.23]{MathewNaumannNoel17}, and this is clearly natural in $R$.
\end{proof}

The following corollary establishes the compatibility between our $\infty$-categorical constructions and the tensor-triangular context; in particular, it serves as a gateway for importing results from Balmer's work to our setting. 

\begin{Cor}\label{cor:homotopy-category}
    Let $G$ be a finite group, $H \subseteq G$ a subgroup, and $R$ a commutative $G$-equivariant ring spectrum. Then there is a tt-equivalence
    \[
    \Ho(\Mod_G(\mathbb{D}(G/H_+)\otimes R))\simeq \Mod_{\Ho(\Mod_G(R))}(\mathbb{D}(G/H_+)\otimes R).
    \]
\end{Cor}
\begin{proof}
This is a special case of \cref{rem:hofinite-etale}, but we can also give a direct argument: On the one hand, \cref{lem:res_finiteetale} shows that $\res_H$ is an \'etale tt-functor, which by definition (\cref{def:finite-etale}) means that $\Ho(\res_H)$ is \'etale, with corresponding separable algebra given by $R \otimes \mathbb{D}(G/H_+)$. It follows that there is a tt-equivalence
\[
\Mod_{\Ho(\Mod_G(R))}(\mathbb{D}(G/H_+)\otimes R) \simeq \Ho(\Mod_H(R)).
\]
On the other hand, \cref{prop:base-change-naturality} provides an equivalence
\[
\Ho(L_H)\colon \Ho(\Mod_G(\mathbb{D}(G/H_+)\otimes R)) \xrightarrow{\sim} \Ho(\Mod_H(R))
\]
of tt-categories. Combining these two equivalences yields the claim.
\end{proof}

\subsection{$R$-linear geometric fixed points} 
In this subsection, we discuss geometric fixed point functors and their $R$-linear version for $R\in\CAlg(\Sp_G)$.

\begin{Rec}\label{rec:gfp}
Fix a finite group $G$. We recall the construction of the \emph{geometric fixed point functors} $\Phi^H=\Phi_G^H \colon \Sp_G \to \Sp$ for each subgroup $H \subseteq  G$, with the subscript $G$ usually omitted from notation. First, for $H=G$, we define $\Phi^G$ to be the finite localization of $\Sp_G$ away from $\SET{G/K_+}{K \subsetneq G}$; it is a theorem of Lewis--May--Steinberger--McClure that the target category of this finite localization is equivalent to $\Sp$, see \cite[Corollary  II.9.6]{LewisMaySteinbergerMcClure86} at the level of homotopy categories, or \cite[Theorem 6.11]{MathewNaumannNoel17} for a lift to the level of $\infty$-categories. Moreover, $\Phi^G$ splits the inflation functor $\infl_e\colon \Sp \to \Sp_G$, i.e., $\Phi^G\circ\infl_e$ is naturally isomorphic to the identity functor.

For arbitrary $H \subseteq  G$, we define $\Phi^H$ as the composite functor
\[
\xymatrix{\Phi^H = \Phi_G^H\colon \Sp_G \ar[r]^-{\res_H} & \Sp_H \ar[r]^-{\Phi^H} & \Sp.}
\]
Note that, by \cref{lem:res_finiteetale} and the definition, this is the composition of a finite \'etale extension and a finite localization, and that $\Phi^H$ is symmetric monoidal. 
\end{Rec}

\begin{Def}\label{def:gfp_functor}
Fix a finite group $G$ and let $R\in \CAlg(\Sp_G)$. The functor $\Phi^G \colon \Sp_G \to \Sp$ induces a symmetric monoidal functor on module categories
\[
\Phi^G \colon \Modd{G}(R) \to \Modd{}(\Phi^G R)
\]
which we refer to as the \emph{($R$-linear) $G$-geometric fixed point functor}.
For an arbitrary subgroup $H \subseteq  G$ we refer to the composite
\[
\xymatrix{\Phi^H=\Phi_G^H \colon \Modd{G}(R) \ar[r]^-{\res_H} & \Modd{H}(R) \ar[r]^-{\Phi^H} & \Mod(\Phi^HR)}
\]
as the \emph{($R$-linear)} $H$-\emph{geometric fixed point functor}.
\end{Def}

\begin{Lem}\label{lem:gfp-conjugacy}
Consider an element $g\in G$, a subgroup $H\subseteq G$ and $R\in\CAlg(\Sp_G)$. Then for all $M\in\Mod_G(R)$ we have $\Phi^H M\simeq 0$ if and only if $\Phi^{H^g}M \simeq 0$. 
\end{Lem}
\begin{proof}
  We observe that the claim can be checked in the homotopy category. The case $R=S^0_G$ follows for example from~\cite[Section 2, (L)]{BalmerSanders17}. The general case follows from this one by considering the commutative diagrams
\[
\begin{tikzcd}
    \Sp_G \arrow[r,"\Phi^H"] & \Sp & &  \Sp_G \arrow[r,"\Phi^{H^g}"] & \Sp \\
    \Mod_G(R) \arrow[u,"U_R"] \arrow[r,"\Phi^H"] & \Mod(\Phi^H R) \arrow[u,"U_{\Phi^H R}"'] & & \Mod_G(R) \arrow[u,"U_R"] \arrow[r,"\Phi^{H^g}"] & \Mod(\Phi^{H^g} R) \arrow[u,"U_{\Phi^{H^g} R}"']
\end{tikzcd}
\]
from \cref{rem:group_functoriality} and using the fact that the forgetful functors are conservative.
\end{proof}

\begin{Lem}\label{lem:gfp_as_verdier}
For any $R \in \CAlg(\Sp_G)$ the geometric fixed point functor 
\[
 \Phi^G \colon \Modd{G}(R) \to \Mod(\Phi^GR)
 \]
 is the finite localization with respect to $\SET{R \otimes G/K_+}{K \subsetneq G}$. More generally, for every subgroup $H\subseteq G$, 
 \[
\Phi^H \colon \Mod_G(R) \to \Mod(\Phi^HR) 
 \]
 is the composite of a finite \'etale extension and a finite localization.
\end{Lem}
\begin{proof}
Using \Cref{rec:gfp} the first claim follows from
\cite[Proposition 3.18]{PatchkoriaSandersWimmer20pp}. The second claim follows from the first and \Cref{lem:res_finiteetale}.  
\end{proof}

\begin{Rec}\label{rec:subgroupfamilies}
We recall that a \emph{family} of subgroups of a finite group $G$ is a non-empty collection of subgroups of $G$ that is closed under conjugation and passage to subgroups. A pivotal example arises from a given subgroup $H \subseteq G$ by taking $[\subset H]$ to be the family of subgroups of $G$ which are $G$-conjugate to a proper subgroup of $H$. We note that $[\subset H] \cap H=\mathcal{P}_H$ is the family of proper subgroups of $H$.

Let $\mathcal{F}$ be a family of subgroups of $G$. There is a pointed $G$-CW-complex $\widetilde{E}\mathcal{F}$ which is characterized, up to homotopy equivalence, by the property 
\begin{equation}\label{def-tildeEP}
(\widetilde{E}\mathcal{F})^K\simeq\begin{cases}
    S^0 & \mathrm{if}\;K\not \in \mathcal{F}; \\
    \ast & \mathrm{if}\; K \in \mathcal{F}.
\end{cases}
\end{equation}
By abuse of notation, we also write $\widetilde{E}\mathcal{F}\in \Sp_G$ for the resulting suspension spectrum. 
\end{Rec}

\begin{Rem}\label{rem:explicitgeomfixedpoints}
Specializing \cref{rec:subgroupfamilies} to $\mathcal{F} = [\subset H]$ for a subgroup $H$ in $G$, we see that \cref{lem:gfp_as_verdier} implies that $\Phi^H R \simeq(R \otimes \widetilde{E}[\subset H])^H$ as objects in $\CAlg(\Sp)$. Moreover, there is an equivalence
\[
\Phi^H M \simeq (M \otimes \widetilde{E}[\subset H])^H\in\Mod(\Phi^H R)
\]
for all $M\in\Mod_G(R)$. When $H=G$, we recover the formula $\Phi^G M \simeq (M \otimes \widetilde{E}\mathcal{P}_G)^G$ (see \cite[Corollary II.9.6]{LewisMaySteinbergerMcClure86}). 
We observe that $\widetilde{E}[\subset H]\in \Sp_G$ is naturally an idempotent commutative algebra object.
\end{Rem}

We now prove a variant of \cref{prop:base-change-naturality}.

\begin{Lem}\label{lem:res_finiteetale-gfp}
  Let $R\in \CAlg(\Sp_G)$ and a subgroup $H \subseteq  G$ be given. There is a symmetric monoidal equivalence
 \[
  \Psi_H\colon \Mod_G(\mathbb{D}(G/H_+)\otimes R \otimes \widetilde{E}[\subset H])\xrightarrow{\sim} \Mod(\Phi^H R)
 \]
 which makes the following diagram commute up to isomorphism
 \[
 \begin{tikzcd}
     \Mod_G(R) \arrow[d,"\Phi^H"'] \arrow[dr,"F_{\mathbb{D}(G/H_+)\otimes R\otimes \widetilde{E}[\subset H]}"] &\\
     \Mod(\Phi^H R)  & \Mod_G(\mathbb{D}(G/H_+)\otimes R \otimes \widetilde{E}[\subset H]) \arrow[l,"\Psi_H"',"\sim"].
 \end{tikzcd}
 \]
\end{Lem}
\begin{proof}
 Consider the following diagram
\[\resizebox{\columnwidth}{!}{$\displaystyle
\begin{tikzcd}[ampersand replacement=\&]
\Mod_G(R) \arrow[r,"\res_H"] \arrow[dr,"F_{\mathbb{D}(G/H_+)\otimes R}"'] \& \Mod_H(R) \arrow[r,"F_{\widetilde{E}\mathcal{P}_H}"] \& \Mod_H(R \otimes \widetilde{E} \mathcal{P}_H) \arrow[r,"(-)^H","\sim"']\& \Mod(\Phi^H R)\\
 \& \Mod_G(\mathbb{D}(G/H_+)\otimes R) \arrow[u,"L_H","\sim"'] \arrow[r,"F_{\widetilde{E}[\subset H]}"] \& \Mod_G(\mathbb{D}(G/H_+) \otimes R \otimes \widetilde{E}[\subset H]), \arrow[u,"\widetilde{L}_H"',"\sim"] \& 	
\end{tikzcd}
$}
\]
where the vertical symmetric monoidal equivalences are obtained by applying \cref{prop:base-change-naturality} to $R$ and $R \otimes \widetilde{E}[\subset H]$. The same result also shows that the above diagram commutes up to isomorphism: the left triangle commutes by definition of $L_H$ and the middle square commutes by naturality.
Finally we note that the top horizontal composite is equivalent to $\Phi^H$ by \cref{rem:explicitgeomfixedpoints}. One checks via \cref{lem:eqmodulegenerators} that $\Mod_H(R \otimes \widetilde{E}\mathcal{P}_H)$ is compactly generated by $R \otimes \widetilde{E}\mathcal{P}_H$, so Morita theory (\cite[Theorem 7.1.2.1]{HALurie}) implies that $\Mod_H(R \otimes \widetilde{E}\mathcal{P}_H)\simeq \Mod(\Phi^H R)$ as symmetric monoidal $\infty$-categories. The functor inducing the equivalence can be identified  with the $H$-fixed points functor. Now set $\Psi_H:=(-)^H \circ \widetilde{L}_H$.
\end{proof}

\subsection{Nilpotence and surjectivity}
We recall the concept of nilpotent algebra objects in $\Sp_G$ from \cite{MathewNaumannNoel17} and explain how this interacts with the geometric fixed point functor. As a consequence, we show that, for any $R \in \CAlg(\Sp_G)$, the geometric fixed point functors provide a cover of the Balmer spectrum of $\Perff{G}(R)$ by non-equivariant spectra.

The following definition was given in \cite[Definition 6.36]{MathewNaumannNoel17}, up to identifying $\mathbb{D}(G/H_+)$ with $G/H_+$, which follows because $G$ is finite:

\begin{Def}\label{def:A_F}
Given a finite group $G$ and a family $\mathcal{F}$ of subgroups of $G$, we consider the commutative algebra object 
\[
A_{\mathcal{F}} \coloneqq \prod_{H \in \mathcal{F}}\mathbb{D}(G/H_+) \in \CAlg(\Sp_G).
\]
We say that $M \in\Sp_G$ is  $\mathcal{F}$-\emph{nilpotent} if $M$ is in the thick $\otimes$-ideal of $\Sp_G$ generated by $A_{\mathcal{F}}$. Note that here we allow tensoring with arbitrary objects of $\Sp_G$.
\end{Def}

\begin{Rem}
As noted in \cite[Definition 1.4]{MathewNaumannNoel2019} there is always a minimal family such that $M$ is $\mathcal{F}$-nilpotent. It is called the \emph{derived defect base} of $M$. 
\end{Rem}
\begin{Not}\label{ex:borel_g_spectra}
Given $R \in \CAlg(\Sp)$, we let $\underline{R}_G$ denote the \emph{Borel completion} of $R$ in $G$. By definition, $\underline{R}_G$ is the image of $R$ under the composite of functors
\[
\Sp \xrightarrow{\infl_e} \Sp_G \xrightarrow{B} (\Sp_G)_{\mathrm{Borel}}\subseteq\Sp_G,
\]
where:
\begin{itemize}
\item $\infl_e$ is the inflation functor from \cref{rec:infl-res};
\item $B$ is the Borel completion functor constructed as the Bousfield localization of $\Sp_G$ with respect to $G_+$, given by $\iHom(EG_+,-)$;
 \item $(\Sp_G)_{\mathrm{Borel}}$ is the full subcategory of $\Sp_G$ spanned by the Borel-equivariant $G$-spectra, i.e., the essential image of $B$, see~\cite[Definition 6.14]{MathewNaumannNoel17}.
\end{itemize}
The functor $B$ is symmetric monoidal by general facts about localizations, see the discussion around \cite[Remark 2.20]{MathewNaumannNoel17}; alternatively, this can be constructed directly by using the diagonal of $EG$. The composite of $B$ with the right adjoint inclusion $(\Sp_G)_{\mathrm{Borel}} \subseteq \Sp_G$ is then lax symmetric monoidal. Therefore, the composite functor $\Sp \to \Sp_G$ is lax symmetric monoidal and hence preserves commutative algebra objects. It follows that $\underline{R}_G\in\CAlg(\Sp_G)$. We will also refer to $\underline{R}_G$ as \emph{Borel-equivariant} $R$-\emph{theory}. Oftentimes, when the group is clear from context, we will omit the subscript and simply write $\underline{R}$ for $\underline{R}_G$.
 
We note that $\res_H \underline{R}_G \simeq \underline{R}_H$ as an easy consequence of \cite[Proposition 6.16]{MathewNaumannNoel17}. We will use this as an identification throughout the document. 
\end{Not}
\begin{Exa}\label{ex:E-nilpotent}
If $M = S_G^0$ is the $G$-sphere spectrum, then $M$ has derived defect base the family of all subgroups \cite[Proposition 4.22]{MathewNaumannNoel2019}. If $M = \underline{E} = \underline{E}_G$ is Borel-equivariant Lubin--Tate $E$-theory of height $n$ at a prime $p$, then $M$ has derived defect base the family of abelian $p$-subgroups which can be generated by $n$ elements \cite[Proposition 5.26]{MathewNaumannNoel2019}. 
\end{Exa}

\begin{Rem}\label{rem:sep-alg}
The commutative algebra $R \otimes A_{\mathcal{F}}\in \CAlg(\Modd{G}(R))$ is separable of finite degree. This follows from \Cref{lem:res_finiteetale}. Note that this holds more generally if $\mathcal{F}$ is just a collection of subgroups of $G$ rather than a family.
\end{Rem}

The following result, which is due to Mathew, Naumann, and Noel,  relates the concept of $\mathcal{F}$-nilpotence applied to ring spectra with the vanishing of geometric fixed points:

\begin{Lem}\label{lem:gfp_vanishing}
Let $R \in \CAlg(\Sp_G)$, then $R$ is $\mathcal{F}$-nilpotent if and only if $\Phi^HR = 0$ for all $H \not \in \mathcal{F}$. Moreover, if $R$ is $\mathcal{F}$-nilpotent and $M \in \Modd{G}(R)$, then $\Phi^HM = 0$ if $H \not \in \mathcal{F}$. 
\end{Lem}
\begin{proof}
The statement about $R$ is \cite[Theorem 6.41]{MathewNaumannNoel17}. The second part follows because $\Phi^H M$ is a module over $\Phi^H R = 0$. 
\end{proof}

\begin{Rem}
    The next result is a generalization of the classical fact that both the categorical fixed points and the geometric fixed points detect equivalences in $\Sp_G$. This is usually proven via isotropy separation, and it is implicitly used in the proof of the proposition below.
\end{Rem}

\begin{Prop}\label{prop:equivalent-R-nilpotent}
Let $G$ be a finite group, $R\in \CAlg(\Sp_G)$ and $\mathcal{F}$ a family of subgroups of $G$. Then the following are equivalent:
\begin{itemize}
    \item[(a)] $\Modd{G}(R)$ is compactly generated by $R \otimes A_{\mathcal{F}}$;
    \item[(b)] $R$ is $\mathcal{F}$-nilpotent;
    \item[(c)] the geometric fixed point functor
    \[
    \Phi^{\mathcal{F}}=(\Phi^H)_{H\in\mathcal{F}}\colon \Modd{G}(R) \to \prod_{H\in\mathcal{F}}\Mod(\Phi^H R)
    \]
    is conservative;
    \item[(d)] the fixed point functor
    \[
    (-)^{\mathcal{F}}=((-)^H)_{ H\in\mathcal{F}}\colon \Modd{G}(R) \to \prod_{H\in\mathcal{F}}\Mod(R^H)
    \]
    is conservative;
    \item[(e)] the restriction functor 
    \[
    \res_{\mathcal{F}}=(\res_H)_{H\in\mathcal{F}}\colon \Modd{G}(R) \to \prod_{H\in\mathcal{F}} \Modd{H}(R)
    \]
    is conservative.
\end{itemize}
\end{Prop}
\begin{proof}
To avoid confusion, throughout this proof we will write $\Phi^H_R$ and $(-)^H_R$ for the $R$-linear version of these functors. We now prove all implications in turn.

$(a) \implies (b) \colon $ Assuming part $(a)$ we know that $R \in \Loc\langle R \otimes A_{\mathcal{F}}\rangle\subseteq \Modd{G}(R)$. Since $R$ and $R \otimes A_{\mathcal{F}}$ are compact objects of $\Modd{G}(R)$, we can apply~\cite[Theorem 2.1]{Neeman96} and deduce that \begin{equation}\label{eq:inclusion_of_r}
 R \in \Loc\langle R \otimes A_{\mathcal{F}}\rangle\cap \Perff{G}(R)=\thick\langle R \otimes A_{\mathcal{F}}\rangle\subseteq \Modd{G}(R).   
\end{equation}
Applying the forgetful functor $\Modd{G}(R) \to \Sp_G$ to \eqref{eq:inclusion_of_r} and keeping in mind \cref{rec:thickstuff}, we see that there are inclusions (computed in $\Sp_G$)
\[
R\in\thick\langle R\otimes A_{\mathcal{F}}\rangle\subseteq \thickt{A_{\mathcal{F}}},
\]
so $(b)$ holds.

$(b) \implies (c) \colon$ 
Suppose now that $(b)$ holds and let us prove $(c)$. When $R = S_G^0$, then $\mathcal{F}$ is the family of all subgroups, and the claim is classical, see, for example, \cite[Proposition 6.13]{MathewNaumannNoel17}. The general case follows from this case and the observation that in the commutative square (\Cref{rem:group_functoriality})
\begin{equation}\label{diagram-fixed-points}
\begin{tikzcd}
 \Sp_G \arrow[r,"\Phi^H"] & \Sp  \\
 \Modd{G}(R) \arrow[u, "U_R"] \arrow[r,"\Phi^H_R"'] & \Mod(\Phi^HR) \arrow[u,"U_{\Phi^HR}"']
\end{tikzcd}
\end{equation}
the forgetful functors are conservative \cite[Corollary 4.2.3.2.]{HALurie}. 
Indeed, suppose $M \in \Modd{G}(R)$ has $\Phi^{\calF}(M) = 0$. We want to show that $M\simeq 0$. By the conservativity of $U_R$ and the previous case, it suffices to show that $\Phi^H(U_RM) = 0$ for all $H \subseteq  G$. By commutativity of the diagram above, we have $\Phi^H(U_RM)=U_{\Phi^H R}\Phi^H_R(M)$. If $H \in \mathcal{F}$, this is zero by assumption, and if $H \not \in \mathcal{F}$ then this is zero by \Cref{lem:gfp_vanishing}. 

$(c) \implies (d)\colon$ Suppose that $(c)$ holds and consider $M\in\Modd{G}(R)$ such that $ 0\simeq M^H \in \Mod(R^H)$ for all $H\in \mathcal{F}$. By \Cref{rem:group_functoriality} there is a commutative square 
\[
\begin{tikzcd}
 \Sp_G \arrow[r,"(-)^H"] & \Sp \\
 \Modd{G}(R) \arrow[u,"U_R"] \arrow[r,"(-)^H_R"'] & \Modd{}(R^H) \arrow[u,"U_{R^H}"'].
\end{tikzcd}
\]
It follows that 
\[
0\simeq U_{R^H}((M)_R^H) \simeq(U_RM)^H
\]
for all $H\in \mathcal{F}$. Since $\mathcal{F}$ is closed under subgroups, we also have $(U_RM)^K\simeq 0$ for all subgroups $K\subseteq H$. Therefore  $\res_H(U_RM)\simeq 0$ for all $H\in \mathcal{F}$. We deduce that $\Phi^H U_RM \simeq 0$ for all $H\in\mathcal{F}$. The commutativity of the diagram \eqref{diagram-fixed-points} tells us that 
\[
0\simeq \Phi^H U_RM \simeq U_{\Phi^HR} \Phi_R^H M
\]
for all $H\in\mathcal{F}$. Since $(c)$ holds by assumption, conservativity of $U_{\Phi^H R}$ then implies that $M\simeq 0$, which proves $(d)$. 

$(d) \implies (e) \colon $ Similarly $(d)$ implies $(e)$ since $\res_{\mathcal{F}}(M)\simeq 0$ in particular implies that $0\simeq M^H \in \Mod(R^H)$ for all $H\in\mathcal{F}$. 

$(e) \implies (a) \colon$ Consider $M\in\Modd{G}(R)$ such that $\Hom_{\Modd{G}(R)}(R \otimes A_{\mathcal{F}}, M)\simeq 0$; we are required to show that $M\simeq 0$, compare with~\cite[Lemma 2.2.1]{SchwedeShipley03}. By adjunction, there are equivalences
\[
0\simeq \Hom_{\Modd{G}(R)}(R \otimes A_{\mathcal{F}}, M)\simeq \Hom_{\Sp_G}(A_{\mathcal{F}}, U_RM)\simeq \bigoplus_{H\in \mathcal{F}} (U_RM)^H.
\]
Thus $(U_RM)^H\simeq 0$ for all $H\in \mathcal{F}$.
Since $\mathcal{F}$ is closed under subgroups, then $\res_{H} U_RM\simeq 0$ for all $H\in \mathcal{F}$. 
Therefore, the commutativity of the following diagram (\Cref{rem:group_functoriality}) 
\[
\begin{tikzcd}
\Modd{G}(R) \arrow[r,"\res_H"] \arrow[d,"U_R"'] & \Modd{H}(R) \arrow[d,"U_R"] \\
\Sp_G \arrow[r,"\res_H"'] &  \Sp_H
\end{tikzcd}
\]
and conservativity of the forgetful functor $U_R$, implies that $0\simeq \res_H M \in \Modd{H}(R)$. Hence 
by part $(e)$, we conclude $M\simeq 0$ as required.
\end{proof}

\begin{Thm}\label{thm:eqnilpotence}
Let $G$ be a finite group, $R \in \CAlg(\Sp_G)$, and $\mathcal{F}$ a family of subgroups of $G$ such that $R$ is $\mathcal{F}$-nilpotent. Then the geometric fixed point functors $(\Phi^H)_{H\in\mathcal{F}}$ detect $\otimes$-nilpotence of morphisms with dualizable source in $\Mod_G(R)$, in the sense of \cref{def:jointlynilfaithful}. In fact, it suffices to choose one representative $(H)$ for each $G$-conjugacy class of subgroup in $\mathcal{F}$ and use the collection $(\Phi^H)_{(H)}$.
\end{Thm}
\begin{proof}
The functor $\Phi^\mathcal{F} = (\Phi^H)_{H\in\mathcal{F}}$ is a coproduct-preserving tt-functor as it is a finite product of such functors. It is also conservative by~\Cref{prop:equivalent-R-nilpotent}(c); in fact, \cref{lem:gfp-conjugacy} implies that it is enough to take one representative from each $G$-conjugacy class of subgroups in $\mathcal{F}$. Therefore, the claim follows from \cref{thm:abstractnilpotence}. 
\end{proof}

\begin{Cor}\label{cor:eqspc_surjectivity}
With notation as in \cref{thm:eqnilpotence}, the geometric fixed point functor
\[
\xymatrix{\Phi^\mathcal{F} \colon \Modd{G}(R) \ar[r] & \prod\limits_{(H) \in \mathcal{F}}\Mod(\Phi^HR)}
\]
induces a surjective map of topological spaces:
\[
\xymatrix{\Spc(\Phi_R^\mathcal{F})\colon \bigsqcup\limits_{(H) \in \mathcal{F}}\Spc(\Perf(\Phi^HR)) \ar@{->>}[r] & \Spc(\Perff{G}(R)).}
\]
Here, the product is taken over a set of representatives of conjugacy classes of subgroups contained in $\mathcal{F}$. 
\end{Cor}
\begin{proof}
The result is a direct consequence of \cref{cor:surjective-map}, which applies by the same argument as in the proof of \cref{thm:eqnilpotence}.
\end{proof}

\begin{Rem}\label{rem:primecollasion}
For a general $R\in\CAlg(\Sp_G)$, the map $\Spc(\Phi_R^{\mathcal{F}})$ will not be bijective, see~\cref{ex:gl-hyp-needed} or~\cref{rem:res-not-injective} together with \cref{lem:ttveplus} for some counterexamples. However, in some favourable situations the map is a bijection. This is the case if $R$ is an equivariant ring spectrum with trivial $G$-action by~\cite[Theorem~13.11]{bhs1} or by \cref{prop:infl-guys} below. We will prove later in \cref{lem:generalspc} that the map is a homeomorphism for $R=\underline{E}_G$, the Borel-completion of a Lubin--Tate $E$-theory of any height at the prime $p$ and any finite abelian $p$-group $G$.
\end{Rem}

\begin{Cor}\label{cor:noetherian-spectrum}
Let $G$ be a finite group and let $R\in \CAlg(\Sp_G)$. Let $\mathcal F$ be a family of subgroups of $G$ for which $R$ is $\mathcal F$-nilpotent. Suppose that for each $H \in \mathcal{F}$ the spectrum $\Spc(\Perf(\Phi^HR))$ is a Noetherian space. Then so is $\Spc(\Perf_G(R))$.  
\end{Cor}
\begin{proof}
   \Cref{cor:eqspc_surjectivity} implies that $\Spc(\Perff{G}(R))$ is covered by the images of finitely many continuous maps with Noetherian source. Therefore, $\Spc(\Perff{G}(R))$ is itself Noetherian.
\end{proof}

\subsection{Stratification for equivariant ring spectra}
In this section we show how to descend stratification for equivariant module categories along geometric fixed point functors. This relies on a version of \'etale descent which we now recall. 
\begin{Rec}\label{rec:abstractstratdescent}
Let $f^*\colon \cat S \to \cat T$ be a finite \'etale functor between rigidly-compactly generated tt-categories with Noetherian spectrum. If $\cat T$ is stratified, then the localizing ideals $\Gamma_{\cat P}\cat S$ are minimal for each $\cat P \in \Img(\Spc(f^*))$. This follows from the discussion of \cite[\S2.2.2]{Barthel2021pre}, in particular the proof of Lemma 2.21. 
\end{Rec}

We continue to use the notation of the previous subsection. 

\begin{Thm}\label{thm:eqstratdescent}
Let $G$ be a finite group and let $R\in \CAlg(\Sp_G)$. Let $\mathcal F$ be a family of subgroups $G$ for which $R$ is $\mathcal F$-nilpotent. Suppose that the following conditions hold for all subgroups $H \in \mathcal F$:
    \begin{enumerate}
        \item $\Spc(\Perf(\Phi^HR))$ is a Noetherian space;
        \item $\Mod(\Phi^HR)$ is stratified.
    \end{enumerate}
Then $\Modd{G}(R)$ is stratified with Noetherian spectrum $\Spc(\Perff{G}(R))$.
\end{Thm}
\begin{proof}
For a given finite group $G$, first observe that it suffices to prove the statement for the family of all subgroups of $G$: Indeed, if $R$ is $\mathcal F$-nilpotent for some family $\mathcal F$ of subgroups of $G$, then both $(a)$ and $(b)$ hold trivially for all subgroups of $G$ not in $\mathcal F$. Therefore, we might as well take $\mathcal F$ to be the family of all subgroups of $G$.

The statement that $\Spc(\Perff{G}(R))$ is Noetherian is \Cref{cor:noetherian-spectrum}. By \Cref{rem:stratification_conditions} it therefore remains to verify the minimality of $\Gamma_{\cat P}\Modd{G}(R)$ at all prime ideals $\cat P \in \Spc(\Perff{G}(R))$. We will argue by induction on the order of the group $G$. The base of the induction, i.e., the case that $G=e$, holds by Hypothesis $(b)$. For the induction step, we introduce some auxiliary notation. Let $\varphi = \Spc(\prod_{H \subseteq G}\Phi_R^H)$ and write $\varphi^H$ for its $H$-component. Moreover, we assemble the restriction maps for all proper subgroups of $G$ into a tt-functor
\[
\xymatrix{\res_{\subsetneq G}\colon \Modd{G}(R) \ar[r] & \prod_{H \subsetneq G}\Modd{H}(R),}
\]
and denote the induced map on spectra by $\psi$. 

Consider a prime ideal $\cat P \in \Spc(\Perff{G}(R))$. We need to show that $\Gamma_{\cat P}\Modd{G}(R)$ is a minimal localizing ideal in $\Modd{G}(R)$. By \cref{cor:eqspc_surjectivity}, $\cat P$ is in the image of $\varphi$. We distinguish two cases. First, suppose that $\cat P$ is in the image of $\varphi^G$. In this case, $\Phi^G$ is a finite localization, so that Zariski descent in the form of \cite[Remark 5.4]{bhs1}, that is descent of minimality along a finite localization functor, implies our minimality claim at $\cat P$. 

Now, suppose that $\cat P$ is not in the image of $\varphi^G$; in other words, assume that $\cat P$ is in the image of $\varphi^H$ for some proper subgroup $H \subseteq G$. Since this map factors through $\psi$, we have $\cat P \in \im(\psi)$. By \cref{lem:res_finiteetale}, the functor $\res_{\subsetneq G}$ is finite \'etale. \Cref{rec:abstractstratdescent} then applies to show that our minimality claim holds at $\cat P$ if $\prod_{H \subsetneq G}\Modd{H}(R)$ is stratified. We can thus deduce the required minimality from our induction hypothesis, thereby finishing the proof.
\end{proof}

\begin{Rem}\label{rem:stmodbikstratification}
The prototypical example of an equivariant ring spectrum $R$ for which $\Modd{G}(R)$ is stratified is the Borel-completion of a field $k$ of characteristic $p$ dividing the order of $G$.  In this case, there is a  tt-equivalence $\Perff{G}(\underline{k})\simeq \mathrm{D}^b(kG)$, the bounded derived category of finitely generated $kG$-modules equipped with the symmetric monoidal structure coming from the coproduct of $kG$. This equivalence extends to a symmetric monoidal equivalence 
\begin{equation}\label{eq:heightinfty}
    \Modd{G}(\underline{k})\simeq \mathrm{K}(\mathrm{Inj}\,kG),
\end{equation}
where the right hand side denotes the homotopy category of unbounded chain complexes of injective $kG$-modules as in \cite{BensonKrause2008}. This category admits the \emph{stable module category} $\StMod(kG)$ of $kG$ as a finite localization away from the localizing subcategory generated by $kG$. Therefore, the equivalence of \eqref{eq:heightinfty} passes to the quotients:
\[
\Modd{G}(\underline{k})/\Loc\langle\underline{k}\otimes G_+\rangle \simeq \StMod(kG).
\]
Both of these equivalences rely on generation by permutation modules, see also \cref{rem:genpermmodules}. The corresponding stratification theorem is then due to Benson, Iyengar, and Krause \cite{BensonIyengarKrause11a}; in fact, their main theorem gives the cohomological stratification of $\Modd{G}(\underline{k})$ in the sense of \cref{rem:first_bik_strat}, from which they deduce the stratification of $\StMod(kG)$ over
\[
\Spc(\StMod(kG)^c) \cong \Proj H^{\bullet}(G,k).
\]
More generally, the analogous results hold with coefficients in any regular commutative ring in place of $k$, see \cite{Barthel2021pre,Barthel2022pre,BIKP2022pre}, so in particular over mixed characteristic discrete valuation rings. In \cref{thm:etheorycohomstratification}, we establish a chromatic analogue of their work. 
\end{Rem}

\section{Strong Quillen stratification for equivariant Balmer spectra} \label{section:strong-quillen}
The goal of this section is to obtain a decomposition of the Balmer spectrum of $\Perf_G(R)$ for $R \in \CAlg(\Sp_G)$ in terms of the Balmer spectra of the {\em non-equivariant} categories $\Perf(\Phi^H(R))$ for $H\in\mathcal F$ together with their Weyl-group actions, see \cref{thm:quillen_decomposition}.  

\subsection{The decomposition result}

The main result of this section will compare to the stratification theorem of Quillen \cite[Stratification Theorem 10.2]{Quillen71} for the spectrum of the mod $p$ cohomology of a finite group, so we follow his lead
in the notation.

\begin{Not}\label{not:quillen_strata}
 Let $H \subseteq  G$ be a subgroup and $R \in \CAlg(\Sp_G)$. 
\begin{itemize}
    \item[(a)]
    We define a functor $\mathcal{V}(R,-)\colon \mathcal{O}(G) \to \mathrm{Top}$  from the orbit category of $G$ to the category of topological spaces by the assignment
    \[
     G/H \mapsto \mathcal{V}(R,H) \coloneqq \Spc(\Perf_G(\mathbb{D}(G/H_+)\otimes R)).
    \]
     The map of $G$-sets $G/H \to G/G$ induces a map $\psi_H\colon \mathcal{V}(R,H)\to \mathcal{V}(R,G)$ on Balmer spectra. \Cref{prop:base-change-naturality} identifies $\mathcal{V}(R,H)$ with $\Spc(\Perf_H(R))$ and $\psi_H$ with $\Spc(\res_H)$.\footnote{While it is true that the functor $\Perf_G(\mathbb{D}((-)_+)\otimes R)$ refines to a category-valued Mackey functor using the bi-adjoint of restriction as transfer maps, since the latter functors are only lax symmetric monoidal, they do not induce maps on Balmer spectra and are therefore not used in this paper.} 
 \item[(b)] The map $S^0_G \to \widetilde{E}[\subset H]$ is the unit map of the finite localization $-\otimes \widetilde{E}[\subset H]$. Consequently, we can identify $\Spc(\Perf_G(\mathbb{D}(G/H_+)\otimes R \otimes \widetilde{E}[\subset H]))$ with an open subset of $\mathcal{V}(R,H)$, which we denote as $\mathcal{V}^+(R,H)\subseteq \mathcal{V}(R,H)$. Note that $\mathcal{V}^+(R,H)$ identifies with 
 $\Spc(\Perf(\Phi^H R))$ by~\cref{lem:res_finiteetale-gfp}.
 \item[(c)] We denote by ${\mathcal V}_G(R,H)$ the image of ${\mathcal V}(R,H)$ under the map $\psi_H \colon {\mathcal V}(R,H)\to {\mathcal V}(R,G)$, and we endow ${\mathcal V}_G(R,H)\subseteq {\mathcal V}(R,G)$ with the subspace topology. We observe that by~\cite[Theorem 3.4(b)]{Balmer16b} the map $\psi_H$ is closed (but not in general a closed immersion), and in particular that ${\mathcal V}_G(R,H)\subseteq {\mathcal V}(R,G)$ is closed (and equal to the support of $R\otimes G/H_+$). 
\item[(d)] We write ${\mathcal V}_G^+(R,H)$ for the image of $\mathcal{V}^+(R,H)$ in $\mathcal{V}(R,G)$ under $\psi_H$. We observe that $\mathcal{V}^+_G(R,H)\subseteq \mathcal{V}_G(R,H)$ is open as its complement can be identified with the support of $R \otimes A_{[\subset H]}$, for $A_{[\subset H]}$ as in \cref{def:A_F}.
 \end{itemize}
\end{Not}

\begin{Cons}\label{con-fund-action-rings}
 Given a subgroup $H \subseteq  G$, we let $W_G(H)=N_G(H)/H$ denote the Weyl group of $H \subseteq  G$, that is, the automorphism group of $G/H$ in $\mathrm{FinSet}_G\op$. Let $R \in \CAlg(\Sp_G)$ as before and consider $R \otimes \widetilde{E}[\subset H]\in \CAlg(\Sp_G)$. There is an action of $W_G(H)$ on $\mathbb{D}(G/H_+)\otimes R$ by functoriality. If we replace $R$ with $R \otimes \widetilde{E}[\subset H]$ in the previous construction, we obtain a $W_G(H)$-action on $\mathbb{D}(G/H_+)\otimes R \otimes \widetilde{E}[\subset H]$. We emphasize that in both cases the Weyl group acts through $\mathbb{D}(G/H_+)$. The canonical map $S^0_G \to \widetilde{E}[\subset H ]$ induces a ring map $\mathbb{D}(G/H_+)\otimes R \to \mathbb{D}(G/H_+) \otimes R \otimes \widetilde{E}[\subset H]$ which is $W_G(H)$-equivariant. By functoriality we then get induced $W_G(H)$-actions on the corresponding Balmer spectra which make $\mathcal{V}^+(R,H)\subseteq \mathcal{V}(R,H)$ into a $W_G(H)$-equivariant map. Throughout this section we will always consider the Balmer spectra $\mathcal{V}(R,H)$ and $\mathcal{V}^+(R,H)$ together with this Weyl group action. By the aforementioned functoriality of $\mathcal{V}(R,-)$, these Weyl group actions are compatible with the maps $\psi_H$.
\end{Cons}

Our main decomposition result for the Balmer spectrum is stated in the next theorem. It generalizes Balmer's tt-theoretic Quillen stratification theorem \cite[Theorem 1.6]{Balmer16a} from the $\Fp$-linear setting to arbitrary equivariant tt-categories and extends the `weak' decomposition (in the form of part $(a)$ below) to its `strong' form (parts $(b)$ and $(c)$). We remark that it's the latter that will be crucial for the proof of our stratification result (\cref{thm:etheorycohomstratification}) for Lubin--Tate theory.

\begin{Thm}\label{thm:quillen_decomposition}
Assume $G$ is a finite group, $R\in\CAlg(\Sp_G)$,
and $\mathcal F$ is a family of subgroups of $G$ such that $R$ is $\mathcal F$-nilpotent. Then:
\begin{itemize}
\item[(a)] The restriction maps induce a homeomorphism
\[ 
\xymatrix{{{\mathop{\colim}\limits_{G/H\in \mathcal{O}_{\mathcal{F}}(G)} {\mathcal V}(R,H) }} \ar[r]^-{\cong} &{\mathcal V}(R,G)},
\]
where ${\mathcal O}_{\mathcal F}(G)$ denotes the category of non-empty, transitive $G$-sets with isotropy in the family $\mathcal F$.
\item[(b)] There is a decomposition into locally closed disjoint subsets
\[ {\mathcal V}(R,G)=\bigsqcup_{(H)\in\mathcal F}{\mathcal V}_G^+(R,H),\]
where the index runs over a set of representatives of conjugacy classes of subgroups in $\mathcal F$.
\item[(c)] For every $H\in\mathcal F$, restriction induces a homeomorphism
\[ \xymatrix{{\mathcal V}^+(R,H)/W_G(H)  \ar[r]^-{\cong} & {\mathcal V}^+_G(R,H).}\]
\end{itemize}
\end{Thm}

We first deduce a corollary which shows that the condition of the comparison map of \Cref{rem:comparison_map} being a homeomorphism descends in the present situation.

\begin{Cor}\label{cor:quillen_stratification}
Under the hypothesis of \cref{thm:quillen_decomposition}, assume in addition that
\begin{enumerate}
    \item[(1)] for every $H\in\mathcal F$, the comparison map \[\xymatrix{\rho_0^H\colon \mathcal V(R,H)    \ar[r]^-{\cong}   &\Spec(\pi_0^G(\mathbb{D}(G/H_+)\otimes R))\cong \Spec(\pi_0^H(R))}\] is a homeomorphism;
    \item[(2)] the ring $\pi_0^G(R)$ is Noetherian, and the restriction morphism $\pi_0^G(R)\to \pi_0^H(R)$ makes $\pi_0^H(R)$ into a finitely generated $\pi_0^G(R)$-module for every $H\in{\mathcal F}$.
\end{enumerate}
Then:
\begin{itemize}
    \item[(a)] There is a commutative diagram of homeomorphisms
    \[\begin{tikzcd}[ampersand replacement=\&]
	{{\mathop{\colim}\limits_{G/H\in \mathcal{O}_{\mathcal{F}}(G)} {\mathcal V}(R,H) }}\& {\mathcal V}(R,G)\\
{{\mathop{\colim}\limits_{G/H\in \mathcal{O}_{\mathcal{F}}(G)} \Spec(\pi_0^H(R))}}\&  {\Spec(\pi_0^G(R)).}
	\arrow["{\rho_0^G}", "\cong"' ,from=1-2, to=2-2]
	\arrow[from=1-1, to=1-2, "\cong" ]
	\arrow["{\colim \rho_0^H}"', "\cong",from=1-1, to=2-1]
	\arrow[ "\cong", from=2-1, to=2-2]
\end{tikzcd}\]
\item[(b)] There is a decomposition into a disjoint union of locally closed subsets
\[ \Spec(\pi_0^G(R)) = \bigsqcup_{(H)\in\mathcal F} V_H/W_G(H),\]
\noindent
where $V_H\subseteq\Spec(\pi_0^H(R))$ denotes the open subset complementary to the vanishing locus of the ideal $\left\{ x\in\pi_0^H(R)\, \mid\, \mathrm{Res}^H_{K}(x)=0\, \forall\, K\subsetneq H\right\}\subseteq \pi_0^H(R)$.
\end{itemize}
\end{Cor}

\begin{proof}
\leavevmode
\begin{itemize}
    \item[$(a)$]  The diagram commutes by the naturality of $\rho_0^{(-)}$.  The upper
    horizontal map is a homeomorphism by \cref{thm:quillen_decomposition} (a). The left vertical map is a homeomorphism by our assumption and the lower horizontal map is a homeomorphism by
    \cite[Theorem 3.26]{MathewNaumannNoel17}. The final reference requires the Noetherian and finiteness assumptions in the hypothesis.
    \item[$(b)$] The statement follows from \Cref{thm:quillen_decomposition} $(b)$ and $(c)$ after observing that the homeomorphism $\rho^H_0$ identifies the open subset ${\mathcal V}^+(R,H)\subseteq {\mathcal V}(R,H)$ with the open subset $V_H\subseteq \Spec(\pi_0^H(R))$. In more detail, we will verify that the respective complements of ${\mathcal V}^+(R,H)$ and $V_H$ identify under $\rho_0^H$. To this end, our assumption combined with the naturality of $\rho_0^{(-)}$ supplies a commutative square
        \[
            \begin{tikzcd}[ampersand replacement=\&,column sep=1.8cm]
	{\bigsqcup_{K\subsetneq H}\mathcal{V}(R,K)}\& \mathcal{V}(R,H)\\
{\bigsqcup_{K\subsetneq H} \Spec(\pi_0^K(R))}\&  {\Spec(\pi_0^H(R)),}
	\arrow["{\rho_0^H}", "\cong"' ,from=1-2, to=2-2]
	\arrow[from=1-1, to=1-2, "(\psi_K)" ]
	\arrow["{\sqcup\rho^K_0}"', "\cong",from=1-1, to=2-1]
	\arrow[ "(\Spec(\mathrm{Res}^H_{K}))", from=2-1, to=2-2]
            \end{tikzcd}
        \]
    in which the horizontal maps are induced by the inclusions $K \subsetneq H$. Recall from \cref{not:quillen_strata} that the complement of ${\mathcal V}^+(R,H)$ in ${\mathcal V}(R,H)$ is given by $\supp(R \otimes \prod_{K \subsetneq H}\mathbb{D}(H/K_+))$, which in turns equals $\bigcup_{K \subsetneq H}\mathcal{V}_H(R,K) = \bigcup_{K \subsetneq H}\mathrm{im}(\psi_K)$. We thus obtain an identification 
        \begin{align*} 
           \rho_0^H({\mathcal V}(R,H)\setminus {\mathcal V}^+(R,H)) & =  \bigcup_{K \subsetneq H}\mathrm{im}(\Spec(\mathrm{Res}^H_{K}))\\
           & = \mathrm{im}(\Spec(\prod_{K \subsetneq H}\mathrm{Res}^H_{K})).
        \end{align*}
    Since these subsets are closed, the latter image identifies with the vanishing locus of $\ker(\prod_{K \subsetneq H}\mathrm{Res}^H_{K})$, as desired.\qedhere
\end{itemize}
\end{proof}

\begin{Rem}
If $R$ is Borel-complete and complex orientable, then one can hope to identify $V_H$ with
$\Spec(\pi_0(\Phi^HR))$,
see for instance~\eqref{eq:disjoint_union_set_spec}. 
On the other hand, for $R=H\underline{\mathbb Z}$
and $G=C_2$, the canonical ring map 
$\pi_0(R^G)\to \pi_0(\Phi^G(R))$ identifies with the unique map
$\mathbb Z\to\mathbb Z/2$ (compare \cref{ex:constantgreen}), while $V_G=\emptyset$ as the restriction map is the identity in this case.
\end{Rem}

The rest of this subsection is devoted to the proof of  \Cref{thm:quillen_decomposition}. An important ingredient is Balmer's descent
for separable algebras, which we now review.

\begin{Rem}\label{rem:transitivegsets}
For subgroups $H,K \subseteq G$, there is a decomposition of the product of transitive $G$-sets into transitive $G$-sets as follows:
\begin{equation}\label{eq:transitivegsets}
    \beta \coloneqq \bigsqcup \beta_g \colon \bigsqcup_{[g] \in H\backslash G/K}G/(H^g \cap K) \xrightarrow{\cong} G/H \times G/K  ,
\end{equation}
where the coproduct on the left side runs through a set of double coset representatives and $\beta_g$ sends $[x]$ to $([xg^{-1}],[x])$. 
\end{Rem}

\begin{Thm}\label{Thm:coequalizer_collection}
Let $\cat S$ be a set of subgroups of $G$, and set
$
A_{\cat S} \coloneqq \prod_{H \in \cat S} \mathbb{D}(G/H_+) \in \CAlg(\Sp_G)$. There is a coequalizer diagram of topological spaces
\[
\begin{tikzcd}[column sep=small]
	\bigsqcup\limits_{\substack{H,K \in \cat S \\ [g] \in H\backslash G/K}}\mathcal{V}(R,H^g \cap K) & {\bigsqcup\limits_{H \in \cat S} \mathcal{V}(R,H)} & \supp(R \otimes A_{\cat S})\subseteq \mathcal{V}(R,G),
	\arrow["{\psi_1}", shift left=1, from=1-1, to=1-2]
	\arrow["{\psi_2}"', shift right=1, from=1-1, to=1-2]
	\arrow["{\psi_{\cat S}}", from=1-2, to=1-3]
\end{tikzcd}
\]
where the map $\psi_{\cat S}$ is induced by the extension of scalars functor, and the maps $\psi_1$ and $\psi_2$ are induced by the maps of $G$-sets
\[
\alpha_g \colon G/(H^g \cap K) \xrightarrow{\beta_g} G/H \times G/K \xrightarrow{\pi_1} G/H
\]
and
\[
\alpha_1 \colon G/(H^g \cap K) \xrightarrow{\beta_g} G/H \times G/K \xrightarrow{\pi_2} G/K,
\]
see notation from \Cref{rem:transitivegsets}.
\end{Thm}
\begin{proof}
This is an adaptation of the proof of \cite[Theorem 4.10]{Balmer16b}. Set $\cat C = \Ho(\Perff{G}(R))$. We first note that the commutative algebra $R \otimes A_{\cat S}$ is separable of finite degree by \cref{rem:sep-alg}. Therefore by~\cite[Theorem 3.14]{Balmer16b}, there is a coequalizer diagram of topological spaces
\[\begin{tikzcd}
	{\Spc(\Mod_{\cat C}(R \otimes A_{\cat S}^{\otimes 2}))} & {\Spc(\Mod_{\cat C}(R \otimes A_{\cat S}))} & {\supp(R \otimes A_{\cat S}),}
	\arrow["{\phi_1}", shift left=1, from=1-1, to=1-2]
	\arrow["{\phi_2}"', shift right=1, from=1-1, to=1-2]
	\arrow["{\phi_{\cat S}}", from=1-2, to=1-3]
\end{tikzcd}\]
where the maps $\phi_1$ and $\phi_2$ are induced by left and right unit maps, and $\phi_{\cat S}$ is induced by the extension of scalars functor. Note that $\mathbb{D}(G/H) \otimes \mathbb{D}(G/K) \simeq \mathbb{D}(G/H \times G/K)$ and that the maps $\phi_1$ and $\phi_2$ are induced by the projection maps of $G$-sets 
\[
G/H \times G/K \xrightarrow{\pi_1} G/H \qquad \text{ and } \qquad G/H \times G/K \xrightarrow{\pi_2} G/K
\]
for $H,K \in \cat S$. We note that for any subgroup $H \subseteq G$, there is a homeomorphism 
\begin{equation}\label{spc}
\Spc(\Mod_{\cat C}(R \otimes \mathbb{D}(G/H_+)))\cong \mathcal{V}(R,H)
\end{equation}
by \cref{cor:homotopy-category} which allows us to commute passage to the homotopy category with forming the corresponding module category. Using this and the definition of $A_{\cat S}$, we get a decomposition 
\[
\Spc(\Mod_{\cat C}(R \otimes A_{\cat S}))\cong \bigsqcup_{H \in \cat S}\mathcal{V}(R,H),
\]
and under this identification $\phi_{\cat S}$ corresponds to the map $\psi_{\cat S}$ induced by the individual extension of scalars functors. Similarly, using the Mackey decomposition formula (\Cref{rem:transitivegsets}) and (\ref{spc}), we get a decomposition
\[
	{\Spc(\Mod_{\cat C}(R \otimes A_{\cat S}^{\otimes 2}))} \cong 	\bigsqcup\limits_{\substack{H,K \in \cat S \\ [g] \in H\backslash G/K}} \mathcal{V}(R,H \cap K^g).
\]
Therefore, we can rewrite the above coequalizer as follows
\[\begin{tikzcd}
	{\bigsqcup\limits_{\substack{H,K \in \cat S \\ [g] \in H\backslash G/K}}\mathcal{V}(R,H \cap K^g)} & {\bigsqcup\limits_{H \in \cat S} \mathcal{V}(R,H)} & {\supp(R \otimes A_{\cat S})}
	\arrow["{\psi_1}", shift left=1, from=1-1, to=1-2]
	\arrow["{\psi_2}"', shift right=1, from=1-1, to=1-2]
	\arrow["{\psi_{\cat S}}",from=1-2, to=1-3].
\end{tikzcd}\]
The maps $\psi_1$ and $\psi_2$ are obtained by composing the left/right unit maps $ R \otimes A_{\cat S} \rightrightarrows R \otimes A_{\cat S}^{\otimes 2}$, 
with the maps 
\[
R \otimes \mathbb{D}(G/H_+) \otimes \mathbb{D}(G/K_+)  \simeq R \otimes \mathbb{D}((G/H \times G/K)_+) \xrightarrow{R \otimes \mathbb{D}((\beta_g)_+)} R \otimes \mathbb{D}(G/H^g \cap K_+)
\]
on components, and then passing to spectra. Thus we can calculate these compositions directly in the category of finite $G$-sets. It follows that $\psi_1$ and $\psi_2$ are induced by the maps of $G$-sets
\[
\alpha_g\colon G/(H^g \cap K) \xrightarrow{\beta_g} G/H \times G/K \xrightarrow{\pi_1} G/H
\]
and
\[
\alpha_1 \colon G/(H^g \cap K) \xrightarrow{\beta_g} G/H \times G/K \xrightarrow{\pi_2} G/K.
\]
The argument is now complete.
\end{proof} 

\begin{proof}[Proof of \cref{thm:quillen_decomposition}]
\leavevmode
\begin{itemize}[leftmargin=*]
    \item[$(a)$] The colimit description follows by considering $\cat S = \mathcal F$ in \Cref{Thm:coequalizer_collection}, as follows. 
    Since $R$ is $\mathcal F$-nilpotent,  we have $\supp(R \otimes A_\mathcal{F})=\mathcal{V}(R,G)$, for instance by combining~\cite[Proposition 3.15(iii)]{Balmer16b} with~\cite[Proposition 2.32]{Mathew2018}. Therefore \Cref{Thm:coequalizer_collection} gives a coequalizer diagram of topological spaces
\[\begin{tikzcd}
	{\bigsqcup\limits_{\substack{H,K\in\mathcal{F} \\ [g]\in H\backslash G/K}}\mathcal{V}(R, H^g \cap K)} & {\bigsqcup\limits_{H\in \mathcal{F}}\mathcal{V}(R,H) }	& \mathcal{V}(R,G).\arrow["{\psi_1}", shift left=1, from=1-1, to=1-2]
	\arrow["{\psi_2}"', shift right=1, from=1-1, to=1-2]
	\arrow["{\psi_{\mathcal F}}",from=1-2, to=1-3]
\end{tikzcd}\]
It remains then to identify this coequalizer description with the colimit over the orbit category in the statement. The map $\psi_{\mathcal F}$ factors through the canonical quotient map
\[
\pi \colon {\bigsqcup\limits_{H\in \mathcal{F}}\mathcal{V}(R,H)} \to {\mathop{\colim}\limits_{H\in \mathcal{O}_{\mathcal{F}}(G)}} \mathcal{V}(R,H)
\]
from the coproduct to the colimit, inducing a continuous map 
\[
\overline {\varphi} \colon {\mathop{\colim}\limits_{H\in \mathcal{O}_{\mathcal{F}}(G)}} \mathcal{V}(R,H) \to \mathcal{V}(R,G),
\]
since $\mathcal{V}(R,-)$ is a functor
on the orbit category. 

On the other hand, the maps $\psi_1$ and $\psi_2$ are equalized by the canonical map $\pi$ to the colimit, because they are induced by maps of $G$-sets, i.e., morphisms in the orbit category. The universal property of the coequalizer then provides us with a continuous map 
\[\overline \pi \colon \mathcal{V}(R,G) \to  {\mathop{\colim}\limits_{H\in \mathcal{O}_{\mathcal{F}}(G)}} \mathcal{V}(R,H).\]
Since both maps to $\mathcal{V}(R,G)$ are topological quotient maps, and both $\overline \pi$ and $\overline \varphi$ factor the identity, $\overline \pi$ and $\overline \varphi$ are inverse homeomorphisms.
\item[$(b)$] We will show that ${\mathcal V}(R,G)$ is a disjoint union $\bigsqcup_{(H)\in\mathcal F}{\mathcal V}_G^+(R,H)$ of locally closed subsets. Recall that ${\mathcal V}_G^+(R,H)\subseteq {\mathcal V}_G(R,H)$ is open, and that
${\mathcal V}_G(R,H)\subseteq {\mathcal V}(R,G)$ is closed (\Cref{not:quillen_strata}). Hence every ${\mathcal V}_G^ +(R,H)\subseteq {\mathcal 
V}(R,G)$ is locally closed. Moreover, it follows from \Cref{cor:eqspc_surjectivity} that there is a union ${\mathcal V}(R,G)=\bigcup_{(H)\in\mathcal F}{\mathcal V}_G^+(R,H)$. It remains to argue that this union is {\em disjoint}. To this end, consider the commutative diagram of symmetric monoidal left-adjoint functors
\[
\begin{tikzcd}[column sep=small]
\Mod_G(R)\arrow[r,"F"]\arrow[rr, bend left,"(\Phi^H)_H"] & \prod_{(H)} \Mod_G(\mathbb{D}(G/H_+) \otimes R \otimes \widetilde{E}[\subset H]) \arrow[r,"\sim"',"\Psi_H"] & \prod_{(H)} \Mod(\Phi^H R) \\
\Sp_G \arrow[u,"F_R"] \arrow[rr,"(\Phi^H)_H"'] & & \prod_{(H)} \Sp \arrow[u,"(F_{\Phi^H R} )_H"'].
\end{tikzcd}
\]
Here $(H)$ runs through a set of representatives for subgroups of $G$ up to conjugacy, the upper part of the diagram comes from \cref{lem:res_finiteetale-gfp} and so it commutes, and the outer diagram commutes by \cref{rem:group_functoriality}.

On Balmer spectra, the previous diagram induces a commutative diagram 
\[
\xymatrix{{\mathcal V}(R,G) \ar[d] & \bigsqcup_{(H)} {\mathcal V}^+(R,H) \ar[l] \ar[d]\\
\Spc(\Sp_G^c) & \bigsqcup_{(H)} \Spc(\Sp^c), \ar[l]^-{\sim} } 
\]
where the bottom map is a bijection by a result of Balmer and Sanders, see~\cite[Theorem 4.9, Theorem 4.14]{BalmerSanders17}. The components of $\bigsqcup_{(H)\in\mathcal{F}}{\mathcal V}^+(R,H)$ have pairwise disjoint images in $\Spc(\Sp_G^c)$, as seen by considering the composition through the lower right corner. It follows that their images in ${\mathcal V}(R,G)$, which by definition are ${\mathcal V}^+_G(R,H)$, must be disjoint as well.
\item[$(c)$] The final claim is that for every subgroup $H \in \mathcal{F}$, the map ${\mathcal V}^+(R,H)\to{\mathcal V}^+_G(R,H)$ induced by restriction gives a homeomorphism
\[
{\mathcal V}^+(R,H)/W_G(H)\to{\mathcal V}^+_G(R,H).
\]
To see this, we  set $\cat S= \{H\}$ in \cref{Thm:coequalizer_collection}, and obtain the following coequalizer
\[
\xymatrix{\bigsqcup_{g \in H\backslash G/H} {\mathcal V}(R,H\cap H^g) \ar@<0.5ex>[r]^-{\psi_2} \ar@<-0.5ex>[r]_-{\psi_1} &  {\mathcal V}(R,H) \ar[r]^-{\psi_H} &{\mathcal V}_G(R,H).  }
\]
 The key observation now is that the open subset ${\mathcal V}^+(R,H)\subseteq {\mathcal V}(R,H)$ is saturated for the equivalence relation generated by the maps $\psi_1$ and $\psi_2$. Two elements $x,y\in{\mathcal V}(R,H)$ are equivalent exactly if there is a chain $x=x_1, x_2, \ldots ,x_n=y$ such that for each $i$ there is some $g\in G$ and some $z\in {\mathcal V}(R,H\cap H^g)$ such that $x_i=\psi_1(z)$ and $x_{i+1}=\psi_2(z)$, or the other way around. For our 
claim about being saturated, we can assume by induction that $n=2$. Since at least one of $x_1,x_2$ lies in ${\mathcal V}^+(R,H)$, the subgroup $H\cap H^g\subseteq H$ cannot be proper, i.e., $g\in N_G(H)$. So we are reduced to seeing that ${\mathcal V}^+(R,H)\subseteq{\mathcal V}(R,H)$
is $W_G(H)$-stable, which is clear. Observe that this argument also shows that the equivalence relation induced on ${\mathcal V}^+(R,H)$ is the one afforded by the action of the Weyl group. Saturation and a general fact from topology \cite[Theorem 22.1]{Munkres} now imply that the composition
\[ {\mathcal V}^+(R,H)\subseteq{\mathcal V}(R,H)\to {\mathcal V}(R,G)\]
restricted to its image is a topological quotient map.
Since this image is by definition ${\mathcal V}^+_G(R,H)$ and we identified the resulting equivalence relation
with the Weyl group action, the proof is complete. \qedhere
\end{itemize}
\end{proof}

We now apply \cref{thm:quillen_decomposition} to give an example where the map in~\cref{cor:eqspc_surjectivity} is in fact bijective. This gives an alternative proof of~\cite[Theorem~13.11]{bhs1}.

\begin{Prop}\label{prop:infl-guys}
    Let $R \in \CAlg(\Sp)$ and let $R_G \coloneqq \infl_e(R) \in \CAlg(\Sp_G)$ be the image of $R$ under the inflation functor.  Then the canonical map
   \[
     \xymatrix{\bigsqcup\limits_{(H)\subseteq G}\mathcal{V}^+(R_G,H) \ar[r]^-{\cong} & \mathcal{V}(R_G,G)}
    \]
    is bijective.
\end{Prop}

\begin{proof}
 We observe first that $\Phi^HR_G \simeq R$ for all $H \subseteq G$ (for example, \cite[Equation 13.5]{bhs1}). This implies that if $R$ is non-trivial then the derived defect base of $R_G$ is the family of all subgroups, see \cref{prop:equivalent-R-nilpotent}. Moreover, the composite 
    \[
    \Ho(\Mod(R)) \xrightarrow{\infl_e} \Ho(\Mod_G(R_G)) \xrightarrow{\Phi^H} \Ho(\Mod(\Phi^H R_G)) \simeq \Ho(\Mod(R))
    \]
    is naturally isomorphic to the identity functor~\cite[Remark 13.4 and Equation 13.6]{bhs1}. Passing to Balmer spectra, we see that the induced map $\mathcal{V}^+(R_G,H)\to \mathcal{V}(R_G,G)$ is injective. Comparing this with \cref{thm:quillen_decomposition}(c) shows that the Weyl group $W_G(H)$ must act trivially 
    on $\mathcal{V}^+(R_G,H)$. We can then apply \cref{thm:quillen_decomposition}$(b)$ to deduce that the required map is bijective. \qedhere
\end{proof}

\section{Strong Quillen stratification for global equivariant spectra}\label{sec:strong-quillen-global}
The goal of this section is to prove a refined version of~\cref{thm:quillen_decomposition} under the additional assumption that the commutative algebra $R \in \CAlg(\Sp_G)$ arises from a global homotopy type in the sense of \cite{Schwede18_global}. Before proving our result, let us recall some background on global stable homotopy theory following~\cite{LNP2022}.

\subsection{Recollections on global homotopy theory}

\begin{Def}\label{def:glo}
    Let $\mathrm{Glo}$ denote the \emph{global category} indexed on the family of finite groups, see~\cite[Definition 6.1]{LNP2022}. This is the $\infty$-category whose objects are all finite groups $G$, and whose morphism spaces are given by $\Hom(H,G)_{hG}$; the homotopy orbits of the conjugation $G$-action on the set of group homomorphisms. Composition in $\mathrm{Glo}$ is induced at the level of groupoids by composition of group homomorphisms and compatibility of homomorphisms with the conjugation action.
\end{Def}

\begin{Def}\label{def:orb}
    The global category contains a wide subcategory $\mathrm{Orb}\subseteq \mathrm{Glo}$ where we take the path connected components spanned by the injective group homomorphisms. 
\end{Def}

\begin{Def}\label{def:rep}
    Finally, we let $\mathrm{Rep}$ denote the homotopy category of $\mathrm{Glo}$; this is the category with the same objects of $\mathrm{Glo}$ but with morphisms given by conjugacy classes of group homomorphisms. To be more explicit, recall that two group homomorphisms $f,g\colon A \to B$ are conjugate if there exists $b\in B$ such that $c_b\circ f=g$ where $c_b(x)=bxb^{-1}$ for all $x\in B$.
\end{Def}

\begin{Rem}
By~\cite[Corollary 10.6]{LNP2022} there is a functor into the $\infty$-category of symmetric monoidal $\infty$-categories 
\begin{equation}\label{eq:functor}
\Sp_{\bullet}\colon \mathrm{Glo}^{\mathrm{op}} \to \Cat_\infty^\otimes
\end{equation}
which sends a finite group $G$ to $\Sp_G$, and a group homomorphism $\alpha \colon H \to G$ to the restriction-inflation functor $\alpha^* \colon \Sp_G \to \Sp_H$. Note that $\alpha^*$ is restriction if $\alpha$ is injective and inflation if $\alpha$ is surjective. In general, it will be the composite of two such functors. 
\end{Rem}

\begin{Def}
We define the $\infty$-\emph{category of global (equivariant) spectra} $\Sp_{gl}$ as the partially lax limit of $\Sp_\bullet$ marked at $\mathrm{Orb}$, see~\cite[Definition 4.11]{LNP2022}. Informally, an object $X\in \Sp_{gl}$ is determined by the following data:
\begin{itemize}
    \item a $G$-spectrum $X_G$ for all finite group $G$;
    \item compatible maps $f_{\alpha}\colon \alpha^* X_G \to X_H$ in $\Sp_H$ for all group homomorphisms $\alpha \colon H \to G$;
\end{itemize}
subject to the condition that $f_\alpha$ is an equivalence whenever $\alpha$ is injective.
\end{Def}

\begin{Rem}
    The $\infty$-category $\Sp_{gl}$ is symmetric monoidally equivalent to Schwede's model of global spectra based on orthogonal spectra, see~\cite[Theorem 11.10]{LNP2022}. As shown by Lenz \cite[Theorem A]{lenz2022globalalgebraicktheoryswan}, $\Sp_{gl}$ is also equivalent to an $\infty$-category of global spectral Mackey functors, globalizing the analogous statement for $G$-spectra due to Guillou--May \cite{GuillouMay2024} and Barwick \cite{Barwick17}.
\end{Rem}

\begin{Exa}\label{ex:borel-are-global}
    By the discussion in \cite[Example 4.5.19]{Schwede18_global}, any Borel-equivariant $G$-spectrum arises from a global homotopy type. Also, topological $G$-equivariant $K$-theory admits a global refinement by~\cite[Section 6.4]{Schwede18_global}. 
\end{Exa}

\begin{Not}
 Given a global spectrum $X\in \Sp_{gl}$, we will denote by $X_G$ its image under the canonical symmetric monoidal functor $\Sp_{gl}\to \Sp_G$, and call this {\em the underlying $G$-spectrum of $X$}.     
\end{Not}

\subsection{Quillen stratification for commutative global ring spectra}
Our goal in this section is to prove a refined version of~\cref{thm:quillen_decomposition} under the additional assumption that $R$ arises from a global homotopy type. The first step is to construct a more refined version of the functor 
\begin{equation}\label{functor-eq}
\Theta_0^R \colon \mathcal{O}_{\mathcal{F}}(G)\op\to \Cat, \quad G/H \mapsto \Ho(\Perf_G(R_G \otimes \mathbb{D}(G/H_+))).
\end{equation}
\begin{Lem}\label{lem:functor-mod}
    Let $R\in \CAlg(\Sp_{gl})$. Then there is a functor 
    \[
    \Mod_{\bullet}(R_\bullet)\colon  \mathrm{Glo}^{\mathrm{op}} \to \Cat_{\infty}^{\otimes}, \quad G \mapsto \Mod_G(R_G)
    \]
    extending (\ref{eq:functor}) to module categories. Moreover there is also an induced functor with values $G \mapsto \Perf_G(R_G)$.
\end{Lem}

\begin{proof}
    The existence of the first functor follows from~\cite[Theorem 5.10]{LNP2022}. The existence of the second functor follows from the first one by observing that by construction, the functor $\alpha^*$ is symmetric monoidal for all group homomorphism $\alpha$, and that dualizable and perfect objects agree in our setting.
\end{proof}

\begin{Rem}\label{rem:conj-id}
    For any $g\in G$, the conjugation isomorphism $c_g\colon G \to G$ lies in the same connected path component of $\Map_{\mathrm{Glo}}(G,G)$ as the identity map. The existence of the functor in \cref{lem:functor-mod}, then provides a natural isomorphism $c_g^* \Rightarrow 1$ between endofunctors of $\Perf_G(R_G)$ or $\Mod_G(R_G)$.
\end{Rem}

We now introduce a variation of Quillen's category.

\begin{Def}
Let $\mathcal{F}$ be a family of subgroups of a finite group $G$. The \emph{Quillen category} $\cat A_{\mathcal{F}}(G)$ is the category whose objects are subgroups of $G$ in $\mathcal{F}$, and morphisms from $H$ to $K$ in $\cat A_{\mathcal{F}}(G)$ are group homomorphims of the form $c_{g}\colon H \to K$, $h\mapsto ghg^{-1}$ for some $g\in G$. Note that the automorphism group of $H\in\cat A_{\mathcal{F}}(G)$ is $N_G(H)/C_G(H)$.
\end{Def}

\begin{Exa}
    Let $\mathcal{E}$ be the family of elementary abelian $p$-subgroups of a finite group $G$. Then $\cat A_{\mathcal{E}}(G)$ is the category $\mathcal A(G)$ defined by Quillen in \cite[page 567]{Quillen71}. 
\end{Exa}

We now introduce the orbit category associated to the Quillen category.

\begin{Def}
 Let $\mathcal{F}$ be a family of subgroups of a finite group $G$. The \emph{Quillen orbit category} $\mathcal{O}_{\mathcal{F}}^Q(G)$ is the category whose objects are subgroups of $G$ in $\mathcal{F}$, and morphisms from $H$ to $K$ in $\mathcal{O}_{\mathcal{F}}^Q(G)$ are given by $K$-conjugacy classes of morphisms in $\Hom_{\cat A_{\mathcal{F}}(G)}(H,K)$. The automorphisms of $H\in\mathcal{O}_{\mathcal{F}}^Q(G)$ are then $N_G(H)/H\cdot C_G(H)$.
\end{Def}

\begin{Not}
    Let $H\subseteq G$ be a subgroup. We set
    \[
     W_G^{gl}(H):=N_G(H)/H \cdot C_G(H)
    \]
    and refer to it as the \emph{global Weyl group} of $H$ in $G$. If $H$ is abelian so that $H \subseteq C_G(H)$, we write $W_G^Q(H)$ instead of $W_G^{gl}(H)$ and refer to it as the \emph{Quillen--Weyl group}.
\end{Not}

\begin{Cons}
  Assume $G$ is a finite group, $R\in\CAlg(\Sp_{gl})$, and $\mathcal F$ is a family of subgroups of $G$ such that $R_G$ is $\mathcal F$-nilpotent.
  Consider the functor 
  \[
  \mathrm{Glo}^{\mathrm{op}} \to \Cat, \quad H\mapsto \Ho(\Perf_H(R_H)),
  \]
  which factors through $\mathrm{Rep}$ by construction. There is also a functor $J \colon \mathcal{O}_{\mathcal{F}}(G) \to \mathrm{Rep}$ which sends a coset $G/H$ to $H$. Given a $G$-equivariant map $f\colon G/H \to G/K$ satisfying $f(eH)=gK$ we set $J(f)=[c_{g^{-1}}]\in \mathrm{Rep}(H,K)$. One easily verifies that $J$ is well-defined. We note that the image of $J$ is precisely $\mathcal{O}_{\mathcal{F}}^Q(G)$. Therefore, by precomposing, we obtain a functor 
    \begin{equation}\label{functor-glo}
     \Theta^R_1 \colon \mathcal{O}_{\mathcal{F}}(G)^{\mathrm{op}} \to \Cat, \quad G/H \mapsto \Ho(\Perf_H(R_H)),
    \end{equation}
    which by construction factors through $\mathcal{O}_{\mathcal{F}}^Q(G)\op$.
\end{Cons}

\begin{Lem}\label{lem:two-functors}
    The functors $\Theta_0^R$ and $\Theta_1^R$ from (\ref{functor-eq}) and (\ref{functor-glo})
    are naturally pseudo-isomorphic and they factor through $\mathcal{O}_{\mathcal{F}}^Q(G)$.
\end{Lem}
\begin{proof}
    For any subgroup $H\subseteq G$, we let $$L_H \colon \Ho(\Perf_G(R_G \otimes \mathbb{D}(G/H_+))) \stackrel{\sim}{\rightarrow} \Ho(\Perf_H(R_H))$$ denote the equivalence of \cref{prop:base-change-naturality}. We will show that the equivalences $L_H$ assemble to give a pseudonatural equivalence between $\Theta_0^R$ and $\Theta_1^R$. As the latter factors through $\mathcal{O}^Q_{\mathcal{F}}(G)$, so does $\Theta_0^R$. 
    
    Let $f\colon G/H\to G/K$ with $f(eH)=gK$ be given so that $J(f)=c_{g^{-1}}\colon H \to K$. We will show that the following diagram commutes up to natural isomorphism
    \[
    \begin{tikzcd}
        \Ho(\Perf_{K}(R_K)) \arrow[r,"c_{g^{-1}}^*"] & \Ho(\Perf_H(R_H)) \\
        \Ho(\Perf_G(R_G \otimes \mathbb{D}(G/ K_+))) \arrow[u, "\sim","L_K"'] \arrow[r,"F_f"]& \Ho(\Perf_G(R_G \otimes \mathbb{D}(G/H_+))),\arrow[u,"\sim","L_H"']
\end{tikzcd}
\]
where $F_f$ is a shorthand for the extension of scalars functor along $R \otimes \mathbb{D}(f_+)$. 

To this end, consider the following diagram
\[\resizebox{\columnwidth}{!}{$\displaystyle
\begin{tikzcd}[ampersand replacement=\&]
	\& {\mathrm{Ho}(\mathrm{Perf}_{K}(R_K))} \& {} \& {\mathrm{Ho}(\mathrm{Perf}_H(R_H))} \\
	\\
	\& {\mathrm{Ho}(\mathrm{Perf}_G(R_G \otimes \mathbb{D}(G/K_+)))} \& {} \& {\mathrm{Ho}(\mathrm{Perf}_G(R_G \otimes \mathbb{D}(G/H_+)))} \\
	\\
	{\mathrm{Ho}(\mathrm{Perf}_G(R_G))} \&\& {} \& {} \& {\mathrm{Ho}(\mathrm{Perf}_G(R_G))}
	\arrow["{F_{f}}"', from=3-2, to=3-4]
	\arrow["\sim"',"L_K", from=3-2, to=1-2]
	\arrow["L_H"',"\sim", from=3-4, to=1-4]
	\arrow["{c_{g^{-1}}^*}", from=1-2, to=1-4]
	\arrow["{\mathrm{res}_H}"', curve={height=30pt}, from=5-5, to=1-4]
	\arrow["{F_{\mathbb{D}(G/H_+)\otimes R_G}}", from=5-5, to=3-4]
	\arrow["{F_{\mathbb{D}(G/K_+)\otimes R_G}}"', from=5-1, to=3-2]
	\arrow["{c_{g^{-1}}^*}", from=5-5, to=5-1]
	\arrow["{\mathrm{res}_K}", curve={height=-30pt}, from=5-1, to=1-2]
\end{tikzcd}
$}
\]
in which the right and left triangles commute by \Cref{prop:base-change-naturality}. 
By \cref{rem:conj-id}, there is a natural isomorphism $ c_{g^{-1}}^* \Rightarrow 1$ between endofunctors of $\Ho(\Perf_G(R_G))$. This natural isomorphism makes the bottom square commute as one readily verifies that $F_f\circ F_{\mathbb{D}(G/K_+)\otimes R_G}\simeq F_{\mathbb{D}(G/H_+)\otimes R_G}$. Note that the outer diagram commutes as well, using that $i_K \circ c_{g^{-1}}=c_{g^{-1}} \circ i_H$, where $i_K \colon K \subseteq G$ and $i_H \colon H \subseteq G$ denote the inclusion of subgroups. This shows that there is a natural isomorphism $\eta \colon L_H\circ F_f\circ F_{\mathbb{D}(G/K_+)} \Rightarrow c_{g^{-1}}^*\circ L_K \circ F_{\mathbb{D}(G/K_+)}$. We then apply \cref{prop:nat-tra} (keeping in mind \cref{cor:homotopy-category}) to obtain a natural transformation $\widetilde{\eta}\colon L_H\circ F_f \Rightarrow c_{g^{-1}}^*\circ L_K$ extending $\eta$. A thick subcategory argument then shows that $\widetilde{\eta}$ is also an isomorphism.
This uses that ${\mathrm{Ho}(\mathrm{Perf}_G(R_G \otimes \mathbb{D}(G/K_+)))}$ admits a set of generators in the image of $F_{\mathbb{D}(G/K_+)\otimes R_G}$, namely the orbits, see \cref{lem:eqmodulegenerators}.
\end{proof}

We are finally ready to state the main result of this section. Recall the notation introduced in \cref{not:quillen_strata}.

\begin{Thm}\label{thm:globalquillenstrat}
Assume $G$ is a finite group, $R\in\CAlg(\Sp_{gl})$,
and $\mathcal F$ is a family of subgroups of $G$ such that $R_G$ is $\mathcal F$-nilpotent. Then:
\begin{itemize}
\item[(a)] The restriction maps induce a homeomorphism
\[ \xymatrix{{{\mathop{\colim}\limits_{H\in \mathcal{O}^Q_{\mathcal{F}}(G)} {\mathcal V}(R_G,H) }}\ar[r]^-{\cong} &{\mathcal V}(R_G,G).}\]
\item[(b)] There is a decomposition into locally closed disjoint subsets
\[ {\mathcal V}(R_G,G)=\bigsqcup_{(H)\in\mathcal F}{\mathcal V}_G^+(R_G,H).\]
The index runs over conjugacy classes of subgroups in $\mathcal F$.
\item[(c)] For every $H\in\mathcal F$, restriction induces a homeomorphism
\[ \xymatrix{{\mathcal V}^+(R_G,H)/W_G^{gl}(H) \ar[r]^-{\cong} &{\mathcal V}^+_G(R_G,H).}\]
\end{itemize}
\end{Thm}

\begin{proof}
With the work done so far, all the claims follow from \cref{thm:quillen_decomposition} as we now explain. Recall from \cref{lem:two-functors} that the functor $G/H \mapsto \mathcal{V}(R_G,H)$ factors through $J \colon\mathcal{O}_{\mathcal{F}}(G) \to \mathcal{O}_{\mathcal{F}}^Q(G)$. Note that $J$ is cofinal since it is bijective on objects and surjective on morphisms. Combining this with \cref{thm:quillen_decomposition}$(a)$ gives part $(a)$. Part $(b)$ is precisely \cref{thm:quillen_decomposition}$(b)$. Finally, part $(c)$ follows from \cref{thm:quillen_decomposition}$(c)$ and the observation that the centralizer $C_G(H)$ must act trivially on $\mathcal{V}^+(R_G,H)$ as $G/H \mapsto \mathcal{V}(R_G,H)$ factors through $J$.
\end{proof}

The globality assumption in \cref{thm:globalquillenstrat} implies that the action of $C_G(H)$ on $\Spc(\Perf \Phi^HR)$ is trivial. The next example shows that it can be non-trivial in general.

\begin{Exa}\label{ex:gl-hyp-needed}
    Let $G$ be the cyclic group of order $2$ and let $R=\mathbb{D}(G_+)\in \CAlg(\Sp_{G})$. We note that $R$ does not come from a global spectrum as the restriction map $\res^{G}_e \colon \pi_0^{G}(R)\to \pi_0(R)$ is not surjective, compare with \cite[Remark 4.1.2]{Schwede18_global}. We also note that $\Phi^{G}R\simeq 0$ and that $\Phi^e R \simeq S^0 \oplus S^0$. By \cref{thm:quillen_decomposition}, we have 
    \begin{equation}\label{eq:weyl}
     \Spc(\Perf_{G}(R))= \Spc(\Perf \Phi^eR)/G= (\Spc(\Sp^c) \sqcup \Spc(\Sp^c))/G.
    \end{equation}
    On the other hand, $ \Spc(\Perf_{G}(R))=\Spc(\Sp^c)$ by \cref{prop:base-change-naturality}. It follows that the Weyl group $W_G(e)=G$ does not act trivially in (\ref{eq:weyl}). In particular we cannot replace $W_G(e)$ with $W_G^{gl}(e)$ in (\ref{eq:weyl}).
\end{Exa}

\subsection{Quillen stratification, revisited}\label{ssec:ttmodpquillen}
The goal of this subsection is to show that \Cref{thm:globalquillenstrat} contains a weak version of Quillen's stratification theorem as a special case. The precise formulation of this stratification result appears in \cref{prop:ttvgeplus}.

\begin{Not}
Let $k$ be a field of characteristic $p$. For any finite group $G$, we set ${\mathcal V}^h_G=\Spec^h(H^*(G,k))$. If $E$ is an elementary abelian $p$-group, we also set ${\mathcal V}_E^{h,+}= {\mathcal V}_E^h \setminus \bigcup_{E'<E} {\mathcal V}^h_{E'}$. Finally, if $E\subseteq G$ we write $\mathcal{V}_{G,E}^{h,+}$ for the image of $\mathcal{V}_{E}^{h,+}$ under the restriction map $r_E^G\colon\mathcal{V}_{E}^h \to \mathcal{V}^h_{G}$.
\end{Not}

\begin{Rem}\label{rem:balmerquillen}
Recall that in this situation, Balmer's comparison map
\[ 
\xymatrix{\rho^G\colon {\mathcal V}(\underline{k},G) \coloneqq \Spc(\Perf_G(\underline{k})) \ar[r]^-{\cong} & {\mathcal V}_G^h}
\]
is a homeomorphism. To see this, note that by \Cref{rem:stmodbikstratification} we have $\Spc(\Perf_G(\underline k)) \simeq \Spc(\mathrm{D}^b(kG))$, while Balmer's elaboration \cite[Proposition 8.5]{Balmer10b} on work of Benson, Carlson, and Rickard \cite{BensonCarlsonRickard97} shows that the comparison map is a homeomorphism in this case. 
\end{Rem}

\begin{Lem}\label{lem:ttveplus}
If $E$ is an elementary abelian $p$-group and $k$ a field of characteristic $p$, then the homeomorphism 
$\rho^E \colon \Spc(\Perf_E(\underline{k}))\to \mathcal{V}^h_E$ restricts to a homeomorphism between open subsets
\[
\xymatrix{{\mathcal V}^+(\underline{k},E) \ar[r]^-{\cong} & {\mathcal V}_E^{h,+}.}
\]
\end{Lem}
\begin{proof}
For any subgroup $E' \subseteq E$, we have $\supp(\mathbb D(E/E'_+)\otimes \underline{k}) \cong \Spc(\Perf_{E'}(\underline{k}))$, thus providing an inclusion $\Spc(\Perf_{E'}(\underline{k})) \hookrightarrow \Spc(\Perf_{E}(\underline{k}))$. Likewise, there are inclusions ${\mathcal{V}}^h_{E'} \hookrightarrow {\mathcal{V}}^h_E$. We obtain a commutative diagram of topological spaces
\[
\xymatrix{ \Spc(\Perf_E(\underline{k})) \ar[r]^{\rho^E}_{\cong} & {\mathcal{V}}^h_E \\
\bigcup_{E' \subsetneq E} \Spc(\Perf_{E'}(\underline{k})) \ar[r]^-{\cong} \ar@{^{(}->}[u] & \bigcup_{E'\subsetneq E}{\mathcal{V}}_{E'}^h, \ar@{^{(}->}[u]}
\]
where both the bottom and the top map are homeomorphisms, as observed in \cref{rem:balmerquillen}. The result then follows by passing to complements and using \cref{thm:quillen_decomposition}$(b)$. 
\end{proof}

We can then give a proof Quillen's stratification theorem in the form of~\cite[Theorem 9.1]{BensonIyengarKrause11a} from a tt-geometric point of view. 

\begin{Prop}\label{prop:ttvgeplus}
Let $E$ be an elementary abelian $p$-subgroup of a finite group $G$ and let $k$ be a field of characteristic $p$. Then, restriction induces a homeomorphism
\[
\xymatrix{\mathcal{V}_{E}^{h,+}/W^{Q}_G(E) \ar[r]^-{\cong}&\mathcal{V}_{G,E}^{h,+}.}
\]
Therefore, there is a decomposition of $\mathcal{V}_G^h$ into locally closed disjoint subsets
\[
\mathcal{V}_G^h= \bigsqcup_{(E)} \mathcal{V}_{E}^{h,+}/W^{Q}_G(E),
\]
where the index runs over $G$-conjugacy classes of elementary abelian $p$-subgroups of $G$.
\end{Prop}
\begin{proof}
There is a commutative diagram of spectra
\[
\begin{tikzcd}
\mathcal{V}(\underline{k}, E) \arrow[r,"\psi_E"] \arrow[d, ,"\cong"] & \mathcal{V}(\underline{k}, G)
 \arrow[dd,"\cong"',"\rho^G"] \\
 \Spc(\Perf_E(\underline{k})) \arrow[ur,"\Spc(\res_E)"'] \arrow[d,"\cong","\rho^E"']& \\
 \mathcal{V}_E^h \arrow[r,"r_E"] & \mathcal{V}_G^h.
\end{tikzcd}
 \]
 The commutativity of the upper triangle comes from \cref{prop:base-change-naturality} and the commutativity of the bottom part follows from naturality of $\rho^{(-)}$.
As shown in \cref{lem:ttveplus}, under the left vertical composite,  $\mathcal{V}_{E}^{h,+}$ identifies with ${\mathcal V}^+(\underline{k},E)$. We conclude that the right vertical homeomorphism identifies $\mathcal{V}_{G,E}^{h,+}$ with $\im \psi_E = {\mathcal V}_G^+(\underline{k},E)$. We may thus apply \cref{thm:globalquillenstrat} and \cref{lem:ttveplus} again to deduce 
\[
\mathcal{V}_{G,E}^{h,+} \cong {\mathcal V}_G^+(\underline{k},E) \cong {\mathcal V}^+(\underline{k},E)/W^{Q}_G(E) \cong \mathcal{V}_{E}^{h,+}/W^{Q}_G(E).
\]
Unwinding the construction, this homeomorphism is indeed induced by the restriction map. 
The final claim then follows from~\cref{thm:globalquillenstrat}.
\end{proof}

\begin{Rem}
 \Cref{prop:ttvgeplus} is close in spirit to Quillen's stratification result for the spectrum of an equivariant cohomology ring~\cite[Stratification Theorem 10.2]{Quillen71}; however there is a  difference. Our stratification result applies to the variety $\Spec^h(H^*(G,k))$ associated to the graded commutative ring $H^*(G,k)$, whereas Quillen's result applies to the varieties $\Spec(H^\bullet(G,k))$ and $\mathrm{MaxSpec}(H^\bullet(G,k))$ associated to the commutative ring
 \[
 H^\bullet(G,k)\coloneqq\begin{cases}
     H^*(G,k) & \mathrm{if}\; p=2\\
     H^{\mathrm{even}}(G,k) & \mathrm{if} \; p \; \mathrm{odd}.\\
 \end{cases}
 \]
 As any homogeneous prime ideal of $H^*(G,k)$ defines a prime ideal of $H^\bullet(G,k)$, one can deduce \cref{prop:ttvgeplus} from Quillen's work. In this sense our result is a weaker version of Quillen's stratification theorem. 
\end{Rem}

\begin{Rem}\label{rem:freeness}
    \cref{prop:ttvgeplus} above implies that the Quillen--Weyl group acts transitively on homogeneous prime ideals lying over a given homogeneous prime ideal of $\mathcal{V}_{G,E}^{h,+}$. As already noted by Quillen~\cite[Section 11]{Quillen71}, this action is in general not free, see~\cite[Warning 5.14]{BIK2012} for an explicit example. The action is however free if $k$ is a separably closed field and we replace the homogeneous spectrum with the variety of maximal ideals, see ~\cite[Proposition 9.6]{Quillen71} for the precise statement. 
\end{Rem}

\begin{Rem}\label{rem:res-not-injective}
\cref{prop:ttvgeplus} shows that $r_E^G \colon \mathcal{V}_E^{h,+} \to \mathcal{V}^{h,+}_{G,E}$ is the quotient map induced by the action of $W_G^{Q}(E)$. In particular, $r_E^G$ is not injective in general. For a specific example consider $G=(\Z/2 \times \Z/2)\rtimes \Z/2$ where $\Z/2$ acts on $\Z/2 \times \Z/2$ by swapping the two copies, and let $E\subseteq G$ be the Klein four-group. Let $k$ be a finite field of cardinality $4$ and let $\alpha$ be a root of the polynomial $x^2+x+1$ in $k$. Note that $W_G^{Q}(E)=\Z/2$ acts on $\mathcal{V}^{h}_E=\Spec^h(k[x,y])$ by swapping $x$ and $y$. One checks that the two distinct homogeneous prime ideals $(x+\alpha y)$ and $(x+(1+\alpha)y)$ belong to $\mathcal{V}^{h,+}_E$ and are identified under $r_{E}^G$.
\end{Rem}

\section{Stratification via a theorem of Drinfel'd--Strickland}\label{sec:cohomologicalstratification}

In this section we investigate the category of modules over the Borel-equivariant $G$-spectrum associated to a Lubin--Tate theory $E$, showing that it is cohomologically stratified. The key input is \Cref{thm:regular} which shows that the ring $\pi_0(\Phi^{A}\underline{E}_A)$ is regular Noetherian for every finite abelian $p$-group $A$. The proof of this relies on the theory of level $A$-structures, as introduced by Drinfel'd \cite{Drinfeld1974Elliptic} and studied further by Strickland \cite[Section 7]{Strickl1997Finite}.

\subsection{The $E$-cohomology of finite abelian groups}\label{ssec:ecohomabgps}

\begin{Not}
We let $E = E_n \in \CAlg(\Sp)$ denote a \emph{Lubin--Tate theory} (aka \emph{Morava $E$-theory}) associated with a connected $p$-divisible group of height $n\ge 1$ over a perfect field of characteristic $p>0$, see \cite[Section 5]{Lurie_elliptic} for an overview. We will denote the associated universal deformation by $\mathbb G_E$, usually viewed of as a $p$-divisible group. As previously, $\underline{E} \in \CAlg(\Sp_G)$ will denote the associated Borel-equivariant $G$-spectrum for a finite group $G$, see \Cref{ex:borel_g_spectra}.
\end{Not}

We recall the following fundamental description of the $E$-cohomology of finite abelian groups, see \cite[Corollary 5.10]{HopkinsKuhnRavenel2000Generalized}.

\begin{Lem}\label{lem:morava_abelian_p_group}
The $E$-cohomology $E^*(BA)$ of any finite abelian group $A$ is concentrated in even degrees, even periodic, and $E^0(BA)$ is a finitely generated free $E^0$-algebra.
\end{Lem}

\begin{Not}
    For a finite abelian $p$-group $A$, we write $A^*$ for the Pontryagin dual of $A$. We also view any $A$ as a group scheme via the constant functor. 
\end{Not}

\begin{Rec}\label{rec:ecohomofabgps}
For a finite abelian $p$-group $A$, recall that the $\pi_0E$-algebra
$\pi_0E^{BA}$ corepresents homomorphisms of group schemes $A^*\to \mathbb G_E$ into the universal deformation $\mathbb G_E$, see \cite[Proposition 5.12]{HopkinsKuhnRavenel2000Generalized}. The functoriality in $A$ of this identification is as follows:
for a map $f\colon A_1\to A_2$ we obtain a map of $\pi_0E$-algebras $\pi_0E^{BA_2}\to \pi_0E^{BA_1}$ which corepresents the composition of $f^*\colon A_2^*\to A_1^*$ with $A_1^*\to \mathbb G_E$. In particular, for every subgroup $A'\subseteq A$ we obtain a map of affine schemes
\[ 
\mathrm{Spec}(\pi_0 E^{BA'})\longrightarrow \mathrm{Spec}(\pi_0 E^{BA}).
\]
The universal property of $\pi_0E^{BA}$, applied to the identity map $\pi_0E^{BA} \to \pi_0E^{BA}$, provides a universal homomorphism $A^* \to \mathbb{G}_E(\pi_0E^{BA})$. Choosing a coordinate on $\mathbb G_E$, we thus obtain a map 
\begin{equation}\label{eq:universalmapeba}
    \alpha \colon A^*\to\mathbb G_E(\pi_0E^{BA})\subseteq \pi_0E^{BA}.
\end{equation} 
From now on, we will implicitly fix such a coordinate on $\mathbb G_E$.
\end{Rec}

\begin{Prop}\label{prop:closed_immersion}
The map $\mathrm{Spec}(\pi_0 E^{BA'}) \hookrightarrow \mathrm{Spec}(\pi_0 E^{BA})$ is a closed immersion, i.e., the map of $\pi_0E$-algebras $\pi_0E^{BA}\to \pi_0E^{BA'}$ given by restriction is surjective.
\end{Prop}
\begin{proof}
Via the universal homomorphism \eqref{eq:universalmapeba}, $(A/A')^*\subseteq A^*$ defines a subset $T$ of $\pi_0E^{BA}$. We denote by $I$ the ideal generated by $T$ and claim that the $\pi_0E^{BA}$-algebras
$\pi_0E^{BA'}$ and $\pi_0E^{BA}/I$ are isomorphic. This will in particular verify the claim of the proposition. Now, the two algebras are isomorphic because they corepresent the same functors over the one corepresented by $\pi_0E^{BA}$: Indeed,
a homomorphism defined on $A^*$ will factor through $(A')^*$
if and only if it annihilates $(A/A')^*$.
\end{proof}

\begin{Not}\label{not:xa}
We define $X_{A} \coloneqq \mathrm{Spec}(\pi_0 E^{BA})$. For a subgroup $A' \subseteq A$, we view $X_{A'}$ as a closed subscheme in $X_A$ via \cref{prop:closed_immersion}, and write 
\[ 
X_{A'}^\circ\coloneqq X_{A'}\setminus\left(\bigcup_{A''\subsetneq A'} X_{A''}\right).
\]
Then $X_{A'}^\circ\subseteq X_{A'}$ is an open subscheme.
\end{Not}

Mapping $A'$ to $X_{A'}$ induces a map of lattices from the lattice of subgroups of $A$ to the lattice of closed subsets of $X_A$. This map preserves meets:

\begin{Prop}\label{prop:intersection}
Let $A',A'' \subseteq A$, then we have $X_{A'\cap A''}=X_{A'}\cap X_{A''}$.
\end{Prop}

\begin{proof}
Using \Cref{prop:closed_immersion} we see that the scheme theoretic intersection $X_{A'}\cap X_{A''}$ is also the fiber product
$X_{A'}\times_{X_A} X_{A''}$. We claim that the given map 
\[ 
X_{A'\cap A''}\to X_{A'}\times_{X_A}X_{A''} 
\]
is an isomorphism of affine schemes over $X_A$. We check that the map induced on points
\[ 
X_{A'\cap A''}(R)\to X_{A'}(R)\times_{X_A(R)}X_{A''}(R) 
\]
is bijective for every $\pi_0E^{BA}$-algebra $R$. Indeed, an $R$-valued point on the left hand side 
corresponds to a homomorphism $(A'\cap A'')^*\to \mathbb G_E(R)$, one of the right hand side
to a pair of homomorphisms $((A')^*\to\mathbb G_E(R), (A'')^*\to\mathbb G_E(R))$
which agree when pulled back to $A^*$, and the comparison map restricts a given homomorphism to 
$(A')^*$ and $(A'')^*$. Observe that $A'\cap A''$ is the pull-back in abelian groups of $A'$ and 
$A''$ mapping to $A$. Then the claim follows from the fact that the dual of a pull-back of finite abelian groups is a push-out in the category of (not necessarily finite) abelian groups.
\end{proof}

\begin{Cor}\label{cor:etheoryzariskistratification}
For every finite abelian group $A$, we have a decomposition of $X_A$ into a disjoint union of locally-closed subsets:
\begin{equation}\label{eq:disjoint_union_set_spec}
X_A = \bigsqcup_{A'\subseteq A} X_{A'}^\circ.
\end{equation}
\end{Cor}
\begin{proof}
Since for every $x\in X_A$, the set $\{ A'\subseteq A\, \mid \, x\in X_{A'}\}$ is non-empty and stable under intersection by \Cref{prop:intersection}, there
is a smallest $A'$ with $x\in X_{A'}$. This implies the claim, because the collection $(X_{A'}^\circ)_{A' \subseteq A}$ covers $X_A$.
\end{proof}

\begin{Rem}
This disjoint union is not topological if $A\neq 0$ because $X_A$ is connected, being the Zariski spectrum of a local ring.
\end{Rem}

\subsection{Regularity of geometric fixed point spectra}\label{sec:regularity}

In this subsection, we study the schemes $X_{A'}^\circ$ appearing in the decomposition \cref{cor:etheoryzariskistratification} in more detail; in particular, we will show that they are regular. We begin with a topological interpretation of $X_{A'}^{\circ}$ in terms of geometric fixed points.

For every $A'\subseteq A$, the canonical map $\underline{E} \to \widetilde{E}\mathcal{P}_{A'}\otimes \underline{E}$ induces a map of commutative ring spectra $E^{BA} \to E^{BA'}\to \Phi^{A'}\underline{E}$ to the $A'$-geometric fixed points of the genuine $A$-spectrum $\underline{E}$. Concretely, $\Phi^{A'}\underline{E}$ is given by inverting on $E^{BA'}$ all
Euler classes of all non-trivial characters of $A'$ (see \cite[Proposition II.9.13]{LewisMaySteinbergerMcClure86}). 
This implies that the first of the following maps on Zariski spectra
\[ \Spec(\pi_0\Phi^{A'}\underline{E})\to\Spec(\pi_0E^{BA'})=X_{A'}\hookrightarrow X_A \]
is an open immersion, and we determine the image of the composition in the next result.

\begin{Lem}\label{lem:geometric_fixed_points_as_open_subscheme} 
For every subgroup $A'\subseteq A$, we have an equality of open subsets
\[ 
\Spec(\pi_0\Phi^{A'}\underline{E}) =  X_{A'}^\circ\subseteq X_{A'}\left( \subseteq X_A\right).
\]
\end{Lem}
\begin{proof}
This is the special case of \cite[Lemma 3.11]{BHNNNS19} in which the family consists of all proper subgroups.
\end{proof}

\begin{Prop}\label{prop:abecohomdecomposition}
For any finite abelian $p$-group $A$, there is a (set-theoretic) decomposition of the Zariski spectrum of its $E$-cohomology as
\[
\Spec(E^0(BA)) \simeq \bigsqcup_{A' \subseteq A} \Spec(\pi_0\Phi^{A'}\underline{E}).
\]
\end{Prop}
\begin{proof}
This follows from \eqref{eq:disjoint_union_set_spec} combined with \cref{lem:geometric_fixed_points_as_open_subscheme}. 
\end{proof}

We next establish a regularity statement which is key to our stratification result for Lubin--Tate theory. 

\begin{Thm}\label{thm:regular} 
For any finite group $G$ and any Lubin--Tate theory $E$, the commutative ring $\pi_0(\Phi^{G}\underline{E})$ is regular Noetherian.
\end{Thm}

The proof will be given after some preliminaries. 

\begin{Rem}\label{rem:regularreduction}
Using \cref{ex:E-nilpotent}, we can reduce the proof of \cref{thm:regular} to the case that $G=A$ is a finite abelian $p$-group which is generated by at most $n$ elements; otherwise, $\pi_0(\Phi^{G}\underline{E}) \simeq 0$ by \cref{prop:equivalent-R-nilpotent}. We will achieve this by comparing with level $A$-structures, as introduced by Strickland \cite{Strickl1997Finite}, generalizing work of Drinfel'd \cite{Drinfeld1974Elliptic}. 
\end{Rem}

\begin{Rec}
To review the construction of the quotient
$\pi_0E^{BA} \rightarrow D$ parametrizing level-$A$-structures on $\mathbb G_E$, we define two monic polynomials $P,Q\in \pi_0E^{BA}[X]$.
Choosing a coordinate on the universal deformation $\mathbb G_E$ as in \cref{rec:ecohomofabgps}, we identify the universal homomorphism $\alpha$ defined over $\pi_0E^{BA}$ with a map  $\alpha\colon A^*\to \mathbb G_E(\pi_0E^{BA})\subseteq \pi_0E^{BA}$. This lets us define our first polynomial as 
\begin{equation}\label{eq:fpolynomial}
P(X) \coloneqq \prod_{a\in A^*\colon pa =0 }(X-\alpha(a))\in \pi_0E^{BA}[X]. 
\end{equation}
Note that $P$ is monic and of degree $p^{\mathrm{rk}_p(A)}$, where $\rk_p(A):=\mathrm{dim}_{\Fp}(A\otimes_\mathbb Z \Fp)$ denotes the $p$-rank of $A$. We also define $Q\in \pi_0E[X]\subseteq \pi_0E^{BA}[X]$ to be the unique monic polynomial such that $Q$ generates the same ideal in $\pi_0E[\![X]\!]$ as does the $p$-series of $\mathbb G_E$ (with respect to the coordinate $X$ chosen above). Existence and uniqueness of $Q$ follow from the Weierstrass preparation theorem for $\pi_0E$. Note that $Q$ is monic of degree $p^n$, where $n$ is the height of $\mathbb G_E$.

Some elementary commutative algebra shows that, given the monic polynomials $P,Q\in\pi_0E^{BA}[X]$, there is an initial
ring map $\rho\colon\pi_0E^{BA}\to D$ such that the image of $P$ in $D[X]$ divides the image of $Q$. The ring $D$ is known as the \emph{Drinfel'd ring}. For more details on this construction of $D$, see \cite[Section 7]{Strickl1997Finite}.

An important result is that $D$ is a regular, local Noetherian ring, see \cite[Theorem 7.3]{Strickl1997Finite}. This implies that all localizations $D_{\mathfrak q}$ are regular for all prime ideals $\mathfrak q \subseteq D$, not just the maximal one, see  \cite[Theorem 19.3]{Matsumura1989}.
\end{Rec}

We will use the following criterion for divisibility. Recall that a polynomial over a field $k$ is said to be separable if all its roots $\alpha\in\overline{k}$ in an algebraic closure $\overline{k}$ of $k$ are simple, i.e., satisfy $f'(\alpha)\neq 0$.

\begin{Prop}\label{prop:divide_over_artin}
Assume $R=(R,\mathfrak m,K)$ is a local $\pi_0E^{BA}$-algebra such that over its residue field $K$, the image of $P$
is separable and divides the image of $Q$. Then the image of $P$ in $R[X]$ divides the image of $Q$.
\end{Prop}
\begin{proof}
We denote by $P_R=\prod_{i=1}^N (X-r_i)\in R[X]$ the image of $P$ in $R$, i.e., the elements $r_i$ enumerate the roots $\alpha(a)$ from \eqref{eq:fpolynomial}, and likewise we denote by $Q_R\in R[X]$ the image of $Q$ in $R$. We now show by induction on $l\ge 1$ that the product $\prod_{i=1}^l (X-r_i)$ divides $Q_R$. The case $l=N$ is the desired claim. The base case $l=1$ claims that $X-r_1$ divides $Q_R$, i.e., that $Q_R(r_1)$ is zero. In other words, we have to show that $Q_R(\alpha(a)) = 0$ for any is a $p$-torsion point $a$ of $A^*$. This holds because all the $\alpha(a)$ are $p$-torsion points of $\mathbb G_E$ and $Q$ has the same zeroes as the $p$-series of $\mathbb G_E$. In the inductive step, we know that $Q_R=\prod_{i=1}^{l-1} (X-r_i)\cdot H$ in $R[X]$ for some $l\leq N$, and a suitable $H$. Since $Q_R(r_l)=0$, we obtain from this by evaluating at $r_l$ the relation $0=\prod_{i=1}^{l-1} (r_l-r_i)\cdot H(r_l)$ in $R$. Since the image of $P$ over the residue field $K$ is separable, none of the elements $r_l-r_i\in R$ reduces to zero in $K$, hence they do not lie in $\mathfrak m$, thus they are units in the local ring $R$. By cancelling these units we obtain $H(r_l)=0$, i.e., $H$ is divisible by $X-r_l$ in $R[X]$, and this completes the induction step.
\end{proof}

Now we obtain a relation between the ring $\pi_0\Phi^A\underline{E}$ which we want to show to be regular, and the ring $D$ which we know to be regular.

\begin{Prop}\label{prop:factor_geometric_and_drinfeld}
The canonical ring homomorphism $\pi_0E^{BA}\to\pi_0\Phi^A\underline{E}$ factors (uniquely) through the surjection $\rho\colon \pi_0E^{BA}\to D$, i.e., there is a commutative diagram
\[
\xymatrix{\pi_0E^{BA}\ar[r] \ar@{->>}[d]_{\rho} & \pi_0\Phi^A\underline{E} \\
D. \ar@{-->}[ru]}
\]
\end{Prop}

This says that the base change of the universal homomorphism $A^*\to\mathbb G_E$ to the geometric fixed points is a level-structure. We first check that \cref{prop:factor_geometric_and_drinfeld} implies our theorem.

\begin{proof}[Proof of \Cref{thm:regular} assuming \Cref{prop:factor_geometric_and_drinfeld}]
By \cref{rem:regularreduction}, we may assume that $G=A$ is an abelian $p$-group which is generated by (at most) $n$ elements. As a localization of the Noetherian ring $\pi_0E^{BA}$, the ring $\pi_0\Phi^{A}\underline{E}$ is also Noetherian. To see that it is regular, note that \cref{lem:geometric_fixed_points_as_open_subscheme} and \cref{prop:factor_geometric_and_drinfeld} say that the open immersion $X_A^\circ\hookrightarrow X_A$ factors through the closed immersion $\Spec(D)\subseteq X_A$. The resulting map $X_A^\circ\to \Spec(D)$ is necessarily an open immersion. Then, since $\Spec(D)$ is
regular, so is $X_A^\circ$ by \cite[Theorem 19.3]{Matsumura1989}, as claimed.
\end{proof}

\begin{proof}[Proof of \Cref{prop:factor_geometric_and_drinfeld}]
Set $I:=\ker(\rho)\subseteq \pi_0E^{BA}$; the claim is $I\cdot\pi_0\Phi^A\underline{E}=0$. It is enough to show this  after replacing $\pi_0\Phi^A\underline{E}$ with an arbitrary localization $R \coloneqq (\pi_0\Phi^A\underline{E})_{\mathfrak p}$ for a prime ideal $\mathfrak p\subseteq \pi_0\Phi^A\underline{E}$. To this end, we will verify that we are in the situation of \Cref{prop:divide_over_artin}, i.e., that the conditions of this proposition apply to the local ring $R=(R,\mathfrak pR,K:=R/\mathfrak p R)$:
    \begin{enumerate}
        \item We first prove that the image of $P$ is separable over $K$. Looking at the definition of $P$, it will be enough to show that the composite
        \[
        f\colon A^* \xrightarrow{\alpha} \mathbb{G}_E(\pi_0E^{BA}) \to  \mathbb{G}_E((\pi_0\Phi^A\underline{E})) \to  \mathbb{G}_E((\pi_0\Phi^A\underline{E})_{\mathfrak p}) \to  \mathbb G_E(K)
        \]
        is injective, where the last map is induced by the reduction $R =(\pi_0\Phi^A\underline{E})_{\mathfrak p} \to K$. To this end, assume otherwise. Then $f$ factors over some non-zero quotient of $A^*$, which implies that its classifying map $\widehat{f}\colon \pi_0E^{BA}\to K$ factors over some non-trivial restriction map $\pi_0E^{BA}\to \pi_0E^{BA'}$ with $A'\subseteq A$ a proper subgroup. This is a contradiction, because by construction $\widehat{f}$ factors over the geometric fixed points and hence inverts all non-zero Euler classes whereas the restriction map annihilates some non-zero Euler class. This contradiction shows that $f$ is indeed injective. 
        \item Secondly, we observe that the image of $P$ in $K$ divides the image of $Q$ because $Q$ annihilates all $p$-torsion points of $\mathbb{G}_E$. 
    \end{enumerate}
Therefore, \Cref{prop:divide_over_artin} implies that the image of $P$ in $R$ divides the image of $Q$ in $R$. Since $D$ was constructed to be initial among $\pi_0E^{BA}$-algebras with this property, this is equivalent to saying that $I\cdot R=0$, as claimed.
\end{proof}

We make the case of height one of the above results explicit and point out a possible sharpening of the regularity result \Cref{thm:regular}. 

\begin{Exa}\label{ex:X_A-cyclic-case}
Assume the height is one, then $E=E_1$ is (a form of) $p$-complete complex $K$-theory. Let $A=C_{p^n}$ be the cyclic group of order $p^n$ for some $n\ge 0$. Then $E^0=\mathbb Z_p$ and a familiar computation in complex oriented cohomology theories gives an isomorphism of $E^0$-algebras
\[ 
E^0(BC_{p^n})\simeq E^0 [\![ X ]\!] /\left( (X+1)^{p^n} -1 \right). 
\]
Here, $X$ corresponds to the Euler class of the defining representation $C_{p^n}\subseteq S^1$. It is convenient to set $Y\coloneqq X+1$. We can decompose 
\[ 
Y^{p^n}-1 =\prod\limits_{i=0}^n \Phi_{p^i}(Y) = (Y-1)\cdot\frac{Y^p-1}{Y-1}\cdot\ldots\cdot\Phi_{p^n}(Y)
\]
into irreducible polynomials over $\mathbb Q_p$. Then $\Phi_{p^i}$ is the $p^i$-th cyclotomic polynomial, the minimal polynomial of any primitive $p^i$-th root of unity $\zeta_{p^i}$. Each ring $\mathbb Z_p[Y]/(\Phi_{p^i})$ is a discrete valuation ring, the ring of integers in the local cyclotomic field $\mathbb Q_p(\zeta_{p^i})$, and has residue field $\Fp$. One can check that the obvious ring homomorphism
\[ 
E^0(BC_{p^n})=\mathbb Z_p[Y]/(Y^{p^n}-1)\to \prod\limits_{i=0}^n \mathbb Z_p[Y]/(\Phi_{p^i}(Y))
\]
glues the spectra of the rings on the right hand side along their common closed point and that one has  $X_{C_{p^i}}^\circ= \Spec(\mathbb Q_p(\zeta_{p^i}))$ for $i\ge 1$ and $X_{e}^\circ=\Spec(\mathbb Z_p)$. The Zariski spectrum of $E^0(BC_{p^n})$ is depicted in \cref{fig:spec(EBCpn)}.
\begin{figure}[!ht]
    \centering
    \begin{tikzcd}[column sep=small]
  &       &         & \arrow[dll,dash]\arrow[dlll,dash]\arrow[ld,dash]\arrow[dll,dash]\textcolor{blue}{\bullet}_{\Fp} \arrow[dr,dash]\arrow[drr,dash] &        &         & \\
 \textcolor{blue}{\bullet}_{\mathbb{Q}_p} & \textcolor{blue}{\bullet}_{\mathbb{Q}_p(\zeta_{p})} & \textcolor{blue}{\bullet}_{\mathbb{Q}_p(\zeta_{p^2})} & \ldots & \textcolor{blue}{\bullet}_{\mathbb{Q}_p(\zeta_{p^{n-1}})} & \textcolor{blue}{\bullet}_{\mathbb{Q}_p(\zeta_{p^n})}
\end{tikzcd}
\caption{An image of $\Spec(E^0(BC_{p^n}))$. Each prime is represented by its residue field, and the lines indicate specialization relations with closure going upwards.}
\label{fig:spec(EBCpn)}
\end{figure}
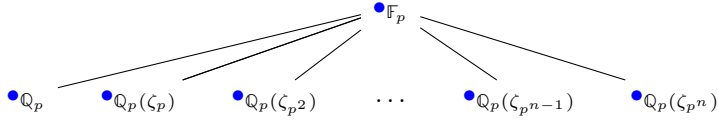
\end{Exa}

\begin{Rem}
We now turn to discuss a conceivable sharpening of \Cref{thm:regular}; this remark can be skipped without loss of continuity. Let again $E$ denote a Lubin--Tate theory of arbitrary height and $A$ a finite abelian $p$-group. We have a ring extension
\[ 
R\coloneqq E^0\to S\coloneqq \Phi^A(\underline{E}).
\]
Fix a prime ideal $\mathfrak q\subseteq S$, and denote $\mathfrak p\coloneqq \mathfrak q\cap R\subseteq R$. The resulting map
\[ 
R_{\mathfrak p}\to S_{\mathfrak q}
\] 
is a finite flat extension of local Noetherian rings, $R_{\mathfrak p}$ is regular, and we showed that $S_{\mathfrak{q}}$ is regular, too. We can consider the condition
\begin{equation}\label{eq:regularityquestion}
    \mathfrak q=\mathfrak p\cdot S_{\mathfrak q}.
\end{equation}
Note this condition implies the regularity of $S_{\mathfrak{q}}$ and more precisely it shows that any regular system of parameters in $R_{\mathfrak p}$ gives a regular system in $S_{\mathfrak{q}}$. The condition \eqref{eq:regularityquestion} is immediate to check in height one and can be checked in height two using computations about modular equations. We have not been able to decide whether \eqref{eq:regularityquestion} holds in a single case of height larger than or equal to three, so we will leave this as an open question. Observe that the condition in \eqref{eq:regularityquestion} does not claim that $R_{\mathfrak p}\to S_{\mathfrak q}$ is \'etale: Indeed, it is known that the \'etale locus of the extension $E^0\to E^0(BA)$ is exactly $E^0[p^{-1}]$, so if \eqref{eq:regularityquestion} holds at a prime ${\mathfrak p}$ of positive characteristic, as it does in height two, then necessarily the residue field extension will be inseparable.
We remind the reader of the reason for the claim about the \'etale locus of the finite flat algebra $E^0\subseteq E^0(BA)$: The height stratification of the corresponding finite flat group scheme is given by a regular sequence $v_0=p, v_1,\ldots, v_{n-1}\in\pi_0E$. In particular, the formal part is zero (equivalently, the group scheme is \'etale), exactly if $p$ is invertible.
\end{Rem}

\subsection{Stratification for \texorpdfstring{$\underline{E}$}{E}-modules}
In this section we show that for any finite group $G$ the category $\Modd{G}(\underline{E})$ is stratified. The key input is the stratification of $\Mod(\Phi^A(\underline{E}))$, which uses the regularity of the geometric fixed points established in \Cref{sec:regularity}. 
\begin{Lem}\label{lem:gfp_balmer}
Let $A$ be a finite abelian $p$-group, then $\Mod(\Phi^{A}\underline{E})$ is cohomological stratified and the ungraded comparison map
\[
\xymatrix{\rho_0 \colon \Spc(\Perf(\Phi^A\underline{E})) \ar[r]^-{\cong} &\Spec(\pi_0(\Phi^A\underline{E}))}
\]
is a homeomorphism.
\end{Lem}
\begin{proof}
By \Cref{thm:regular}, $\pi_0(\Phi^A\underline{E})$ is a regular Noetherian ring and the graded ring $\pi_*(\Phi^A(\underline{E}))$ is concentrated in even degrees. In fact, this ring is even periodic (as being obtained by inverting finitely many elements in the even periodic ring $E^*(BA)$), and hence  $\pi_*(\Phi^A\underline{E})$ is regular as a \emph{graded} commutative ring. Now we can apply \cite[Theorem 1.1 and Theorem 1.3]{DellAmbrogioStanley16}\footnote{Note that there is a gap in one of the lemmas of that paper, which however does not affect the main theorems, see \cite{DellAmbrogioStanley16erratum}.}, keeping in mind \Cref{rem:comparison_map}. 
\end{proof}

We now use \cref{lem:gfp_balmer} to prove that $\Modd{G}(\underline{E})$ is stratified for any finite group $G$. In the next section, we will then identify this Balmer spectrum with the Zariski spectrum of the cohomology ring and deduce that $\Mod_G(\underline{E})$ is in fact cohomologically stratified.

\begin{Thm}\label{thm:stratification_finite_g}
For any group $G$ and any Lubin--Tate theory $E$, the category $\Modd{G}(\underline{E})$ is stratified with Noetherian spectrum $\Spc(\Perff{G}(\underline{E}))$. 
\end{Thm}
\begin{proof}
We recall that $\underline{E}$ is $\mathcal{F}$-nilpotent for the family of abelian $p$-groups which are generated by $n$ elements (\Cref{ex:E-nilpotent}). Let $A$ be such a group, then it suffices by \Cref{thm:eqstratdescent} to show that $\Spc(\Perf(\Phi^A\underline{E}))$ is Noetherian and $\Mod(\Phi^A\underline{E})$ is stratified. This follows from \Cref{lem:gfp_balmer}.
\end{proof}

\section{The Balmer spectrum for Borel-equivariant \texorpdfstring{$E$}{E}-theory}\label{sec:spectrum-borel-e}

In this section we compute the spectrum of $\Perf_G(\underline{E})$ for any finite group $G$ and any Lubin--Tate theory $E$ of height $n$ and prime $p$, by showing that the comparison map $\rho_0\colon\Spc(\Perf_G(\underline{E})) \to \Spec(E^0(BG))$ is a homeomorphism. Along with \Cref{thm:stratification_finite_g}, this will show that $\Modd{G}(\underline{E})$ is cohomologically stratified. Finally, we deduce that the category of modules over the cochain spectrum $E^{BG}$ is also cohomologically stratified. 

\subsection{The spectrum of Borel-equivariant $E$-theory}

In order to compute the spectrum of $\Perff{G}(\underline{E})$, we first show that this category is Noetherian (\Cref{def:noeth-tt-cat}). 

\begin{Lem}\label{lem:e-theory-end-finite}
For any finite group $G$, the rings $E^0(BG)$ and $E^*(BG)$  are Noetherian. Moreover for any subgroup $H \subseteq  G$, the ring maps $E^0(BG)\to E^0(BH)$ and $E^*(BG)\to E^*(BH)$ are finite. In particular, the category $\Perff{G}(\underline{E})$ is Noetherian.
\end{Lem}

\begin{proof}
The ring $E^*(BG)$ is graded Noetherian and the composite 
\[
E^* \to E^*(BG) \to E^*(BH),
\]
is finite by~\cite[Corollary 4.4]{greenleesstrickland_varieties}. This also implies that $E^*(BG) \to E^*(BH)$ is finite. Passing to degree zero elements and using that $E^*$ is concentrated in even degrees and even periodic, we deduce that the composite $E^0 \to E^0(BG)\to E^0(BH)$ is finite. It then follows that $E^0(BG)\to E^0(BH)$ is finite and that $E^0(BG)$ is Noetherian. The final claim then follows from \cref{lem:abstract_end_finite}. 
\end{proof}

\begin{Lem}\label{lem:Balmer-spectrum-E_A}
For any finite abelian $p$-group $A$ which can be generated by $n$ elements, the ungraded comparison map
\[
\xymatrix{\rho^A_0 \colon \Spc(\Perff{A}(\underline{E})) \ar[r]^-{\cong} & \Spec(E^0(BA))}
\]
is a homeomorphism.
\end{Lem}
\begin{proof}
Let $\Phi \colon \Perf_A(\underline{E}) \to \prod_{A' \subseteq A}\Perf(\Phi^{A'}\underline{E})$ denote the symmetric monoidal functor induced by the product of the geometric fixed point functors $\Phi^{A'}$. By naturality of the comparison map \cite[Corollary 5.6]{Balmer10b}, there is a commutative diagram
\[\begin{tikzcd}
	{\Spc(\Perf_A(\underline{E}))} & {\bigsqcup_{A' \subseteq A} \Spc(\Perf(\Phi^{A'}\underline{E}))} \\
	{\Spec(E^0(BA))} & {\bigsqcup_{A' \subseteq A}\Spec(\pi_0(\Phi^{A'}\underline{E})).}
	\arrow["{\rho_0^A}"', from=1-1, to=2-1]
	\arrow["\bigsqcup\rho_0", from=1-2, to=2-2]
	\arrow["{\Spc(\Phi)}"', from=1-2, to=1-1]
	\arrow["{\Spec(\Phi_0)}", from=2-2, to=2-1]
\end{tikzcd}\]
We claim that $\rho^A_0$ is a bijection. By \Cref{lem:gfp_balmer}, the right hand vertical arrow is a homeomorphism. Using \cref{lem:geometric_fixed_points_as_open_subscheme}, we see that the bottom map is equivalent to the map 
\[
\bigsqcup_{A' \subseteq A}X_{A'}^{\circ}\cong{\bigsqcup_{A' \subseteq A}\Spec(\pi_0(\Phi^{A'}\underline{E}))} \to \Spec(E^0(BA)) = X_A.
\]
By \Cref{prop:abecohomdecomposition} this map is a bijection. It follows from \cref{cor:eqspc_surjectivity} that $\Spc(\Phi)$ is surjective, hence the commutativity of the square implies that $\rho^A_0$ is a bijection. Finally, by \Cref{lem:e-theory-end-finite} and \Cref{prop:bijectioniffhomeomorphism}, this shows that $\rho_0^A$ is a homeomorphism. 
\end{proof}

\begin{Lem}\label{lem:generalspc}
For any finite group $G$, the comparison map
\[
\xymatrix{\rho_0^G\colon \Spc(\Perff{G}(\underline{E})) \ar[r]^-{\cong} & \Spec(E^0(BG))}
\]
is a homeomorphism. 
\end{Lem}
\begin{proof}
This follows from \Cref{cor:quillen_stratification}$(a)$. Indeed, $\underline{E}$ is $\mathcal{F}$-nilpotent for the family of abelian $p$-subgroups of $G$ which are generated by (at most) $n$ elements, see~\Cref{ex:E-nilpotent}. Then the assumptions of \Cref{cor:quillen_stratification}$(a)$ hold by \Cref{lem:Balmer-spectrum-E_A} and \cref{lem:e-theory-end-finite}.
\end{proof}

\begin{Thm}\label{thm:etheorycohomstratification}
For any finite group $G$ and any Lubin--Tate theory $E$, the category $\Modd{G}(\underline{E})$ is cohomologically stratified by $E^0(BG)$. 
\end{Thm}
\begin{proof}
Combine \Cref{thm:stratification_finite_g} and \Cref{lem:generalspc}. 
\end{proof}

\begin{Rem}\label{rem:proofcomparison}
\Cref{thm:etheorycohomstratification} provides a chromatic counterpart in intermediate characteristic to the celebrated stratification theorem of Benson--Iyengar--Krause \cite{BensonIyengarKrause11a}, extending their work from height $\infty$ to all finite heights. The strategy of proof, however, is fundamentally different: we first establish stratification relative to the Balmer spectrum of $\Modd{G}(\underline{E})$ and then lift our chromatic Quillen stratification to an identification of the Balmer spectrum. As mentioned, the fundamental input to our approach is a regularity result for the geometric fixed points of $\underline{E}$ via Drinfel'd level structures, while the proof of \cite{BensonIyengarKrause11a} ultimately relies on the Bernstein--Gelfand--Gelfand(BGG)-correspondence.
\end{Rem}

As a corollary, we obtain a generalization of \cref{prop:abecohomdecomposition}.

\begin{Cor}\label{cor:decomp-specEBG}
    Let $E$ be a Lubin--Tate theory at height $n$ and prime $p$. For any finite group $G$, we have a disjoint decomposition into locally closed subsets
    \[
     \Spec(E^0(BG))= \bigsqcup_{A} \Spec(\pi_0(\Phi^A \underline{E}))/W^{Q}_G(A),
    \]
    where the coproduct ranges over all $G$-conjugacy classes of abelian $p$-subgroups of $G$ which are generated by (at most) $n$ elements.
\end{Cor}

\begin{proof}
    Recall from \cref{ex:borel-are-global} that $\underline{E}$ arises from a global spectrum, and that it is $\mathcal{F}$-nilpotent for the family of abelian $p$-subgroups of $G$ which are generated by $n$ elements by \cref{ex:E-nilpotent}. The result then follows from~\cref{thm:globalquillenstrat}, \cref{lem:gfp_balmer}, and \cref{lem:generalspc}.
\end{proof}

\begin{Rem}\label{rem:genpermmodules}
Let $R$ be a commutative ring spectrum and $G$ a finite group. By \cite[Corollary 6.21]{MathewNaumannNoel17}, Borel-completion is an exact symmetric monoidal and fully faithful functor
\[
\xymatrix{\psi(R)\colon\Perff{G}(\underline{R}) \ar[r] & \Fun(BG,\Perf(R)),}
\]
whose essential image is given by the thick subcategory generated by the permutation modules $\{R[G/H]\}_{H\subseteq G}$. We say that $\Fun(BG,\Perf(R))$ is \emph{generated by permutation modules} if $\psi(R)$ is an equivalence. For example, for a discrete commutative ring $A$, the tt-category $\Fun(BG,\Perf(A))$ is equivalent to the bounded derived category of $A[G]$-representations whose underlying $A$-module is perfect. A result of Rouquier, Mathew~\cite[Theorem A.4]{treumannmathew_reps}, and Balmer--Gallauer \cite[Theorem~1.4]{BalmerGallauer2022} then implies that $\psi(A)$ is an equivalence for any regular Noetherian $A$. It is an open question, first raised in \cite[Question A.2]{treumannmathew_reps}, whether $\Fun(BG,\Perf(E))$ is generated by permutation modules for any Lubin--Tate theory $E$.
This is answered affirmatively by Mathew in loc. cit. for the case of height $1$ and $G=C_p$.
\end{Rem}

\subsection{Stratification for cochains}
In this subsection, we study the category of modules over the cochain algebra $E^{BG}\in \CAlg(\Sp)$. We will deduce from \cref{thm:etheorycohomstratification} that $\Mod(E^{BG})$ is cohomologically stratified. We start off with the following observation.

\begin{Lem}\label{lem:endo}
The endomorphism ring of the unit $\underline{E}$ in $\Modd{G}(\underline{E})$ can be identified with $E^{BG}$ as a commutative ring spectrum. 
\end{Lem}

\begin{proof}
By adjunction, we have equivalences of commutative algebras 
\[
\Map_{\Mod_G(\underline{E})}(\underline{E}, \underline{E})\simeq \Map_{\Sp_G}(S^0, \underline{E})\simeq (\underline{E})^G
\]
so we only need to identify the right hand side with the spectrum $E^{BG}$. To this end recall that 
$\underline{E}$ belongs to the full subcategory $(\Sp_G)_{\mathrm{Borel}}\subseteq\Sp_G$ of Borel-equivariant  
$G$-spectra as defined in~\cite[Definition 6.14]{MathewNaumannNoel17}. Moreover by~\cite[Proposition 6.19]{MathewNaumannNoel17}, there is a commutative square 
\[
\begin{tikzcd}
(\Sp_G)_{\mathrm{Borel}} \arrow[d,"\sim", "\otimes"'] \arrow[r, hook] & \Sp_G \arrow[d,"(-)^G"] \\
\Fun(BG, \Sp) \arrow[r,"(-)^{hG}"'] & \Sp,
\end{tikzcd}
\]
where the left vertical functor takes the associated underlying spectrum with $G$-action, and this is a symmetric monoidal equivalence by~\cite[Proposition 6.17]{MathewNaumannNoel17}. Note that all the other functors in the diagram are lax monoidal so they preserve commutative algebra objects. Chasing $\underline{E}$ around the diagram, we obtain the identification $(\underline{E})^G\simeq E^{hG}$ as commutative algebras. It is only left to note that $E^{hG} \simeq E^{BG}$, as $E$ has trivial $G$-action.
\end{proof}

\begin{Ter}
Given a rigidly-compactly generated tt-category $\cat T$ we let 
\[
\cat T_{\cell} \coloneqq \Loc^{\cat T}{\langle \unit \rangle} \subseteq \cat T
\]
denote the localizing subcategory of $\cat T$ generated by the unit, referred to as the full subcategory of \emph{cellular objects} in $\cat T$.
\end{Ter}

\begin{Rem}\label{rem:compact_objects}
Because the unit is compact in $\cat T$ by assumption, Neeman's theorem \cite[Theorem 2.1]{Neeman96} implies that $\left( \cat T_{\cell}\right)^c = \cat T_{\cell} \cap \cat T^c.$
\end{Rem}

The next result identifies the category of modules over the cochain algebra as the subcategory of cellular objects in $\underline{E}$-modules.

\begin{Lem}\label{lem:cell-objs-cochain}
There is a fully faithful tt-functor 
\[
\Mod(E^{BG})  \hookrightarrow \Modd{G}(\underline{E}) 
\]
with essential image given by $\Loc^{\Modd{G}(\underline{E})}{\langle\underline{E}\rangle}$.
\end{Lem}
\begin{proof}
By Morita theory and \Cref{lem:endo}, there is a cocontinuous symmetric monoidal functor 
\[
\Mod(E^{BG})  \xrightarrow{\sim} \Loc^{\Modd{G}(\underline{E})}{\langle\underline{E}\rangle} \subseteq \Modd{G}(\underline{E}).
\]
The first functor is an equivalence, and it is symmetric monoidal by \cite[Corollary 4.8.1.14]{HALurie}.
\end{proof}

\begin{Thm}\sloppy
The category $\Mod(E^{BG})$ is cohomologically stratified. In particular, the comparison map 
\[
\xymatrix{\rho_0^G\colon\Spc(\Perf(E^{BG})) \ar[r]^-{\cong} & \Spec(E^0(BG))}
\]
is a homeomorphism.
\end{Thm}

\begin{proof}
By \Cref{lem:cell-objs-cochain} we can identify $\Mod(E^{BG})$ with the subcategory of cellular objects in $\Modd{G}(\underline{E})$. Applying \cite[Theorem 1.2]{BensonIyengarKrause2011Localising} and noting that $\Modd{G}(\underline{E})$ is cohomologically stratified (\Cref{rem:cohomstrat}), we see that stratification for $\Modd{G}(\underline{E})$ implies stratification for $\Mod(E^{BG})$.
\end{proof}

\subsection{Some calculations of $\Spec(E^0(BG))$ at height one}\label{subsection-height-one}
Throughout this section we will assume that the height is one so that $E=E_1$ is (a form of) $p$-complete complex $K$-theory. We will describe $\Spec(E^0(BG))$ explicitly for a large class of finite groups $G$.  

Consider a finite group $G$ and let $\mathcal{C}$ denote the family of cyclic $p$-subgroups of $G$. The Borel-equivariant $G$-spectrum $\underline{E}$ is $\mathcal{C}$-nilpotent by \Cref{ex:E-nilpotent} and arises from a global homotopy type by \cref{ex:borel-are-global}. It follows from \Cref{thm:globalquillenstrat} that we have a homeomorphism
\begin{equation}\label{colim-spec-E}
\Spec(E^0(BG))\cong \mathop{\colim}_{A \in \mathcal{O}^Q_{\mathcal{C}}(G)} \Spec(E^0(BA)).
\end{equation}
The Zariski spectrum $\Spec(E^0(BA))$ has already been discussed in detail in \cref{ex:X_A-cyclic-case}.
We make the following additional observations:
\begin{itemize}
\item[(1)] For $n\geq m \geq 1$, the canonical inclusion $\Spec(E^0(BC_{p^{m}}))\subseteq \Spec(E^0(BC_{p^n}))$ hits the prime ideals corresponding to $\Fp$ and $\mathbb{Q}_p(\zeta_{p^i})$ for $0 \leq i\leq m$.
\item[(2)] The Quillen--Weyl group $W_G^Q(C_{p^n})$ acts trivially on $\Spec(E^0(BC_{p^n}))$. To see this one checks that any automorphism $\sigma$ of $C_{p^n}$ preserves the decomposition (\ref{eq:disjoint_union_set_spec}), namely satisfies $\sigma.X_{C_{p^i}}^\circ= X_{C_{p^i}}^\circ$ for all $i \leq n$. Since $X_{e}^\circ=\Spec(\Z_p)$ and $X_{C_{p^i}}^\circ=\Spec(\mathbb{Q}_p(\zeta_{p^i}))$ by \cref{ex:X_A-cyclic-case}, one immediately deduces that $\sigma$ must act trivially on each $X_{C_{p^i}}^\circ$ and so on all $\Spec(E^0(BC_{p^n}))$.
\end{itemize}
Using these observations together with \eqref{colim-spec-E} we can calculate $\Spec(E^0(BG))$ for many examples of $G$.

\begin{Exa}\label{ex:exp-p}
Let $G$ be a finite group with a $p$-Sylow subgroup $G_p$ of exponent $p$. This means that $p$ is the least common multiple of the orders of all elements of $G_p$. For instance, $G_p$ could be cyclic of order $p$ or an elementary abelian $p$-group. Our assumptions force any $A\in\mathcal{C}$ to be cyclic of order $p$. In this case the Zariski spectrum of $E^0(BG)$ is depicted in \cref{fig:spec(EBG)}.
\begin{figure}[!ht]
    \centering
    \begin{tikzcd}[column sep=small]
  &       &         & \arrow[dll,dash]\arrow[dlll,dash]\arrow[ld,dash]\arrow[dll,dash]\textcolor{blue}{\bullet} \arrow[dr,dash]\arrow[drr,dash] &        &         & \\
 \textcolor{blue}{\bullet} & \textcolor{blue}{\bullet} & \textcolor{blue}{\bullet} & \ldots & \textcolor{blue}{\bullet} & \textcolor{blue}{\bullet}
\end{tikzcd}
\caption{An image of $\Spec(E^0(BG))$ for a finite group $G$ with $p$-Sylow subgroup of exponent $p$. The lines indicate specialization relations with closure going upwards. The bottom row has $n+1$ prime ideals where $n$ is the maximal number of cyclic subgroups of $G$ of order $p$ which are pairwise non-conjugate.}
\label{fig:spec(EBG)}
\end{figure}
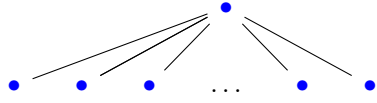
\end{Exa}

\begin{Exa}
Suppose that $p=2$ and $G=D_n$ a dihedral group of order $2n$ for some $n\geq 1$. One checks that there is a unique maximal cyclic subgroup $A\subseteq G$ which has order $n$. Therefore we immediately get that $\Spec(E^0(BG))\cong \Spec(E^0(BA))$. In this case the Zariski spectrum is depicted in \cref{fig:spec(EBCpn)}.
\end{Exa}

\begin{Rem}\label{rem:non-free-action}
  Inspired by Quillen's work~\cite{Quillen71}, one might wonder if the action of $W_G^{Q}(A)$ on $\Spec(\pi_0\Phi^A \underline E)$ is free on closed points, see discussion in~\cref{rem:freeness}. To see this is not the case take $p=3$, $G=\Sigma_3$ the symmetric group on three letters, and $A=\mathbb{Z}/3$ the Sylow $3$-subgroup generated by the cycle $(123)$. The Quillen--Weyl group $W_G^{Q}(A)$ is not trivial but acts trivially on all of $\Spec(E^0(BA))$ and so on $\Spec(\pi_0\Phi^A\underline{E})$ too, by observation (2) above. 
\end{Rem}

\section{Further examples}\label{sec:examples}

In this final section, we discuss some further examples for which we can prove stratification for equivariant module spectra as a consequence of our general techniques, namely for $H\underline{\bbZ}$-modules when $G = C_{p^n}$, equivariant complex $K$-theory $KU_G$ for all finite groups $G$, and the real $K$-theory spectrum $K\mathbb{R}$ when $G = C_2$. 

\subsection{Stratification for modules over the integral constant Mackey functor}
We recall that \Cref{thm:quillen_decomposition} determines the spectrum of $\Perf_G(R)$ in terms of the spectra of the non-equivariant tt-categories $\Perf(\Phi^H R)$, up to the question how the strata are glued together. In order to settle this question, different techniques are required, such as a study of the Tate squares governing the local-to-global principles. This ambiguity is illustrated by the next example.

\begin{Exa}\label{ex:constantgreen}
Let $G = C_{p^n}$ be a cyclic $p$-group and write $H\underline{\bbZ} \in \CAlg(\Sp_{C_{p^n}})$ for the Eilenberg--MacLane spectrum associated to the constant Green functor with value $\bbZ$. By \cite[Proposition 3.18]{HillHopkinsRavenel16} for $p = 2$ and \cite[p.~34]{zeng_comps} for $p>2$, its geometric fixed points have coefficients
\[
\pi_*\Phi^HH\underline{\bbZ} \cong
\begin{cases}
    \bbZ & \text{if } H=e, \\
    \bbZ/p[t] & \text{non-trivial } H \subseteq  C_{p^n},
\end{cases}
\]
as graded commutative rings, where $t$ is in degree 2. In particular, all geometric fixed points of $H\underline{\bbZ}$ are even and regular. It then follows from \cite{DellAmbrogioStanley16} that the tt-categories $\Mod(\Phi^HH\underline{\bbZ})$ are stratified with Balmer spectrum
\[
\Spc(\Perf(\Phi^HH\underline{\bbZ})) \cong
\begin{cases}
    \Spec(\bbZ) & \text{if } H=e, \\
    \Spec^h(\bbZ/p[t]) & \text{non-trivial } H \subseteq  C_{p^n}.
\end{cases}
\]
Therefore, \cref{thm:eqstratdescent} and \cref{thm:quillen_decomposition} imply that $\Modd{C_{p^n}}(H\underline{\bbZ})$ itself is stratified, and we depict its Balmer spectrum in \cref{spc:constantgreen}.
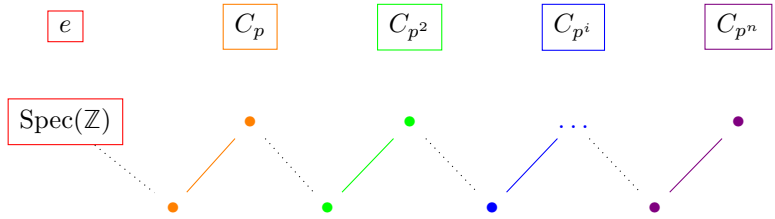
\begin{figure}[!ht]
    \centering
    \begin{tikzcd}[column sep=small]
  |[draw=red, rectangle]| e & & |[draw=orange, rectangle]| C_p & &  |[draw=green, rectangle]| C_{p^2}& & |[draw=blue, rectangle]| C_{p^i}& & |[draw=violet, rectangle]| C_{p^n}\\
 |[draw=red, rectangle]|  \Spec(\mathbb{Z}) \arrow[dr, dotted, dash] &    & \textcolor{orange}{\bullet} \arrow[dr, dotted, dash] & & \textcolor{green}{\bullet} \arrow[dr, dotted, dash] & & \textcolor{blue}{\ldots} \arrow[dr, dash, dotted] & & \textcolor{violet}{\bullet}   \\
                            & \textcolor{orange}{\bullet}  \arrow[ur, dash, orange] &                     & \textcolor{green}{\bullet}  \arrow[ur, dash, green]&         & \textcolor{blue}{\bullet} \arrow[ur, dash,  blue] & & \textcolor{violet}{\bullet}  \arrow[ur,dash, violet]& 
\end{tikzcd}
    \caption{An image of $\Spc(\Perff{C_{p^n}}(H\underline{\bbZ}))$. The lines indicate specialization relations, with closure going upwards. Contributions from the same subgroup are displayed in the same color.} \label{spc:constantgreen}
\end{figure}
The solid part of this figure is determined by \cref{thm:quillen_decomposition} together with the observation that the corresponding Weyl groups act trivially on the geometric fixed points in this case, since $H\underline{\mathbb{Z}}$ arises from a global homotopy type. Consequently, we obtain a parametrization of localizing ideals in this setting: 
\[
\begin{Bmatrix}
\text{Localizing $\otimes$-ideals} \\
\text{of $\Modd{C_{p^n}}(H\underline{\bbZ})$} 
\end{Bmatrix} 
\xymatrix@C=2pc{ \ar[r]^{\sim} &  }
\begin{Bmatrix}
\text{Subsets of $\Spec(\bbZ) \sqcup \bigsqcup_{i=1}^n \Spec^h(\bbZ/p[t])$}
\end{Bmatrix}.
\]
The dashed lines in the figure---gluing the strata of the spectrum together---are known to exist by work of Balmer and Gallauer, see \cite{BalmerGallauer_announcement}. In fact, they further determine the underlying set of $\Spc(\Perff{G}(H\underline{\bbZ}))$ for all finite groups $G$ and establish stratification of $\Modd{G}(H\underline{\bbZ})$ in this generality.
\end{Exa}

\subsection{Stratification for modules over equivariant $K$-theory}
The goal of this subsection is to use our methods to show that the category $\Modd{G}(KU_G)$ is cohomologically stratified, where $KU_G$ denotes \emph{$G$-equivariant complex $K$-theory}, \cite{Segal1968Equivariant}. Our argument is modelled on the case of Borel-equivariant Lubin--Tate theory treated in previous sections. We will also show cohomological stratification for the category of modules over the $C_2$-spectrum of real $K$-theory $K\mathbb{R}$. 

\begin{Rem}
The $G$-equivariant stable homotopy groups are
\[
\pi_*^G(KU_G) \cong R(G)[\beta^{\pm 1}], \quad |\beta| = -2,
\]
where $R(G)$ denotes the complex representation ring of $G$. In particular, $\pi_*^G(KU_G)$ is an even periodic Noetherian ring. Using the fact that these rings are even periodic, we deduce the following from Segal's seminal work on $\pi_0^G(KU_{G})$.
\end{Rem}

\begin{Lem}\label{lem:end-finiteness-kug}
For any finite group $G$, the category $\Perff{G}(KU_G)$ is Noetherian. 
\end{Lem}
\begin{proof}
For any finite group $G$, the ring $\pi_*^G(KU_G)$ is Noetherian. Moreover, if $H$ is a subgroup of $G$, then $\pi_*^H(KU_G) \cong \pi_*^H(KU_H)$ is a finitely generated $\pi_*^G(KU_G)$-module. Both of these statements were proven in \cite[Proposition 3.2 and Corollary 3.3]{Segal68a}. The result then follows from \cref{lem:abstract_end_finite}. 
\end{proof}

We will also need the following result, proved in \cite[Proposition 5.6]{MathewNaumannNoel2019}. 

\begin{Lem}\label{lem:nilpotence-kug}
The $G$-spectrum $KU_G$ has derived defect base equal to the family of cyclic subgroups of $G$.
\end{Lem}

\begin{Rem}
By equivariant Bott periodicity, $KU_G$ is \emph{complex stable} (\cite[Section 4]{Greenlees1999Equivariant}): for every complex representation $V$, there are compatible isomorphisms
\[
\pi^G_{k + |V|}(S^V \wedge X) \cong \pi^G_k(X).
\]
For such complex stable equivariant theories, geometric fixed points are given by inverting suitable classes in {\em integer} degrees, and this is the key ingredient in the computational part of the following result.
\end{Rem}

\begin{Lem}\label{lem:gfp_equivaiant_k-theory}
Let $G$ be a finite group. The homotopy of the $G$-geometric fixed points of $KU_G$ are given by
\[
\pi_*(\Phi^GKU_G) = \begin{cases}
    \pi_*(KU)[1/n,\zeta_n] & G \cong C_n\\
    0 &  \text{otherwise}, 
\end{cases}
\]
where $\zeta_n$ denotes a primitive $n$-th root of unity. In particular, the ring $\pi_*(\Phi^G(KU_G))$ is even periodic, regular and Noetherian.
\end{Lem}
\begin{proof}
The nilpotence result of \Cref{lem:nilpotence-kug} implies that $\Phi^GKU_G = 0$ unless $G \cong C_n$ (\Cref{lem:gfp_vanishing}). In the case $G$ is a cyclic group, the computation follows from \cite[Proposition 7.7.7]{tomDieck1979} (see also \cite[Lemma 3.1]{Greenlees1999Equivariant}). For the regularity claim, we use the explicit description observing that $\mathbb Z[\zeta_n]\subseteq\mathbb Q(\zeta_n)$ is the full ring of integers, in particular it is a Dedekind ring, and thus regular.
\end{proof}

\begin{Prop}\label{prop:stratification_ku_g_gfp}
For any finite group $G$, the category $\Modd{}(\Phi^GKU_G)$ is cohomologically stratified. In particular, the comparison map 
\[
\xymatrix{\rho_0^G \colon \Spc(\Perf(\Phi^GKU_G))\ar[r]^-{\cong} &\Spec(\pi_0(\Phi^GKU_G))}
\] 
is a homeomorphism. 
\end{Prop}
\begin{proof}
By \Cref{lem:gfp_equivaiant_k-theory}, the ring $\pi_*(\Phi^GKU_G)$ is even periodic, regular and Noetherian.
 Now apply \cite[Theorem 1.1 and 1.3]{DellAmbrogioStanley16} noting that we can use the spectrum of degree zero elements instead of the graded homogeneous spectrum (see \Cref{rem:comparison_map}). 
\end{proof}

\begin{Thm}\label{thm:kug_strat}
For any finite group $G$, the $tt$-category $\Modd{G}(KU_G)$ is stratified. 
\end{Thm}
\begin{proof}
Note that the restriction of $KU_G$ to an $H$-spectrum is $KU_H$. By \Cref{thm:eqstratdescent}, the theorem is then a consequence of \Cref{prop:stratification_ku_g_gfp}, since $\Phi^HKU_G \simeq \Phi^HKU_H$ for any subgroup $H$ in $G$. 
\end{proof}

\begin{Rem}
It remains to show how the Balmer spectrum in \Cref{thm:kug_strat} is determined by the representation ring of $G$. The next result resolves this in the case of a cyclic group.
\end{Rem}

\begin{Lem}\label{lem:kug_cyclic_balmer}
For any cyclic group $C_n$, the comparison map
\[
\xymatrix{\rho_0^{C_n}\colon \Spc(\Perff{C_n}(KU_{C_n})) \ar[r]^-{\cong} & \Spec(R(C_n))}
\]
is a homeomorphism. 
\end{Lem}
\begin{proof}
Arguing as in \Cref{lem:Balmer-spectrum-E_A}, we consider the commutative diagram
\[\begin{tikzcd}
	{\Spc(\Perf_{C_n}(KU_{C_n}))} & {\bigsqcup_{A \subseteq C_n} \Spc(\Perf(\Phi^{A}(KU_{C_n})))} \\
	{\Spec(R(C_n))} & {\bigsqcup_{A \subseteq C_n}\Spec(\pi_0(\Phi^{A}(KU_{C_n}))).}
	\arrow["{\rho_0^{C_n}}"', from=1-1, to=2-1]
	\arrow["\bigsqcup\rho_0", from=1-2, to=2-2]
	\arrow["{\Spc(\Phi)}"', from=1-2, to=1-1]
	\arrow["{\Spec(\Phi_0)}", from=2-2, to=2-1]
\end{tikzcd}\]
We first prove that the bottom map is bijective. This can be deduced from work of Segal and Bojanowska for general finite groups, see  \cite[Proposition 3.7]{Segal68a} and \cite[Theorem 4.13]{Bojanowska1983spectrum}.
But the special case at hand of a cyclic group admits the following direct argument. Our map is induced by the ring homomorphism
\[ 
\pi \colon R(C_n)\cong \mathbb Z[X]/(X^n-1)\to \prod\limits_{1\leq d\mid n}\mathbb Z[\frac 1d,\zeta_d]\cong \prod\limits_{A\subseteq C_n}\pi_0(\Phi^A(KU_{C_n}))
\]
sending $X$ to the tuple $(\zeta_d)_{d\mid n}$ of $d$-th primitive roots of unity. To say this induces a bijection on Zariski spectra is equivalent to saying that 
every ring homomorphism $\omega\colon\mathbb Z[X]/(X^n-1)\to \Omega$
into an algebraically closed field $\Omega$ factors uniquely through exactly one of the component maps $\pi_d \colon \mathbb Z[X]/(X^n-1)\to \mathbb Z[\frac 1d,\zeta_d]$ of $\pi$. Indeed, denoting 
$x \coloneqq \omega(X)\in\Omega^*$, the map $\omega$ factors uniquely through $\pi_d$ for 
$d$ the order of $x\in\Omega^*$.

It follows that the bottom map in the commutative diagram is a bijection. Moreover, the right-hand vertical map is a homeomorphism by \Cref{prop:stratification_ku_g_gfp}, while the top horizontal map is a surjection by \Cref{cor:eqspc_surjectivity}. It follows that $\rho_0^{C_n}$ is a bijection, and hence a homeomorphism by \Cref{prop:bijectioniffhomeomorphism}, which is applicable by \Cref{lem:end-finiteness-kug}. 
\end{proof}

\begin{Lem}\label{lem:generalspc_ku_g}
For any finite group $G$, the comparison map
\[
\xymatrix{\rho_0^G\colon \Spc(\Perff{G}(KU_G)) \ar[r]^-{\cong} & \Spec(R(G))}
\]
is a homeomorphism. 
\end{Lem}
\begin{proof}
  Recall from~\cite[Proposition 3.2, Corollary 3.3]{Segal68a} that $\pi_0^G(KU_G)\cong R(G)$ is a Noetherian ring and that for all subgroup $H\subseteq G$, the $R(G)$-module $\pi_0^H(KU_G)\cong R(H)$ is finitely generated. Combining this with \Cref{lem:nilpotence-kug} and \Cref{lem:kug_cyclic_balmer} we see that the assumptions of \Cref{cor:quillen_stratification}(a) are satisfied so the result applies to give the claim.
\end{proof}

Putting together \Cref{thm:kug_strat,lem:generalspc_ku_g} we deduce the desired stratification result for equivariant $K$-theory.

\begin{Thm}\label{thm:kug-stratification}
For any finite group $G$, the category $\Modd{G}(KU_G)$ is cohomologically stratified.
\end{Thm}

\begin{Exa}
We give an explicit description for the case $G = C_p$ following \cite[Remark 6.4]{DellAmbrogioMeyer2021spectrum}. In this case $R(C_p) \cong \bbZ[x]/(x^p-1)$. We have $(x^p-1) = (x-1)\Phi_p(x)$ where $\Phi_p(x) \coloneqq 1 + x + \cdots x^{p-1}$. The two quotient maps
\[
R(C_p) \to R(e) \cong \bbZ  \quad \text{ and } \quad R(C_p) \to \bbZ[x]/\Phi_p(x)
\]
give rise to jointly surjective embeddings
\[
\Spec(\bbZ) \xrightarrow{\psi_e} \Spec(R(C_p))  \xleftarrow{\psi_{C_p}} \Spec(\bbZ[x]/\Phi_p(x)).
\]
There is only one prime ideal in the intersection of their images, namely
\[
\psi_e((p)) = (p,x-1) = (p,\Phi_p(x)) = \psi_{C_p}((p)).
\]
Correspondingly, we have a decomposition
\begin{equation}\label{eq:rcp_decomposition}
\Spec(R(C_p)) \cong \Spec(\bbZ) \sqcup \Spec(\bbZ[x,p^{-1}]/\Phi_p(x)). 
\end{equation}
Write $\bbZ[\theta,p^{-1}] \coloneqq \bbZ[x,p^{-1}]/\Phi_p(x)$, where $\theta$ is the image of $x$ under the quotient map (so $\theta$ is a $p$-th root of unity). The proof of \Cref{lem:kug_cyclic_balmer} also shows that the following diagram commutes
\[\begin{tikzcd}
	{\Spec(\Perf(KU))} & {\Spc(\Perf_{C_p}(KU_{C_p}))} & {\Spc(\Perf(\Phi^{C_p}KU_{C_p}))} \\
	{\Spec(\bbZ)} & {\Spec(R(C_p))} & {\Spec(\bbZ[\theta,p^{-1}]).}
	\arrow["\varphi^e",from=1-1, to=1-2]
	\arrow["\varphi^{C_p}"',from=1-3, to=1-2]
	\arrow["{\rho_0^e}"', "\cong", from=1-1, to=2-1]
	\arrow["{\rho_0^{C_p}}", "\cong"',from=1-3, to=2-3]
	\arrow["{\psi_{e}}"', from=2-1, to=2-2]
	\arrow["{\psi_{C_p}}", from=2-3, to=2-2]
	\arrow["{\rho_0}"', "\cong", from=1-2, to=2-2]
\end{tikzcd}\]
When $p=2$, we depict $\Spec(R(C_2)) \cong \Spc(\Perf_{C_2}(KU_{C_2}))$ in \Cref{fig:spcku}; it consists of a copy of $\Spec(\bbZ[x]/(x-1))$ and a copy of $\Spec(\bbZ[x]/(x+1))$ (both of which are homeomorphic to $\Spec(\bbZ)$) glued together at the prime 2.
 \begin{figure}[!ht]
    \centering
\[\begin{tikzcd}[ampersand replacement=\&]
	{{\quad \quad \textcolor{blue}\bullet_{\bbF_{\hspace{-0.1em}3}} \textcolor{blue}\bullet_{\bbF_{\hspace{-0.1em}5}} \textcolor{blue}\bullet_{\bbF_{\hspace{-0.1em}7}}}\cdots} \& {\scalebox{0.65}{\color{blue}\LEFTCIRCLE}\kern-.5em{\color{red}\scalebox{0.65}\RIGHTCIRCLE_{\color{black}\bbF_{\hspace{-0.1em}2}}}} \& {{ \quad \quad  \textcolor{red}\bullet_{\bbF_{\hspace{-0.1em}3}} \textcolor{red}\bullet_{\bbF_{\hspace{-0.1em}5}} \textcolor{red}\bullet_{\bbF_{\hspace{-0.1em}7}}\cdots}} \\
	{\textcolor{blue} \bullet_{\bbQ}} \&\& {\textcolor{red} \bullet_{\bbQ}}
	\arrow[no head, from=2-1, to=1-2]
	\arrow[no head, from=1-2, to=2-3]
    \arrow[no head, from=2-1, to=1-1,end anchor={[xshift=2ex]}]
	\arrow[no head, from=2-1, to=1-1,end anchor={[xshift=-2ex]}]
	\arrow[no head, from=2-1, to=1-1]
	\arrow[no head, from=2-3, to=1-3,,end anchor={[xshift=2ex]}]
	\arrow[no head, from=2-3, to=1-3,,end anchor={[xshift=-2ex]}]
	\arrow[no head, from=2-3, to=1-3]
\end{tikzcd}\]\caption{An illustration of the spectrum $\Spc(\Perf_{C_2}(KU_{C_2})) \cong \Spec(R(C_2))$. Here closure goes upwards, and the primes are labeled by their residue fields. The blue points denote primes in $\Spec(\bbZ[x]/(x-1)) \cong \Spec(\bbZ)$, while the red ones denote primes in $\Spec(\bbZ[x]/(x+1)) \cong \Spec(\bbZ)$. The point denoted $\bbF_{\hspace{-0.1em}2}$ shown in blue and red denotes the gluing of the two copies of $\Spec(\bbZ)$ over the prime ideal (2).}\label{fig:spcku}
\end{figure}
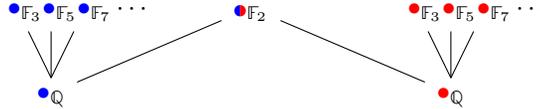
\end{Exa}

We refer the reader to
\cite[Section 11.4]{Serre78} for a discussion of $\Spec(R(G))$ for general finite groups $G$. For example, there it is shown that 
$\Spec(R(G))$ is connected.

\begin{Exa}\label{exa:kr-stratification}
 Let $K\mathbb{R}\in\CAlg(\Sp_{C_2})$ denote Atiyah's \emph{$K$-theory with reality}, introduced in \cite{Atiyah1966reality}. This is a $C_2$-ring spectrum with underlying spectrum given by $KU$, whose homotopy groups are regular Noetherian concentrated in even degrees. Its $C_2$-geometric fixed points are trivial by~\cite[Proposition 6.1]{HeardStojanoska}. It follows from \Cref{thm:eqstratdescent} that $\Modd{C_2}(K\mathbb{R})$ is stratified, and by \Cref{thm:quillen_decomposition} we have an identification 
\[
\Spc(\Perff{C_2}(K\mathbb{R})) \cong \Spec^h(KU_*)/C_2.
\]
Note that $C_2$ acts trivially on the affine scheme $\Spec^h(KU_*)\cong \Spec(\Z)$. It then follows that $\Spc(\Perff{C_2}(K\mathbb{R})) \cong \Spec(\Z)$, and that $\Spc(\Perf_{C_2}(K\mathbb{R}))$ is actually cohomologically stratified. In particular, the telescope conjecture holds in this case and the Bousfield lattice is isomorphic to the lattice of subsets of $\Spec(\Z)$, see \cref{rem:stratification-implies-telconj}.
\end{Exa}

\bibliographystyle{alpha}\bibliography{bibliography}

\newcommand{\etalchar}[1]{$^{#1}$}
\begin{thebibliography}{GHMR05}

\bibitem[AMGR24]{AMGR2019}
David Ayala, Aaron Mazel-Gee, and Nick Rozenblyum.
\newblock Stratified noncommutative geometry.
\newblock {\em Mem. Amer. Math. Soc.}, 297(1485):iii+260, 2024.

\bibitem[AMR21]{AMGR2021}
David {Ayala}, Aaron {Mazel-Gee}, and Nick {Rozenblyum}.
\newblock {Derived Mackey functors and $C_{p^n}$-equivariant cohomology}.
\newblock {\em arXiv e-prints}, page arXiv:2105.02456, May 2021.

\bibitem[Ati66]{Atiyah1966reality}
Michael~F. Atiyah.
\newblock {$K$}-theory and reality.
\newblock {\em Quart. J. Math. Oxford Ser. (2)}, 17:367--386, 1966.

\bibitem[Bal05]{Balmer05a}
Paul Balmer.
\newblock The spectrum of prime ideals in tensor triangulated categories.
\newblock {\em J. Reine Angew. Math.}, 588:149--168, 2005.

\bibitem[Bal07]{Balmer07}
Paul Balmer.
\newblock Supports and filtrations in algebraic geometry and modular
  representation theory.
\newblock {\em Amer. J. Math.}, 129(5):1227--1250, 2007.

\bibitem[Bal10]{Balmer10b}
Paul Balmer.
\newblock Spectra, spectra, spectra -- tensor triangular spectra versus
  {Z}ariski spectra of endomorphism rings.
\newblock {\em Algebr. Geom. Topol.}, 10(3):1521--1563, 2010.

\bibitem[Bal11]{Balmer11}
Paul Balmer.
\newblock Separability and triangulated categories.
\newblock {\em Adv. Math.}, 226(5):4352--4372, 2011.

\bibitem[Bal14]{Balmer14}
Paul Balmer.
\newblock Splitting tower and degree of tt-rings.
\newblock {\em Algebra Number Theory}, 8(3):767--779, 2014.

\bibitem[Bal16a]{Balmer16a}
Paul Balmer.
\newblock The derived category of an \'{e}tale extension and the separable
  {N}eeman-{T}homason theorem.
\newblock {\em J. Inst. Math. Jussieu}, 15(3):613--623, 2016.

\bibitem[Bal16b]{Balmer16b}
Paul Balmer.
\newblock Separable extensions in tensor-triangular geometry and generalized
  {Q}uillen stratification.
\newblock {\em Ann. Sci. \'{E}c. Norm. Sup\'{e}r. (4)}, 49(4):907--925, 2016.

\bibitem[Bal18]{Balmer18}
Paul Balmer.
\newblock On the surjectivity of the map of spectra associated to a
  tensor-triangulated functor.
\newblock {\em Bull. Lond. Math. Soc.}, 50(3):487--495, 2018.

\bibitem[Bar17]{Barwick17}
Clark Barwick.
\newblock Spectral {M}ackey functors and equivariant algebraic {$K$}-theory
  ({I}).
\newblock {\em Adv. Math.}, 304:646--727, 2017.

\bibitem[{Bar}21]{Barthel2021pre}
Tobias {Barthel}.
\newblock {Stratifying integral representations of finite groups}.
\newblock Preprint, 36~pages, available online at
  \href{https://arxiv.org/abs/2109.08135}{arXiv:2109.08135}, 2021.

\bibitem[{Bar}22]{Barthel2022pre}
Tobias {Barthel}.
\newblock {Stratifying integral representations via equivariant homotopy
  theory}.
\newblock {\em arXiv e-prints}, page arXiv:2203.14946, March 2022.

\bibitem[BCR97]{BensonCarlsonRickard97}
Dave Benson, Jon~F. Carlson, and Jeremy Rickard.
\newblock Thick subcategories of the stable module category.
\newblock {\em Fund. Math.}, 153(1):59--80, 1997.

\bibitem[BDS15]{BalmerDellAmbrogioSanders15}
Paul Balmer, Ivo Dell'Ambrogio, and Beren Sanders.
\newblock Restriction to finite-index subgroups as \'etale extensions in
  topology, {KK}-theory and geometry.
\newblock {\em Algebr. Geom. Topol.}, 15(5):3025--3047, 2015.

\bibitem[BDS16]{BalmerDellAmbrogioSanders16}
Paul Balmer, Ivo Dell'Ambrogio, and Beren Sanders.
\newblock Grothendieck--{N}eeman duality and the {W}irthm\"uller isomorphism.
\newblock {\em Compos. Math.}, 152(8):1740--1776, 2016.

\bibitem[BF11]{BalmerFavi11}
Paul Balmer and Giordano Favi.
\newblock Generalized tensor idempotents and the telescope conjecture.
\newblock {\em Proc. Lond. Math. Soc. (3)}, 102(6):1161--1185, 2011.

\bibitem[BG22a]{BalmerGallauer2022}
Paul Balmer and Martin Gallauer.
\newblock Permutation modules and cohomological singularity.
\newblock {\em Comment. Math. Helv.}, 97(3):413--430, 2022.

\bibitem[BG22b]{BalmerGallauer2022pre}
Paul {Balmer} and Martin {Gallauer}.
\newblock {The tt-geometry of permutation modules. Part I: Stratification}.
\newblock {\em arXiv e-prints}, page arXiv:2210.08311, October 2022.

\bibitem[BG24]{BalmerGallauer_announcement}
Paul Balmer and Martin Gallauer.
\newblock The spectrum of {A}rtin motives.
\newblock arXiv https://arxiv.org/abs/2401.02722, 2024.

\bibitem[BGH20]{bgh_balmer}
Tobias Barthel, John P.~C. Greenlees, and Markus Hausmann.
\newblock On the {B}almer spectrum for compact {L}ie groups.
\newblock {\em Compos. Math.}, 156(1):39--76, 2020.

\bibitem[BHN{\etalchar{+}}19]{BHNNNS19}
Tobias Barthel, Markus Hausmann, Niko Naumann, Thomas Nikolaus, Justin Noel,
  and Nathaniel Stapleton.
\newblock The {B}almer spectrum of the equivariant homotopy category of a
  finite abelian group.
\newblock {\em Invent. Math.}, 216(1):215--240, 2019.

\bibitem[BHS23]{bhs1}
Tobias Barthel, Drew Heard, and Beren Sanders.
\newblock Stratification in tensor triangular geometry with applications to
  spectral {M}ackey functors.
\newblock {\em Camb. J. Math.}, 11(4):829--915, 2023.

\bibitem[BIK11a]{BensonIyengarKrause2011Localising}
Dave Benson, Srikanth~B. Iyengar, and Henning Krause.
\newblock Localising subcategories for cochains on the classifying space of a
  finite group.
\newblock {\em C. R. Math. Acad. Sci. Paris}, 349(17-18):953--956, 2011.

\bibitem[BIK11b]{BensonIyengarKrause11a}
Dave Benson, Srikanth~B. Iyengar, and Henning Krause.
\newblock Stratifying modular representations of finite groups.
\newblock {\em Ann. of Math. (2)}, 174(3):1643--1684, 2011.

\bibitem[BIK11c]{BensonIyengarKrause11b}
Dave Benson, Srikanth~B. Iyengar, and Henning Krause.
\newblock Stratifying triangulated categories.
\newblock {\em J. Topol.}, 4(3):641--666, 2011.

\bibitem[BIK12]{BIK2012}
David~J. Benson, Srikanth Iyengar, and Henning Krause.
\newblock {\em Representations of finite groups: local cohomology and support},
  volume~43 of {\em Oberwolfach Seminars}.
\newblock Birkh\"{a}user/Springer Basel AG, Basel, 2012.

\bibitem[BIKP24]{BIKP2022pre}
David~John Benson, Srikanth~B. Iyengar, Henning Krause, and Julia Pevtsova.
\newblock Fibrewise stratification of group representations.
\newblock {\em Annals of Representation Theory}, 1(1):97--124, 2024.

\bibitem[BK08]{BensonKrause2008}
David~John Benson and Henning Krause.
\newblock Complexes of injective {$kG$}-modules.
\newblock {\em Algebra Number Theory}, 2(1):1--30, 2008.

\bibitem[Boj83]{Bojanowska1983spectrum}
Agnieszka Bojanowska.
\newblock The spectrum of equivariant {$K$}-theory.
\newblock {\em Math. Z.}, 183(1):1--19, 1983.

\bibitem[BS17]{BalmerSanders17}
Paul Balmer and Beren Sanders.
\newblock The spectrum of the equivariant stable homotopy category of a finite
  group.
\newblock {\em Invent. Math.}, 208(1):283--326, 2017.

\bibitem[BS20]{BehrensShah2020equivariant}
Mark Behrens and Jay Shah.
\newblock {$C_2$}-equivariant stable homotopy from real motivic stable
  homotopy.
\newblock {\em Ann. K-Theory}, 5(3):411--464, 2020.

\bibitem[CMNN24]{CMNN2020}
Dustin Clausen, Akhil Mathew, Niko Naumann, and Justin Noel.
\newblock Descent and vanishing in chromatic algebraic {$K$}-theory via group
  actions.
\newblock {\em Ann. Sci. \'{E}c. Norm. Sup\'{e}r. (4)}, 57(4):1135--1190, 2024.

\bibitem[DHS88]{DevinatzHopkinsSmith88}
Ethan~S. Devinatz, Michael~J. Hopkins, and Jeffrey~H. Smith.
\newblock Nilpotence and stable homotopy theory. {I}.
\newblock {\em Ann. of Math. (2)}, 128(2):207--241, 1988.

\bibitem[DM21]{DellAmbrogioMeyer2021spectrum}
Ivo Dell'Ambrogio and Ralf Meyer.
\newblock The spectrum of equivariant {K}asparov theory for cyclic groups of
  prime order.
\newblock {\em Ann. K-Theory}, 6(3):543--558, 2021.

\bibitem[Dri74]{Drinfeld1974Elliptic}
Vladimir~G. Drinfel'd.
\newblock Elliptic modules.
\newblock {\em Mat. Sb. (N.S.)}, 94(136):594--627, 656, 1974.

\bibitem[DS16]{DellAmbrogioStanley16}
Ivo Dell'Ambrogio and Donald Stanley.
\newblock Affine weakly regular tensor triangulated categories.
\newblock {\em Pacific J. Math.}, 285(1):93--109, 2016.

\bibitem[DS18]{DellAmbrogioSanders18}
Ivo Dell'Ambrogio and Beren Sanders.
\newblock A note on triangulated monads and categories of module spectra.
\newblock {\em C. R. Math. Acad. Sci. Paris}, 356(8):839--842, 2018.

\bibitem[DS22]{DellAmbrogioStanley16erratum}
Ivo Dell'Ambrogio and Donald Stanley.
\newblock Erratum to: ``{A}ffine weakly regular tensor triangulated
  categories'' [{P}acific {J}. {M}ath. {\bf 285} (2016), no. 1, 93--109].
\newblock 2022.
\newblock Available at
  \url{https://math.univ-lille1.fr/~dellambr/affreg_erratum.pdf}.

\bibitem[Erg22]{ergus2022hopf}
Aras Ergus.
\newblock {\em {H}opf algebras and {H}opf--{G}alois extensions in
  {$\infty$}-categories}.
\newblock PhD thesis, EPFL, 2022.
\newblock EPFL PhD thesis, available at
  \url{https://infoscience.epfl.ch/record/295824}.

\bibitem[GHMR05]{GHMR2005}
Paul Goerss, Hans-Werner Henn, Mark Mahowald, and Charles Rezk.
\newblock A resolution of the {$K(2)$}-local sphere at the prime 3.
\newblock {\em Ann. of Math. (2)}, 162(2):777--822, 2005.

\bibitem[GM23]{GepnerMeier}
David Gepner and Lennart Meier.
\newblock On equivariant topological modular forms.
\newblock {\em Compos. Math.}, 159(12):2638--2693, 2023.

\bibitem[GM24]{GuillouMay2024}
Bertrand~J. Guillou and J.~Peter May.
\newblock Models of {$G$}-spectra as presheaves of spectra.
\newblock {\em Algebr. Geom. Topol.}, 24(3):1225--1275, 2024.

\bibitem[Gre99]{Greenlees1999Equivariant}
John P.~C. Greenlees.
\newblock Equivariant forms of connective {$K$}-theory.
\newblock {\em Topology}, 38(5):1075--1092, 1999.

\bibitem[GS99]{greenleesstrickland_varieties}
John P.~C. Greenlees and Neil~P. Strickland.
\newblock Varieties and local cohomology for chromatic group cohomology rings.
\newblock {\em Topology}, 38(5):1093--1139, 1999.

\bibitem[HHR16]{HillHopkinsRavenel16}
Michael~A. Hill, Michael~J. Hopkins, and Douglas~C. Ravenel.
\newblock On the nonexistence of elements of {K}ervaire invariant one.
\newblock {\em Ann. of Math. (2)}, 184(1):1--262, 2016.

\bibitem[HKR00]{HopkinsKuhnRavenel2000Generalized}
Michael~J. Hopkins, Nicholas~J. Kuhn, and Douglas~C. Ravenel.
\newblock Generalized group characters and complex oriented cohomology
  theories.
\newblock {\em J. Amer. Math. Soc.}, 13(3):553--594, 2000.

\bibitem[Hop87]{Hopkins87}
Michael~J. Hopkins.
\newblock Global methods in homotopy theory.
\newblock In {\em Homotopy theory (Durham, 1985)}, volume 117 of {\em LMS Lect.
  Note}, pages 73--96. Cambridge Univ. Press, 1987.

\bibitem[HPS97]{HoveyPalmieriStrickland97}
Mark Hovey, John~H. Palmieri, and Neil~P. Strickland.
\newblock Axiomatic stable homotopy theory.
\newblock {\em Mem. Amer. Math. Soc.}, 128(610), 1997.

\bibitem[HS98]{HopkinsSmith98}
Michael~J. Hopkins and Jeffrey~H. Smith.
\newblock Nilpotence and stable homotopy theory. {II}.
\newblock {\em Ann. of Math. (2)}, 148(1):1--49, 1998.

\bibitem[HS14]{HeardStojanoska}
Drew Heard and Vesna Stojanoska.
\newblock {$K$}-theory, reality, and duality.
\newblock {\em J. K-Theory}, 14(3):526--555, 2014.

\bibitem[Lau23]{Lau2021Balmer}
Eike Lau.
\newblock The {B}almer spectrum of certain {D}eligne-{M}umford stacks.
\newblock {\em Compos. Math.}, 159(6):1314--1346, 2023.

\bibitem[Len22]{lenz2022globalalgebraicktheoryswan}
Tobias Lenz.
\newblock Global algebraic {K}-theory is {S}wan {K}-theory.
\newblock Preprint, 53 pages, available online at \texttt{2202.07272}, 2022.

\bibitem[LMS86]{LewisMaySteinbergerMcClure86}
L.~Gaunce Lewis, Jr., J.~Peter May, and Mark Steinberger.
\newblock {\em Equivariant stable homotopy theory}, volume 1213 of {\em Lecture
  Notes in Mathematics}.
\newblock Springer-Verlag, Berlin, 1986.
\newblock With contributions by J. E. McClure.

\bibitem[LNP22]{LNP2022}
Sil Linskens, Denis Nardin, and Luca Pol.
\newblock Global homotopy theory via partially lax limits.
\newblock Preprint, accepted for publication in Geometry \& Topology. 70~pages,
  available online at
  \href{https://arxiv.org/abs/2206.01556}{arXiv:2206.01556}, 2022.

\bibitem[Lur09]{HTTLurie}
Jacob Lurie.
\newblock {\em Higher topos theory}, volume 170 of {\em Annals of Mathematics
  Studies}.
\newblock Princeton University Press, Princeton, NJ, 2009.

\bibitem[Lur17]{HALurie}
Jacob Lurie.
\newblock Higher algebra.
\newblock 1553~pages, available from the author's website, 2017.

\bibitem[Lur18]{Lurie_elliptic}
Jacob Lurie.
\newblock Elliptic {C}ohomology {II}: {O}rientations.
\newblock Preprint, available online,
  \url{https://www.math.ias.edu/~lurie/papers/Elliptic-II.pdf}, 2018.

\bibitem[Lur19]{Lurie_elliptic3}
Jacob Lurie.
\newblock Elliptic {C}ohomology {III}: {T}empered {C}ohomology.
\newblock Preprint, available online,
  \url{https://www.math.ias.edu/~lurie/papers/Elliptic-III-Tempered.pdf}, 2019.

\bibitem[Mas99]{Maschke1899}
Heinrich Maschke.
\newblock Beweis des {S}atzes, dass diejenigen endlichen linearen
  {S}ubstitutionsgruppen, in welchen einige durchgehends verschwindende
  {C}oefficienten auftreten, intransitiv sind.
\newblock {\em Math. Ann.}, 52(2-3):363--368, 1899.

\bibitem[Mat89]{Matsumura1989}
Hideyuki Matsumura.
\newblock {\em Commutative ring theory}, volume~8 of {\em Cambridge Studies in
  Advanced Mathematics}.
\newblock Cambridge University Press, Cambridge, second edition, 1989.
\newblock Translated from the Japanese by M. Reid.

\bibitem[Mat18]{Mathew2018}
Akhil Mathew.
\newblock Examples of descent up to nilpotence.
\newblock In {\em Geometric and topological aspects of the representation
  theory of finite groups}, volume 242 of {\em Springer Proc. Math. Stat.},
  pages 269--311. Springer, Cham, 2018.

\bibitem[MNN17]{MathewNaumannNoel17}
Akhil Mathew, Niko Naumann, and Justin Noel.
\newblock Nilpotence and descent in equivariant stable homotopy theory.
\newblock {\em Adv. Math.}, 305:994--1084, 2017.

\bibitem[MNN19]{MathewNaumannNoel2019}
Akhil Mathew, Niko Naumann, and Justin Noel.
\newblock Derived induction and restriction theory.
\newblock {\em Geom. Topol.}, 23(2):541--636, 2019.

\bibitem[Mor85]{Morava1985}
Jack Morava.
\newblock Noetherian localisations of categories of cobordism comodules.
\newblock {\em Ann. of Math. (2)}, 121(1):1--39, 1985.

\bibitem[Mun00]{Munkres}
James~R. Munkres.
\newblock {\em Topology}.
\newblock Prentice Hall, Inc., Upper Saddle River, NJ, 2000.
\newblock Second edition of [MR0464128].

\bibitem[Nee92]{Neeman92a}
Amnon Neeman.
\newblock The chromatic tower for {$D(R)$}.
\newblock {\em Topology}, 31(3):519--532, 1992.

\bibitem[Nee96]{Neeman96}
Amnon Neeman.
\newblock The {G}rothendieck duality theorem via {B}ousfield's techniques and
  {B}rown representability.
\newblock {\em J. Amer. Math. Soc.}, 9(1):205--236, 1996.

\bibitem[PSW22]{PatchkoriaSandersWimmer20pp}
Irakli Patchkoria, Beren Sanders, and Christian Wimmer.
\newblock The spectrum of derived {M}ackey functors.
\newblock {\em Trans. Amer. Math. Soc.}, 375(6):4057--4105, 2022.

\bibitem[Qui71]{Quillen71}
Daniel Quillen.
\newblock The spectrum of an equivariant cohomology ring. {I}, {II}.
\newblock {\em Ann. of Math. (2)}, 94:549--572; ibid. (2) {\bf 94} (1971),
  573--602, 1971.

\bibitem[{Rob}12]{Robalo}
Marco {Robalo}.
\newblock {Noncommutative Motives I: A Universal Characterization of the
  Motivic Stable Homotopy Theory of Schemes}.
\newblock {\em arXiv e-prints}, page arXiv:1206.3645, June 2012.

\bibitem[San22]{Sanders21pp}
Beren Sanders.
\newblock {A characterization of finite \'etale morphisms in tensor triangular
  geometry}.
\newblock {\em {Épijournal de Géométrie Algébrique}}, {Volume 6}, 2022.

\bibitem[Sch18]{Schwede18_global}
Stefan Schwede.
\newblock {\em Global homotopy theory}, volume~34 of {\em New Mathematical
  Monographs}.
\newblock Cambridge University Press, Cambridge, 2018.

\bibitem[Seg68a]{Segal1968Equivariant}
Graeme Segal.
\newblock Equivariant {$K$}-theory.
\newblock {\em Inst. Hautes \'{E}tudes Sci. Publ. Math.}, 34:129--151, 1968.

\bibitem[Seg68b]{Segal68a}
Graeme Segal.
\newblock The representation ring of a compact {L}ie group.
\newblock {\em Inst. Hautes \'Etudes Sci. Publ. Math.}, 34:113--128, 1968.

\bibitem[Ser78]{Serre78}
Jean-Pierre Serre.
\newblock {\em Repr\'esentations lin\'eaires des groupes finis}.
\newblock Hermann, Paris, 1978.

\bibitem[SS03]{SchwedeShipley03}
Stefan Schwede and Brooke Shipley.
\newblock Stable model categories are categories of modules.
\newblock {\em Topology}, 42(1):103--153, 2003.

\bibitem[Sta13]{Stapleton2013}
Nathaniel Stapleton.
\newblock Transchromatic generalized character maps.
\newblock {\em Algebr. Geom. Topol.}, 13(1):171--203, 2013.

\bibitem[Ste13]{Stevenson13}
Greg Stevenson.
\newblock Support theory via actions of tensor triangulated categories.
\newblock {\em J. Reine Angew. Math.}, 681:219--254, 2013.

\bibitem[Str97]{Strickl1997Finite}
Neil~P. Strickland.
\newblock Finite subgroups of formal groups.
\newblock {\em J. Pure Appl. Algebra}, 121(2):161--208, 1997.

\bibitem[tD79]{tomDieck1979}
Tammo tom Dieck.
\newblock {\em Transformation groups and representation theory}, volume 766 of
  {\em Lecture Notes in Mathematics}.
\newblock Springer, Berlin, 1979.

\bibitem[Tho97]{Thomason97}
Robert~W. Thomason.
\newblock The classification of triangulated subcategories.
\newblock {\em Compositio Math.}, 105(1):1--27, 1997.

\bibitem[{Tre}15]{treumannmathew_reps}
David {Treumann}.
\newblock {Representations of finite groups on modules over K-theory (with an
  appendix by Akhil Mathew)}.
\newblock {\em arXiv e-prints}, page arXiv:1503.02477, March 2015.

\bibitem[Zen17]{zeng_comps}
Mingcong Zeng.
\newblock Equivariant {E}ilenberg--{M}ac{L}ane spectra in cyclic $p$-groups.
\newblock Preprint, 52~pages, available online at \texttt{1710.01769v2}, 2017.

\bibitem[Zou23]{zou}
Changhan Zou.
\newblock Support theories for non-{N}oetherian tensor triangulated categories.
\newblock Preprint, 34 pages, available online at \texttt{2312.08596}, 2023.

\end{thebibliography}
\vskip -0.3in 
\end{document}